\documentclass{scrartcl}

\usepackage[utf8]{inputenc}
\usepackage[T1]{fontenc}
\usepackage[lined,boxed,commentsnumbered]{algorithm2e}
\usepackage{amsfonts}
\usepackage{amsmath}
\usepackage{amsthm}
\usepackage{amssymb}
\usepackage{array}
\usepackage[english]{babel}
\usepackage{caption}
\usepackage{comment}
\usepackage{color}
\usepackage{enumerate}
\usepackage{etoolbox}
\usepackage[a4paper,left=25mm,right=25mm,top=25mm,bottom=30mm]{geometry}
\usepackage{graphicx}
\usepackage{lscape}
\usepackage{mathabx}
\usepackage{mathtools}
\usepackage{multicol}
\usepackage{multido}
\usepackage{multirow}
\usepackage{pgfplots}
\usepackage{pifont}
\usepackage{stmaryrd}
\usepackage{subcaption}
\usepackage{textcomp}
\usepackage{tikz}
\usetikzlibrary{arrows,backgrounds,calc,patterns,plotmarks,decorations.pathmorphing}
\usepackage{url}
\usepackage{wasysym}
\usepackage{xcolor}

\definecolor{verypalegray}{RGB}{248,248,248}
\definecolor{palegray}{RGB}{220,220,220}
\definecolor{gray}{RGB}{195,195,195}
\definecolor{darkgray}{RGB}{170,170,170}
\definecolor{blackgray}{RGB}{145,145,145}

\newcolumntype{L}[1]{>{\raggedright\let\newline\\\arraybackslash\hspace{0pt}}m{#1}}
\newcolumntype{C}[1]{>{\centering\let\newline\\\arraybackslash\hspace{0pt}}m{#1}}
\newcolumntype{R}[1]{>{\raggedleft\let\newline\\\arraybackslash\hspace{0pt}}m{#1}}

\newcommand{\calA}{{\mathcal{A}}} \newcommand{\calB}{{\mathcal{B}}}
\newcommand{\calC}{{\mathcal{C}}} \newcommand{\calD}{{\mathcal{D}}}
\newcommand{\calE}{{\mathcal{E}}}

 \newcommand{\calL}{{\mathcal{L}}}
\newcommand{\calM}{{\mathcal{M}}} \newcommand{\calN}{{\mathcal{N}}}

\newcommand{\calS}{{\mathcal{S}}}

 \newcommand{\calZ}{{\mathcal{Z}}}

\newcommand{\bA}{{\textbf{A}}}

\newcommand{\bG}{{\textbf{G}}} 
 
 \newcommand{\bL}{{\textbf{L}}}
 \newcommand{\bN}{{\textbf{N}}}
 
 \newcommand{\bR}{{\textbf{R}}}

 \newcommand{\bX}{{\textbf{X}}}
\newcommand{\bY}{{\textbf{Y}}} 

\newcommand{\ba}{{\textbf{a}}} \newcommand{\bb}{{\textbf{b}}}
\newcommand{\bc}{{\textbf{c}}} \newcommand{\bd}{{\textbf{d}}}

 \newcommand{\bv}{{\textbf{v}}}
\newcommand{\bw}{{\textbf{w}}} \newcommand{\bx}{{\textbf{x}}}
\newcommand{\by}{{\textbf{y}}}

\newcommand{\CC}{{\mathbb{C}}}

 \newcommand{\RR}{{\mathbb{R}}}

\newcommand{\ol}{\overline}
\newcommand{\oL}{{\ol{\calL}}}
\newcommand{\op}{{\ol{p}}}
\newcommand{\sss}{\scriptstyle}

\newcommand{\upbA}{\mathbf{A}^\uparrow}

\newcommand{\upA}{\calA^\uparrow}
\newcommand{\lowA}{\calA^\downarrow}

\newcommand{\rext}{\overrightarrow{\mathbf{e}}}
\newcommand{\RNF}{\mathbf{RNF}}
\newcommand{\dec}{\mathbf{decide}}
\newcommand{\comp}{\mathbf{compute}}

\newcommand{\punc}[1]{
\draw[fill=white,draw=black,thick] (#1,0) circle (1);}
\newcommand{\PUNC}[1]{
\draw[fill=black,draw=black,thick] (#1,0) circle (1);}
\newcommand{\puncture}[2]{
\draw[fill=white,draw=black,thick] (#1,0) circle (1.5);
\node at (#1,0) {\small$p_{#2}$};}
\newcommand{\PUNCTURE}[2]{
\draw[fill=black,draw=black,thick] (#1,0) circle (1.5);
\node at (#1,0) {\small\color{white}$p_{#2}$};}
\newcommand{\bvertex}[2]{\draw[fill=white,draw=black,thick] (#1,#2) circle (0.4);}
\newcommand{\dvertex}[1]{\draw[fill=black,draw=black,thick] (#1,0) circle (0.4);}

\newtheorem{theorem}{Theorem}[section]
\newtheorem{lemma}[theorem]{Lemma}
\newtheorem{corollary}[theorem]{Corollary}
\newtheorem{proposition}[theorem]{Proposition}
\newtheorem{definition}[theorem]{Definition}

\newtheorem{example}[theorem]{Example}

\newenvironment{thm}{\begin{theorem}~\\}{\end{theorem}}
\newenvironment{thm*}[1]{\begin{theorem}[#1]~\\}{\end{theorem}}
\newenvironment{lem}{\begin{lemma}~\\}{\end{lemma}}
\newenvironment{lem*}[1]{\begin{lemma}[#1]~\\}{\end{lemma}}
\newenvironment{cor}{\begin{corollary}~\\}{\end{corollary}}
\newenvironment{cor*}[1]{\begin{corollary}[#1]~\\}{\end{corollary}}
\newenvironment{pro}{\begin{proposition}~\\}{\end{proposition}}
\newenvironment{pro*}[1]{\begin{proposition}[#1]~\\}{\end{proposition}}
\newenvironment{dfn}[1]{\begin{definition}[#1]~\\}{\end{definition}}

\newenvironment{xam*}[1]{\begin{example}[#1]\rm~\\}{\end{example}}

\title{The Relaxation Normal Form of Braids is Regular}
\author{Vincent Jugé
\thanks{Mines ParisTech, 60 boulevard Saint-Michel, 75272 Paris Cedex 06, France \newline
Université Paris Diderot, Sorbonne Paris Cité, IRIF, UMR 8243 CNRS, F-75205 Paris, France \newline
ENS Paris-Saclay, LSV, UMR 8643 CNRS, France}}
\date{}

\begin{document}

\maketitle

\begin{abstract}
Braids can be represented geometrically as laminations of punctured disks.
The geometric complexity of a braid is the minimal complexity of a lamination that represents it,
and tight laminations are representatives of minimal complexity.
These laminations give rise to a normal form of braids, via a relaxation algorithm.
We study here this relaxation algorithm and the associated normal form.
We prove that this normal form is regular and prefix-closed.
We provide an effective construction of a deterministic automaton that recognizes this normal form.
\end{abstract}

\section{Introduction}

Braid groups can be approached from various points of view,
including algebraic and geometric ones.
In the algebraic viewpoint, the braid group is defined
by a finite presentation, i.e. a finite generating family $\bX$ and finite number of relations.
Two of the most widely used generating families are
the Artin generators and the Garside generators.
In this approach, braids are viewed as equivalence classes of finite words over the
finite alphabet $\bX$.

A normal form consists in choosing exactly one representative in each equivalence class.
Three important desirable properties are computability, regularity and geodicity.
The normal form is \emph{computable} if, from each word $\bw$ over the alphabet $\bX$,
one can compute the word $\bx$ that is equivalent to $\bw$ and that
belongs to the normal form.
In particular, the existence of a computable normal form
implies that the word problem is solvable.
The normal form is \emph{regular} if it forms a regular subset of the free monoid $\bX^\ast$.
The normal form is \emph{geodesic} if its elements are shortest representatives
of their equivalence classes.

Computability and regularity do \emph{not} depend on the family $\bX$.
Indeed, if $\bX$ and $\bY$ are two finite generating families,
embedding $\bX$ into the free monoid $\bY^\ast$ allows to consider normal forms over the alphabet $\bX$
as normal forms over the alphabet $\bY$:
if a normal form is computable (respectively, regular) on the alphabet $\bX$,
it will remain computable (respectively, regular) on the alphabet $\bY$.

Geodicity is related to shortest paths in the Cayley graph of the group associated to $\bX$.
Geodicity depends on the family $\bX$, i.e. a normal form may be geodesic for $\bX$ and not for $\bY$.
In fact, even the existence of regular geodesic normal forms depends on $\bX$
(see~\cite[Example~4.4.2]{Epstein:1992:WPG:573874}).

Let us come back to the specific case of the braid group.
The Artin definition of the braid group is
\[B_n = \left\langle \sigma_1, \dots, \sigma_{n-1} \left|\begin{array}{l}
\sigma_i \sigma_{i+1} \sigma_i = \sigma_{i+1} \sigma_i \sigma_{i+1} \text{ for all } i \\
\sigma_i \sigma_j = \sigma_j \sigma_i \text{ for all } i,j \text{ s.t. } |i-j| \geq 2 \end{array}\right\rangle\right..\]
The set $\{\sigma_1, \sigma_1^{-1}, \ldots, \sigma_{n-1}, \sigma_{n-1}^{-1}\}$
is that of Artin generators.
They are often considered as the ``most natural'' generators because of their role
in the representation of braids as an isotopy class of \emph{braid diagrams}, as illustrated in Figure~\ref{fig:strands}.

\begin{figure}[!ht]
\begin{center}
\begin{tikzpicture}[scale=0.20]
\draw[draw=black,thick] (-2,5) -- (-2,18);
\draw[draw=black,thick] (2,5) -- (2,18);
\draw[draw=black,thick] (9,5) -- (9,18);
\draw[draw=black,thick] (13,5) -- (13,9) -- (17,13) -- (17,18);
\draw[draw=black,thick] (17,5) -- (17,9) -- (15.25,10.75);
\draw[draw=black,thick] (14.75,11.25) -- (13,13) -- (13,18);
\draw[draw=black,thick] (21,5) -- (21,18);
\draw[draw=black,thick] (28,5) -- (28,18);
\draw[draw=black,thick] (32,5) -- (32,18);
\node[anchor=south] at (-2,18) {\scriptsize $1$};
\node[anchor=south] at (2,18) {\scriptsize $2$};
\node[anchor=south] at (9,18) {\scriptsize $i-1$};
\node[anchor=south] at (13,18) {\scriptsize $i$};
\node[anchor=south] at (17,18) {\scriptsize $i+1$};
\node[anchor=south] at (21,18) {\scriptsize $i+2$};
\node[anchor=south] at (28,18) {\scriptsize $n-1$};
\node[anchor=south] at (32,18) {\scriptsize $n$};
\node at (5.5,11) {$\cdots$};
\node at (24.5,11) {$\cdots$};
\end{tikzpicture}
\end{center}
\caption{Braid diagram of the generator $\sigma_i$ ($1 \leq i \leq n-1$)}
\label{fig:strands}
\end{figure}
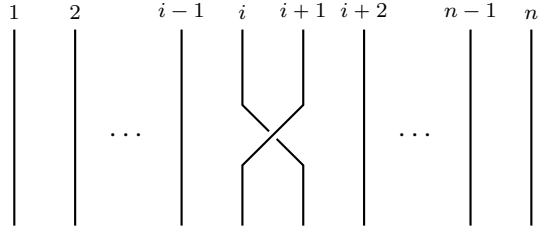

The monoid of positive braids $B_n^+$ is the monoid generated by the \emph{positive} Artin generators $\sigma_1,\ldots,\sigma_{n-1}$.
It has a lattice structure for the left-divisibility relation~\cite{dehornoy2008ordering}.
Hence, consider the braid $\Delta_n = \mathrm{LCM}(\sigma_1,\ldots,\sigma_{n-1})$, which we call the \emph{Garside element} of the monoid.
Note that the set $\bG^+ = \{\beta \in B_n^+ : \beta$ left-divides $\Delta_n\}$ generates the braid monoid.
The set $\bG = \bG^+ \cup \{\beta : \beta^{-1} \in \bG^+\}$ is that of \emph{Garside generators} and generates the braid group.

The \emph{symmetric Garside normal form} over the braid group is built as follows.
First, each braid word $\bw$ is rewritten under the form $\Delta_n^{-k} \cdot \bv$,
where $k \geq 0$ and $\bv \in B_n^+$,
for instance by using rewriting rules of the form
\[\ba \cdot \sigma_i^{-1} \cdot \bb \mapsto \Delta^{-2} \cdot \ba \cdot \left(\sigma_i^{-1} \Delta\right) \cdot \Delta \cdot \bb.\]
Second, positive left-divisors of $\bv$ are selected incrementally,
thereby factoring $\bv$ as a product $\bv = u_1 \ldots u_\ell$, where
$u_i = \mathrm{GCD}(\Delta_n,u_i u_{i+1} \ldots u_\ell)$ and $u_\ell$ is non-trivial.
Finally, the elements $\Delta_n^{-1}$ appearing in the power $\Delta_n^{-k}$ are canceled out
with the leftmost factors of $v$.
For instance, if $k \leq \ell$, we obtain a word of the form $w_1^{-1} \cdot \ldots \cdot w_k^{-1} \cdot u_{k+1} \cdot \ldots \cdot u_\ell$,
where each factor $w_i$ or $u_i$ belongs to the set $\bG^+$.
The word $w_1^{-1} \cdot \ldots \cdot w_k^{-1} \cdot u_{k+1} \cdot \ldots \cdot u_\ell$ is the symmetric Garside normal form of the braid word $\bw$.

The Garside normal form is very successful because it is simultaneously regular,
synchronously automatic (hence, very easy to compute incrementally) and
geodesic on the family $\bG$. In particular, the word problem in $B_n$ is decidable.

Another normal form is the \emph{ShortLex normal form} over Artin generators, which can be described as follows.
Consider some (arbitrary) linear ordering on the set of Artin generators.
Let $\beta$ be some braid, and let $W_\beta$ be the set of words that represent $\beta$.
The ShortLex normal form of $\beta$ is the word $\bw \in W_\beta$
such that, for all words $\bx \in W_\beta$, we have $|\bw| \leq |\bx|$ and
$|\bw| = |\bx| \Rightarrow \bw \leq_\mathrm{lex} \bx$.

Since the word problem in the braid group $B_n$ is decidable,
the ShortLex normal form over Artin generators is computable, in addition to being geodesic.
However, the regularity of ShortLex is an open problem.
In fact, except for $B_2$ and $B_3$, it is not known whether there exists a regular geodesic normal form
for the set of Artin generators.

Let us now turn our attention to the geometric viewpoint.
A braid is viewed as a class of \emph{laminations}, and acts on such laminations.
This is to be compared with the algebraic viewpoint, in which a braid is a class of words and acts on words.

\begin{figure}[!ht]
\begin{center}
\begin{tikzpicture}[scale=0.15]
\draw [draw=black,fill=verypalegray] (-47,-12.5) -- (-34,-12.5) -- (-34,12.5) -- (-47,12.5) -- cycle;
\draw [draw=black] (-40.5,12.5) -- (-40.5,14.5);
\draw [draw=black] (-30.5,14.5) -- (-30.5,20.8) -- (-50.5,20.8) -- (-50.5,14.5) -- cycle;
\node at (-40.5,19) {Braid diagram $\calD$};
\node at (-40.5,16) {representing $\sigma_2$};

\draw [draw=black,ultra thick] (-35.5,11.5) -- (-35.5,2.5) -- (-40.5,-2.5) -- (-40.5,-11.5);
\draw [draw=black,ultra thick] (-40.5,11.5) -- (-40.5,2.5) -- (-38.4,0.4);
\draw [draw=black,ultra thick] (-37.6,-0.4) -- (-35.5,-2.5) -- (-35.5,-11.5);
\draw [draw=black,ultra thick] (-45.5,11.5) -- (-45.5,-11.5);

\draw[->,>=stealth',draw=black,thick] (-32,0) -- (-20,0);

\draw [draw=black,fill=verypalegray] (-16,-6.5) -- (13,-6.5) -- (11,-16.5) -- (-18,-16.5) -- cycle;
\draw [draw=palegray,fill=palegray,domain=-180:180,samples=1000] plot ({10*sin( \x )}, {-11.5-3*cos( \x )});
\draw [draw=black,very thin,densely dotted] (10,-11.5) -- (-15.5,-11.5);
\draw [draw=black,fill=white] (-16,-6.5) -- (-16.4,-8.5) -- (-12.4,-8.5) -- (-12,-6.5) -- cycle;
\node at (-14.2,-7.5) {\tiny $3$};

\draw [draw=black,fill=white,thick,domain=-180:180,samples=1000] plot ({5+sin( \x )}, {-11.5-0.3*cos( \x )} );
\draw [draw=black,fill=white,thick,domain=-180:180,samples=1000] plot ({sin( \x )}, {-11.5-0.3*cos( \x )} );
\draw [draw=black,fill=white,thick,domain=-180:180,samples=1000] plot ({-5+sin( \x )}, {-11.5-0.3*cos( \x )} );
\draw [draw=black,fill=black,thick,domain=-180:180,samples=1000] plot ({-10+sin( \x )}, {-11.5-0.3*cos( \x )} );

\draw [draw=black,ultra thick] (5,11.5) -- (5,2.5) -- (0,-2.5) -- (0,-11.5);
\draw [draw=black,ultra thick] (0,11.5) -- (0,2.5) -- (2.2,0.3);
\draw [draw=black,ultra thick] (2.8,-0.3) -- (5,-2.5) -- (5,-11.5);
\draw [draw=black,ultra thick] (-5,11.5) -- (-5,-11.5);
\draw [draw=black,thick] (-10,11.5) -- (-10,-11.5);

\draw [draw=black,fill=verypalegray] (-16,5) -- (13,5) -- (11,-5) -- (-18,-5) -- cycle;
\draw [draw=palegray,fill=palegray,domain=-180:180,samples=1000] plot ({10*sin( \x )}, {-3*cos( \x )});
\draw [draw=black,very thin,densely dotted] (10,0) -- (-15.5,0);
\draw [draw=black,fill=white] (-16,5) -- (-16.4,3) -- (-12.4,3) -- (-12,5) -- cycle;
\node at (-14.2,4) {\tiny $2$};

\draw [draw=black,fill=white,thick,domain=-180:180,samples=1000] plot ({1.25+sin( \x )}, {1.25-0.3*cos( \x )} );
\draw [draw=black,fill=white,thick,domain=-180:180,samples=1000] plot ({1.25+sin( \x )}, {-1.25-0.3*cos( \x )} );
\draw [draw=black,fill=white,thick,domain=-180:180,samples=1000] plot ({-5+sin( \x )}, {-0.3*cos( \x )} );
\draw [draw=black,fill=black,thick,domain=-180:180,samples=1000] plot ({-10+sin( \x )}, {-0.3*cos( \x )} );

\draw [draw=black,thick,densely dotted] (-5,0) -- (-5,-5);
\draw [draw=black,thick,densely dotted] (1.25,1.25) -- (5,-2.5) -- (5,-5);
\draw [draw=black,thick,densely dotted] (1.25,-1.25) -- (0,-2.5) -- (0,-5);
\draw [draw=black,thick,densely dotted] (-10,0) -- (-10,-5);

\draw [draw=black,ultra thick] (5,5) -- (5,2.5) -- (1.6,-1.4);
\draw [draw=black,ultra thick] (0,5) -- (0,2.5) -- (1.4,1.1);
\draw [draw=black,ultra thick] (-5,5) -- (-5,0);
\draw [draw=black,thick] (-10,5) -- (-10,0);

\draw [draw=black,fill=white,thick,domain=-90:90,samples=1000] plot ({1.25+sin( \x )}, {1.25-0.3*cos( \x )} );
\draw [draw=black,fill=white,thick,domain=-90:90,samples=1000] plot ({1.25+sin( \x )}, {-1.25-0.3*cos( \x )} );
\draw [draw=black,fill=white,thick,domain=-90:90,samples=1000] plot ({-5+sin( \x )}, {-0.3*cos( \x )} );

\draw [draw=black,fill=verypalegray] (-16,16.5) -- (13,16.5) -- (11,6.5) -- (-18,6.5) -- cycle;
\draw [draw=palegray,fill=palegray,domain=-180:180,samples=1000] plot ({10*sin( \x )}, {11.5-3*cos( \x )});
\draw [draw=black,very thin,densely dotted] (10,11.5) -- (-15.5,11.5);
\draw [draw=black,fill=white] (-16,16.5) -- (-16.4,14.5) -- (-12.4,14.5) -- (-12,16.5) -- cycle;
\node at (-14.2,15.5) {\tiny $1$};

\draw [draw=black,thick,densely dotted] (5,11.5) -- (5,6.5);
\draw [draw=black,thick,densely dotted] (0,11.5) -- (0,6.5);
\draw [draw=black,thick,densely dotted] (-5,11.5) -- (-5,6.5);
\draw [draw=black,thick,densely dotted] (-10,11.5) -- (-10,6.5);

\draw [draw=black,fill=white,thick,domain=-180:180,samples=1000] plot ({5+sin( \x )}, {11.5-0.3*cos( \x )} );
\draw [draw=black,fill=white,thick,domain=-180:180,samples=1000] plot ({sin( \x )}, {11.5-0.3*cos( \x )} );
\draw [draw=black,fill=white,thick,domain=-180:180,samples=1000] plot ({-5+sin( \x )}, {11.5-0.3*cos( \x )} );
\draw [draw=black,fill=black,thick,domain=-180:180,samples=1000] plot ({-10+sin( \x )}, {11.5-0.3*cos( \x )} );

\draw [draw=black,thick,domain=-180:180,samples=1000] plot ({-10+2.5*sin( \x )}, {11.5-0.75*cos( \x )} );
\draw [draw=black,thick,domain=-180:180,samples=1000] plot ({-8+5.5*sin( \x )}, {11.5-1.65*cos( \x )} );
\draw [draw=black,thick,domain=-180:180,samples=1000] plot ({-6+8.5*sin( \x )}, {11.5-2.55*cos( \x )} );
\draw [draw=black,thick,domain=-180:180,samples=1000] plot ({-4+11.5*sin( \x )}, {11.5-3.45*cos( \x )} );

\draw [draw=black,thick,domain=-180:180,samples=1000] plot ({-10+2.5*sin( \x )}, {-0.75*cos( \x )} );
\draw [draw=black,thick,domain=-180:180,samples=1000] plot ({-8+5.5*sin( \x )}, {-1.65*cos( \x )} );
\draw [draw=black,thick,domain=90:270,samples=1000] plot ({-8+6.5*sin( \x )}, {2.55*cos( \x )} );
\draw [draw=black,thick,domain=-90:90,samples=1000] plot ({-3.75+10.75*sin( \x )}, {2.925*cos( \x )} );
\draw [draw=black,thick,domain=-180:180,samples=1000] plot ({-3.75+11.75*sin( \x )}, {3.525*cos( \x )} );
\draw [draw=black,thick] (-1.6,0) -- (7.1,0);

\draw [draw=black,thick,domain=-180:180,samples=1000] plot ({-10+2.5*sin( \x )}, {-11.5-0.75*cos( \x )} );
\draw [draw=black,thick,domain=-180:180,samples=1000] plot ({-8.25+5.25*sin( \x )}, {-11.5-1.575*cos( \x )} );
\draw [draw=black,thick,domain=90:270,samples=1000] plot ({-8.25+6.25*sin( \x )}, {-11.5+2.175*cos( \x )} );
\draw [draw=black,thick,domain=-90:90,samples=1000] plot ({0.25+2.25*sin( \x )}, {-11.5+1.175*cos( \x )} );
\draw [draw=black,thick,domain=90:270,samples=1000] plot ({4.75+2.25*sin( \x )}, {-11.5+1.175*cos( \x )} );
\draw [draw=black,thick,domain=-90:90,samples=1000] plot ({-3.75+10.75*sin( \x )}, {-11.5+2.925*cos( \x )} );
\draw [draw=black,thick,domain=-180:180,samples=1000] plot ({-3.75+11.75*sin( \x )}, {-11.5-3.525*cos( \x )} );

\draw[draw=black,thick] (15,0) -- (27,0);
\draw[draw=black,thick] (27,0) arc (90:0:3);
\draw[draw=black,thick] (30,-3) -- (30,-39);
\draw[draw=black,thick] (30,-39) arc (0:-90:3);
\draw[->,>=stealth',draw=black,thick] (27,-42) -- (23.5,-42);

\draw [draw=black,fill=verypalegray] (-55,-54) -- (-55,-30) -- (-31,-30) -- (-31,-54) -- cycle;
\draw [draw=black,fill=white] (-55,-30) -- (-55,-32) -- (-51,-32) -- (-51,-30) -- cycle;
\node at (-53,-31) {\tiny $1$};
\draw [draw=black] (-43,-30) -- (-43,-28);
\draw [draw=black] (-53,-28) -- (-53,-22) -- (-33,-22) -- (-33,-28) -- cycle;
\node at (-43,-23.5) {\small Lamination $\bL$};
\node at (-43,-26.5) {\small representing $\varepsilon$};

\draw[fill=palegray,draw=palegray] (-41,-42) circle (8);
\draw[draw=black] (-54,-42) -- (-33,-42);
\draw[fill=black,draw=black,thick] (-49,-42) circle (1);

\draw[draw=black,ultra thick] (-47,-42) arc (0:360:2) -- cycle;
\draw[draw=black,ultra thick] (-43,-42) arc (0:360:4.5) -- cycle;
\draw[draw=black,ultra thick] (-39,-42) arc (0:360:7) -- cycle;
\draw[draw=black,ultra thick] (-35,-42) arc (0:360:9.5) -- cycle;

\draw[fill=white,draw=black,thick] (-45,-42) circle (1);
\draw[fill=white,draw=black,thick] (-41,-42) circle (1);
\draw[fill=white,draw=black,thick] (-37,-42) circle (1);

\draw [draw=black,fill=verypalegray] (-30,-54) -- (-30,-30) -- (-5,-30) -- (-5,-54) -- cycle;
\draw [draw=black,fill=white] (-30,-30) -- (-30,-32) -- (-26,-32) -- (-26,-30) -- cycle;
\node at (-28,-31) {\tiny $2$};

\draw[fill=palegray,draw=palegray] (-15.5,-42) circle (8.5);
\draw[draw=black] (-28,-42) -- (-7,-42);
\draw[fill=black,draw=black,thick] (-24,-42) circle (1);

\draw[draw=black,ultra thick] (-22,-42) arc (0:360:2) -- cycle;
\draw[draw=black,ultra thick] (-18,-42) arc (0:360:4.5) -- cycle;
\draw[draw=black,ultra thick] (-28,-42) arc (180:0:5.5) -- (-10,-42) arc (360:180:9) -- cycle;
\draw[draw=black,ultra thick] (-9,-42) arc (360:0:10) -- cycle;

\draw[fill=white,draw=black,thick] (-20,-42) circle (1);
\draw[fill=white,draw=black,thick] (-15,-40) circle (1);
\draw[fill=white,draw=black,thick] (-15,-44) circle (1);

\draw [draw=black,fill=verypalegray] (-4,-54) -- (-4,-30) -- (22,-30) -- (22,-54) -- cycle;
\draw [draw=black,fill=white] (-4,-30) -- (-4,-32) -- (0,-32) -- (0,-30) -- cycle;
\node at (-2,-31) {\tiny $3$};
\draw [draw=black] (9,-30) -- (9,-28);
\draw [draw=black] (-1,-28) -- (-1,-22) -- (19,-22) -- (19,-28) -- cycle;
\node at (9,-23.5) {\small Lamination $\bL \cdot \sigma_2$};
\node at (9,-26.5) {\small representing $\sigma_2$};

\draw[fill=palegray,draw=palegray] (11,-42) circle (9);
\draw[draw=black] (-3,-42) -- (20,-42);
\draw[fill=black,draw=black,thick] (2,-42) circle (1);

\draw[draw=black,ultra thick] (4,-42) arc (0:360:2) -- cycle;
\draw[draw=black,ultra thick] (8,-42) arc (0:360:4.5) -- cycle;
\draw[draw=black,ultra thick] (9,-42) arc (360:180:5.5) arc (180:0:9.5) arc (360:180:2) arc (0:180:2) -- cycle;
\draw[draw=black,ultra thick] (18,-42) arc (360:0:10.5) -- cycle;

\draw[fill=white,draw=black,thick] (6,-42) circle (1);
\draw[fill=white,draw=black,thick] (11,-42) circle (1);
\draw[fill=white,draw=black,thick] (15,-42) circle (1);
\end{tikzpicture}
\end{center}
\caption{Braid acting on a lamination}
\label{fig:braid-action-intro}
\end{figure}
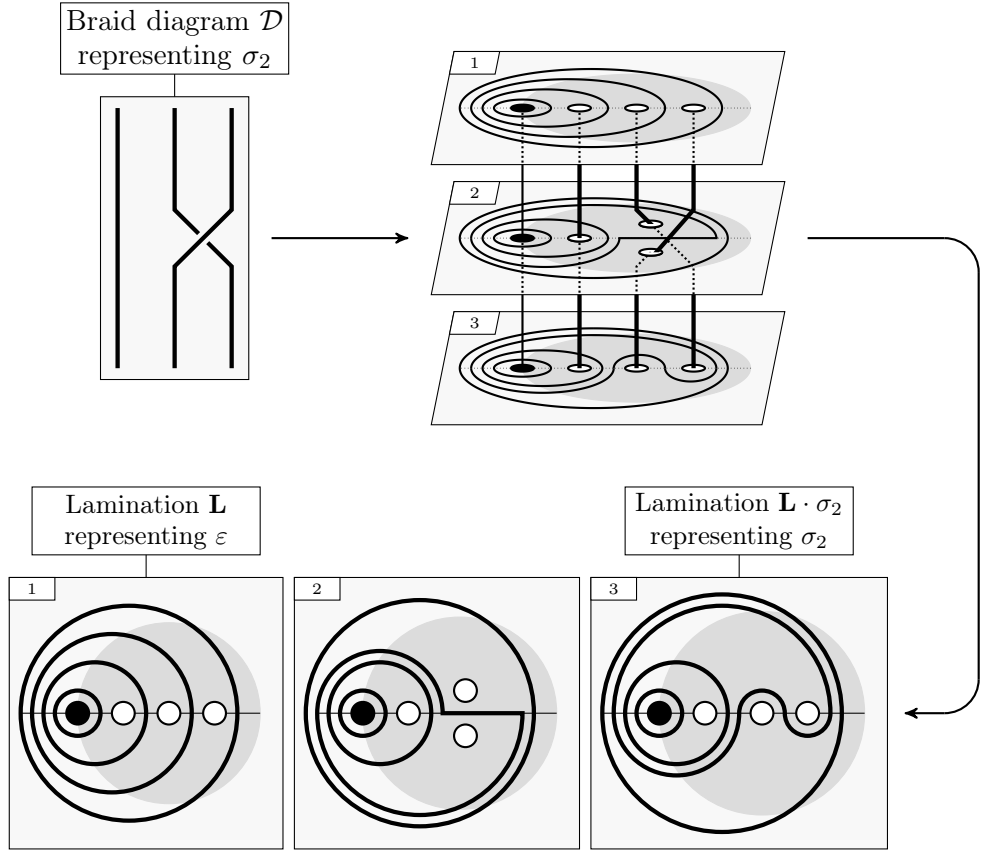

The lamination $\bL$, represented in the left bottom part of Figure~\ref{fig:braid-action-intro},
consists in one \emph{fixed puncture} (in black) and
$n$ \emph{mobile punctures} (in white) situated on an horizontal axis,
as well as $n+1$ non-intersecting closed curves $\bL_0, \bL_1, \ldots, \bL_n$,
such that the finite part of the plane delimited by each curve $\bL_i$ contains both the fixed puncture and $i$ mobile punctures.
This lamination $\bL$ is called \emph{trivial}, and represents the trivial braid $\varepsilon$.

A braid $\beta \in B_n$, represented by some braid diagram $\calD$, acts on the lamination $\bL$ as follows.
We place the $n$ mobile punctures on the top of the $n$ strands of $\calD$,
then let these punctures slide along the strands of $\calD$, until we reach the bottom of $\calD$.
At the same time, we force the $n+1$ curves to follow the motion prescribed by the punctures.
Doing so, we obtain the lamination $\bL \cdot \beta$, which will henceforth represent the braid $\beta$ itself.
Figure~\ref{fig:braid-action-intro} illustrates the action of the braid $\sigma_2$ on $\bL$.

In spite of its differences with the algebraic approach,
the lamination-based point of view on braids also provides a way of
obtaining normal forms.
First, to each lamination we associate an equivalent \emph{tight} lamination.
Second, thanks to a \emph{relaxation algorithm}, we perform successive simplifications of this tight lamination.
Each simplification consists in making some braid act on our lamination
in order to obtain a simplified tight lamination.
We perform such simplifications until reaching the trivial lamination.
Finally, having found a product of braids that maps the original lamination to the trivial one,
inverting that product yields a normal form for the original lamination.

Note the similarity between the Garside procedure and the above one:
we begin with some cleaning step, then we perform iteratively simplifications,
and finally we gather our results in order to obtain the desired normal form.

Relaxation algorithms play a central role in the study of the \emph{geometric complexity} of braids
by Dynnikov and Wiest~\cite{dynnikov:hal-00001267}.
Observing that well-chosen relaxation algorithms provide normal forms of short length, Dynnikov and Wiest prove
that two measures of complexity of braids are comparable,
although the first measure stems from the geometric world of laminations,
while the second measure is related to factorisations of braids with repeated factors.

Another remarkable achievement obtained with relaxation algorithms is the Bressaud normal form~\cite{Bressaud-normal-form}.
This normal form relies on an alternative geometric representation of braids,
and was one seminal example of regular $\sigma$-consistent normal form.
However, although the Bressaud normal form is known to be asynchronously automatic,
it is yet unknown whether it maps braids to factorisations of minimal length (up to a multiplicative constant).

We focus below on a specific relaxation algorithm,
which is the right-relaxation algorithm considered by Caruso in her PhD Thesis~\cite{carusoPhD}.
The associated normal form is called the \emph{relaxation normal form}.
Each simplification is obtained through the action of a braid chosen from a finite family of \emph{sliding braids},
so that this braid be ``maximal'' in some sense.

The right-relaxation algorithm is a variant of the relaxation algorithm used by Dynnikov and Wiest.
The difference between both algorithms resides in the family of braids that may act on laminations
and in the criteria used for choosing which element of the family should be chosen.
In~\cite{carusoPhD}, Caruso identifies geometrical features of a lamination $\bL \cdot \beta$
that may be used to estimate the complexity of the relaxation normal form of the braid $\beta$.

Our main contribution consists in proving that the relaxation normal form is regular
and in constructing an automaton that recognizes the relaxation normal form.
To the best of our knowledge, our result provides the first known example of regular normal form
obtained by applying a relaxation algorithm on braid laminations.

The relaxation normal form does not preserve the braid positivity
(i.e. the normal form of a positive braid may contain negative factors) and is not geodesic
(for Artin, Garside or sliding generators).
For example, the relaxation normal form of the braid $\sigma_2^2 \sigma_1$ is
the word $\sigma_2 \cdot \sigma_2 \cdot \sigma_1 \sigma_2 \cdot \sigma_2^{-1}$.
Yet, it enjoys several additional properties.

In what follows, we prove that the relaxation normal form is regular,
and provide a deterministic automaton that recognizes its language.
We sketch a proof in section~\ref{section:braid-laminations},
indicating the main ideas and objects
that lead to the regularity of the relaxation normal form,
and we provide rigorous proofs in section~\ref{section:proofs}.
We study several side problems in subsequent sections.
We investigate the automaticity of the relaxation normal form, and we prove that it is
synchronously biautomatic for $n \leq 3$ and not asynchronously right automatic for $n \geq 4$.
Then, we show how to read the $\sigma$-positivity of a braid on its relaxation normal form.
We also prove that the above-mentioned automaton is approximately minimal.
Finally, we review some variants of the relaxation normal form,
and explain for each of them to which extent our results also apply to these variants.

\section{The Relaxation Normal Form is Regular: Key Ideas}
\label{section:braid-laminations}

In section~\ref{section:braid-laminations} we first present standard definitions
and theorems about braids and \emph{laminations} of the punctured disk.
These results mainly come from algebraic topology and they can be found in standard literature,
e.g. in~\cite{birman-braids-links-mcg,dehornoy2008ordering,dynnikov:hal-00001267,farb2011primer,Fenn_orderingthe}.

Then, we present relaxation algorithms operating on braid laminations in general,
including the right-relaxation algorithm that we study later on.
Finally, we describe the main ideas that will allow us to later prove that the relaxation normal form is regular.
This last part is only meant to provide the reader with a general idea of why the normal form is regular,
anticipating the more rigorous proofs that will be the focal point of section~\ref{section:proofs}.

\subsection{Braids and Laminations}

The group of braids with $n$ strands is usually known
by its algebraic description, due to Artin~\cite{Artin-1926}.

\begin{dfn}{Braid group}
The group of braids with $n$ strands is the group
\[B_n = \left\langle \sigma_1, \dots, \sigma_{n-1} \left|\begin{array}{l}
\sigma_i \sigma_{i+1} \sigma_i = \sigma_{i+1} \sigma_i \sigma_{i+1} \text{ for all } i \\
\sigma_i \sigma_j = \sigma_j \sigma_i \text{ for all } i,j \text{ s.t. } |i-j| \geq 2 \end{array}\right\rangle\right..\]
\end{dfn}

However, in this paper we focus on another, equivalent, approach
of the group of braids, and we consider
the group of braids with $n$ strands as
the mapping class group of the unit disk with $n$ punctures
(see~\cite{birman-braids-links-mcg} for details).

Let $D^2 \subseteq \CC$ be the open unit disk
and let $P_n \subseteq (-1,1)$ be a set of size $n$.
We will refer to the elements of $P_n$ as being \emph{mobile punctures}
in the disk $D^2$.
In addition, we call \emph{fixed puncture} the point $-1$.
Finally, let $H_n$ be the group of orientation-preserving homeomorphisms
$h : \CC \mapsto \CC$ such that $h(P_n) = P_n$ and $h(z) = z$ for all $z \in (-\infty,-1) \cup (1,+\infty)$.

\begin{thm}
The group $B_n$ of braids with $n$ strands is isomorphic to the quotient group
of $H_n$ by the isotopy relation.
\label{thm:braids}
\end{thm}

Hence, a braid is the isotopy class $[h]$ of some homeomorphism $h$.
According to standard notations for braids, we will
let braids act on the complex plane \emph{on the right},
i.e. denote by $[g] [h]$ the isotopy class of the homeomorphism $[h \circ g]$:
composition to the left gives rise to a
braid multiplication to the right, and vice-versa.

This characterisation does \emph{not}
depend on which set $P_n$ of mobile punctures we choose.
Therefore, in what follows, we shall never consider the set $P_n$ as fixed.
In addition, we will always order punctures from left to right (i.e. from the smallest to the greatest), as follows:
the fixed puncture is $p_0 = -1$, and the mobile punctures are $p_1 < \ldots < p_n$.
We will also abuse notations and denote by $p_{n+1}$ the point $+1$.

Each braid appears as a class of homeomorphisms of $\CC$,
which conveys the idea of giving a graphical representation of the braid.

\begin{dfn}{Lamination}
Let us consider a set $P_n$ of $n$ mobile punctures in the real interval $(-1,1)$.

An \emph{$n$-strand lamination} of the disk $D^2$
is the union, hereafter denoted by $\calL$, of the set $P_n \cup \{-1\}$,
and of $n+1$ non-intersecting smooth, closed, simple curves $\calL_0,\ldots,\calL_n$ such that
each curve $\calL_j$ crosses exactly once the real interval $(-\infty,-1)$,
does not cross the real interval $[1,+\infty)$ and splits the plane $\CC$ into
(a) one \emph{inner} part that contains $-1$ and $j$ mobile punctures, and
(b) one \emph{outer} part that contains $+1$ and $n-j$ mobile punctures.
\label{dfn:trivial-lamination}
\end{dfn}

We identify braids with isotopy classes of laminations (see~\cite{birman-braids-links-mcg} for details):
if $\calL$ and $\calL'$ are laminations (defined up to isotopy), then there exists
a unique braid $\beta$ that sends $\calL$ to $\calL'$. 
In Figure~\ref{fig:laminations} as well as in the sequel of the document,
mobile punctures are indicated by white dots
and the fixed puncture is indicated by a black dot;
the gray area represents the unit disk $D^2$,
and the real axis $\RR$ is drawn with a thin horizontal line.
Hereafter, and depending on the context, we may
omit to represent the unit disk $D^2$,
as well as the names $p_0,\ldots,p_n$.

In addition, we will always denote by $\calL_j$ the $j$-th curve of the lamination $\calL$,
i.e. the curve of $\calL$ whose left part contains exactly $j$ mobile punctures.
In particular, note that the names of the punctures $p_0,\ldots,p_n$ depend uniquely of the order of the punctures.
Hence, when applying the braid $\sigma_i^{\pm 1}$ on a lamination,
both punctures $p_i$ and $p_{i+1}$ move, and they are respectively renamed $p_{i+1}$ and $p_i$.
On the contrary, and although the curves $\calL_j$ may move, they are not renamed,
as shown in Figure~\ref{fig:laminations}.

\begin{figure}[!ht]
\begin{center}
\begin{tikzpicture}[scale=0.23]
\draw[fill=palegray,draw=palegray] (16,0) circle (8);
\draw[draw=black] (3,0) -- (24,0);
\PUNCTURE{8}{0}

\draw[draw=black,ultra thick] (10,0) arc (0:360:2);
\draw[draw=black,ultra thick] (14,0) arc (0:360:4.5);
\draw[draw=black,ultra thick] (18,0) arc (0:360:7);
\draw[draw=black,ultra thick] (22,0) arc (0:360:9.5);

\puncture{12}{1}
\puncture{16}{2}
\puncture{20}{3}

\node[anchor=north] at (13.5,-11.5) {\small(a) Trivial lamination};

\node[anchor=south] at (12.5,9.25) {\small $\calL_3$};
\node[anchor=south] at (11,6.75) {\small $\calL_2$};
\node[anchor=south] at (9.5,4.25) {\small $\calL_1$};
\node[anchor=south] at (8,1.75) {\small $\calL_0$};

\draw[fill=palegray,draw=palegray] (47,0) circle (9);
\draw[draw=black] (33,0) -- (56,0);
\PUNCTURE{38}{0}

\draw[draw=black,ultra thick] (40,0) arc (0:360:2);
\draw[draw=black,ultra thick] (44,0) arc (0:360:4.5);
\draw[draw=black,ultra thick] (45,0) arc (0:180:5.5);
\draw[draw=black,ultra thick] (53,0) arc (360:180:9.5);
\draw[draw=black,ultra thick] (54,0) arc (360:0:10.5);

\draw[draw=black,ultra thick] (49,0) arc (180:0:2);
\draw[draw=black,ultra thick] (45,0) arc (180:360:2);

\puncture{42}{1}
\puncture{47}{2}
\puncture{51}{3}

\node[anchor=north] at (44.5,-11.5) {\small(b)Non-trivial lamination};

\node[anchor=south] at (43.5,10.25) {\small $\calL_3$};
\node[anchor=south] at (39.5,5.25) {\small $\calL_2$};
\node[anchor=north] at (39.5,-4.25) {\small $\calL_1$};
\node[anchor=south] at (38,1.75) {\small $\calL_0$};
\end{tikzpicture}
\end{center}
\caption{Two laminations}
\label{fig:laminations}
\end{figure}
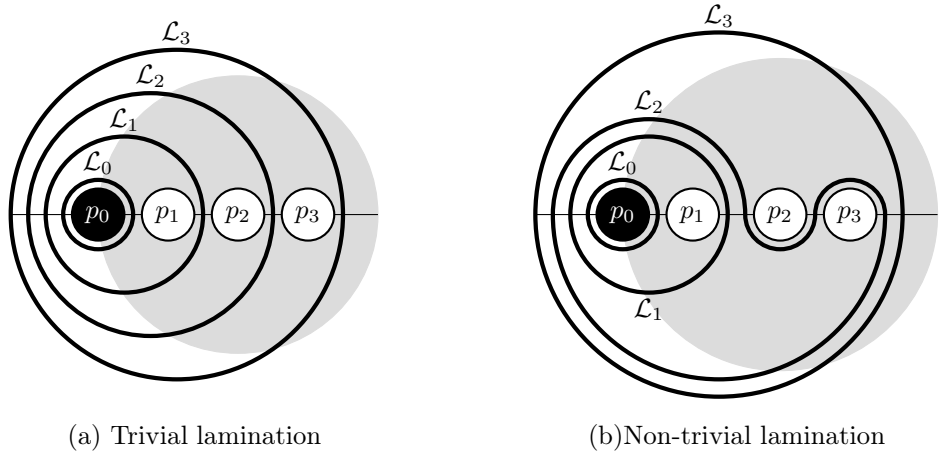

Following Dynnikov and Wiest~\cite{dynnikov:hal-00001267},
we define the norm of a lamination and the \emph{laminated norm} of a braid.

\begin{dfn}{Norm of a lamination and tightness}
Let $\beta \in B_n$ be a braid with $n$ strands,
and let $\calL$ be a lamination representing $\beta$.
The \emph{norm} of $\calL$, denoted by $\|\calL\|$,
is defined as the cardinality of the set $\calL \cap \RR$.

If, among all the laminations that represent $\beta$,
$\calL$ is a lamination with the minimal norm, then we say that
$\calL$ is \emph{tight}.
In this case, we also define the
\emph{laminated norm} of the braid $\beta$, which we denote by $\|\beta\|$, as
the integer $\|\calL\|$ itself.
\label{dfn:laminated-norm}
\end{dfn}

\begin{figure}[!ht]
\begin{center}
\begin{subfigure}[!ht]{0.32\textwidth}
\begin{center}
\begin{tikzpicture}[scale=0.14]
\draw[fill=palegray,draw=palegray] (16,0) circle (8);
\draw[draw=black] (3,0) -- (24,0);
\PUNC{8}

\draw[draw=black,ultra thick] (10,0) arc (0:360:2);
\draw[draw=black,ultra thick] (14,0) arc (0:360:4.5);
\draw[draw=black,ultra thick] (18,0) arc (0:360:7);
\draw[draw=black,ultra thick] (22,0) arc (0:360:9.5);

\punc{12}
\punc{16}
\punc{20}

\node at (10,12.5) {};
\node at (10,-12.5) {};
\end{tikzpicture}
\end{center}
\caption{$\varepsilon$}
\label{fig:lamination:epsilon}
\end{subfigure}
{\tiny~}
\begin{subfigure}[!ht]{0.32\textwidth}
\begin{center}
\begin{tikzpicture}[scale=0.14]
\draw[fill=palegray,draw=palegray] (17,0) circle (9);
\draw[draw=black] (3,0) -- (26,0);
\PUNC{8}

\draw[draw=black,ultra thick] (10,0) arc (0:360:2);
\draw[draw=black,ultra thick] (14,0) arc (0:360:4.5);
\draw[draw=black,ultra thick] (15,0) arc (0:180:5.5);
\draw[draw=black,ultra thick] (23,0) arc (360:180:9.5);
\draw[draw=black,ultra thick] (24,0) arc (0:360:10.5);

\draw[draw=black,ultra thick] (19,0) arc (180:0:2);
\draw[draw=black,ultra thick] (15,0) arc (180:360:2);

\punc{12}
\punc{17}
\punc{21}

\node at (10,12.5) {};
\node at (10,-12.5) {};
\end{tikzpicture}
\end{center}
\caption{$\sigma_2^{-1}$}
\label{fig:lamination:s2(-1)}
\end{subfigure}
{\tiny~}
\begin{subfigure}[!ht]{0.32\textwidth}
\begin{center}
\begin{tikzpicture}[scale=0.14]
\draw[fill=palegray,draw=palegray] (20,0) circle (11);
\draw[draw=black] (4,0) -- (31,0);
\PUNC{9}

\draw[draw=black,ultra thick] (11,0) arc (0:360:2);
\draw[draw=black,ultra thick] (12,0) arc (360:180:3);
\draw[draw=black,ultra thick] (22,0) arc (0:180:8);
\draw[draw=black,ultra thick] (23,0) arc (0:180:9);
\draw[draw=black,ultra thick] (28,0) arc (360:180:11.5);
\draw[draw=black,ultra thick] (29,0) arc (0:360:12.5);

\draw[draw=black,ultra thick] (12,0) arc (180:0:3);
\draw[draw=black,ultra thick] (13,0) arc (180:0:2);
\draw[draw=black,ultra thick] (24,0) arc (180:0:2);

\draw[draw=black,ultra thick] (13,0) arc (180:360:5.5);
\draw[draw=black,ultra thick] (17,0) arc (180:360:3);
\draw[draw=black,ultra thick] (18,0) arc (180:360:2);

\punc{15}
\punc{20}
\punc{26}
\end{tikzpicture}
\end{center}
\caption{$\sigma_2^{-1} \sigma_1$}
\label{fig:lamination:s2(-1).s1}
\end{subfigure}
\end{center}
\caption{Identifying braids with tight laminations}
\label{fig:3-laminations}
\end{figure}
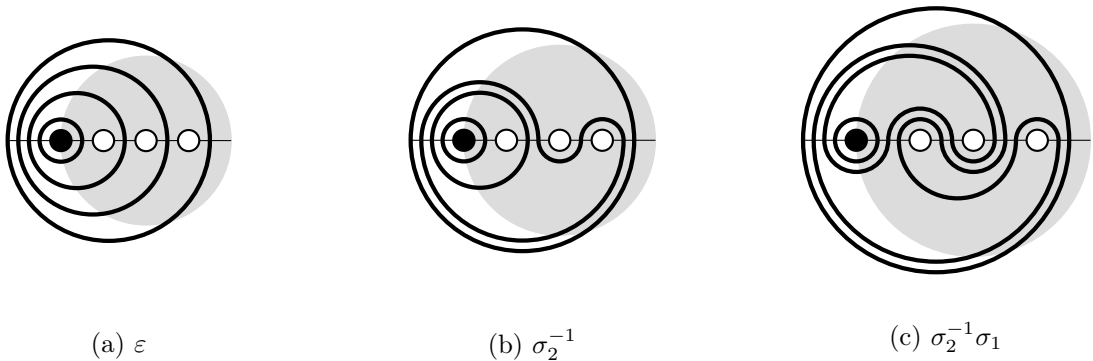

Tight laminations are important, due to the following classical result
(see~\cite{dehornoy2008ordering,dynnikov:hal-00001267,Fenn_orderingthe} for details).
We will henceforth refer to \emph{the} tight lamination of a braid,
as illustrated in Figure~\ref{fig:3-laminations}.

\begin{thm}
Two tight laminations represent the same braid if and only if they are related by an isotopy that
preserves the real axis $\RR$ setwise, and both points $-1$ and $1$ (pointwise).
\label{thm:braid=tight}
\end{thm}

\subsection{Relaxation Normal Form}

Performing relaxation algorithms on laminations of the punctured disk
requires identifying which simplifications might be performed on a lamination,
in order to decrease its complexity.
Such simplifications may arise from considering the notion of \emph{bigon} that we introduce now.

\begin{dfn}{Arcs, bigons and coverage}
Let $\calL$ be a lamination, which consists in closed curves $\calL_0, \ldots, \calL_n$
and in punctures $p_0, \ldots, p_n$.
For each curve $\calL_i$, we call \emph{arcs} of $\calL_i$ the connected components of $\calL_i \setminus \RR$.
By extension, we call \emph{arcs} of the lamination $\calL$ the arcs of its curves.

By construction, every arc $\calA$ of every lamination either lies in the upper half-plane (we say that $\calA$ is an \emph{upper} arc)
or in the lower half-plane (we say that $\calA$ is a \emph{lower} arc),
and it has two distinct endpoints, which both lie on the real axis. Let us denote them by $e_\calA$ and $E_\calA$.
We say that a real point $p$ is \emph{covered} by the arc $\calA$ if $p$ belongs to the open interval $(e_\calA,E_\calA)$.
By extension, for all arcs $\calB$, with endpoints $e_\calB$ and $E_\calB$, we say that
$\calB$ is \emph{covered} by $\calA$ if both $e_\calB$ and $E_\calB$ are covered by $\calA$, i.e. if $e_\calA < e_\calB < E_\calB < E_\calA$.
Finally, we say that $\calA$ is a \emph{bigon} of $\calL$ if $\calA$ does not cover any endpoint of any arc of $\calL$.
\label{dfn:arcs-bigons}
\end{dfn}

\begin{figure}[!ht]
\begin{center}
\begin{tikzpicture}[scale=0.18]
\draw[draw=gray,ultra thick] (28,0) arc (360:180:11.5);

\draw[draw=gray] (4,0) -- (34,0);

\draw[draw=black,ultra thick] (7,0) arc (180:-180:2);
\draw[draw=gray,ultra thick] (6,0) arc (180:0:8) arc (360:180:2) arc (0:180:3) arc (360:180:3);
\draw[draw=gray,ultra thick] (5,0) arc (180:0:14) arc (360:180:2) arc (0:180:3) arc (360:180:3) arc (0:180:2) arc (180:360:5.5) arc (180:0:2) arc (360:180:11.5);
\draw[draw=gray,ultra thick] (4,0) arc (180:-180:15);
\draw[draw=black,ultra thick] (13,0) arc (180:0:2);
\draw[draw=black,ultra thick] (18,0) arc (180:360:2);
\draw[draw=black,ultra thick] (24,0) arc (180:0:2);
\draw[draw=black,ultra thick] (29,0) arc (180:360:2);

\draw[fill=gray,draw=gray,thick] (9,0) circle (1);
\draw[fill=white,draw=gray,thick] (15,0) circle (1);
\draw[fill=white,draw=gray,thick] (20,0) circle (1);
\draw[fill=white,draw=gray,thick] (26,0) circle (1);

\draw[->,draw=black,thick] (9.5,13.5) -- (9.5,-1.6);
\draw[->,draw=black,thick] (8.5,13.5) -- (8.5,2.2);
\draw[->,draw=black,thick] (15,13.5) -- (15,2.3);
\draw[->,draw=black,thick] (20,13.5) -- (20,-1.7);
\draw[->,draw=black,thick] (26,13.5) -- (26,2.3);
\draw[draw=black,thick] (8.5,13) -- (8.5,13.5) -- (26,13.5) -- (26,13);
\draw[draw=black,thick] (17.25,13.5) -- (17.25,15);
\draw[->,draw=black,thick] (31,-15) -- (31,-2.3);

\node[anchor=south] at (17.25,15) {useful bigons};
\node[anchor=north] at (31,-15) {useless bigon};

\draw[->,draw=gray,ultra thick] (35,0) -- (43,0);

\draw[draw=gray] (44,0) -- (74,0);

\draw[draw=gray,ultra thick] (47,0) arc (180:-180:2);
\draw[draw=gray,ultra thick] (46,0) arc (180:0:8) arc (360:180:2) arc (0:180:3) arc (360:180:3);
\draw[draw=gray,ultra thick] (45,0) arc (180:2.0467:14) -- (68.9580,0.5) arc (9.5941:180:3) arc (360:180:3) arc (0:180:2) arc (180:360:5.5) arc (180:0:2) arc (360:180:11.5);
\draw[draw=gray,ultra thick] (44,0) arc (180:-180:15);

\draw[fill=gray,draw=gray,thick] (49,0) circle (1);
\draw[fill=white,draw=gray,thick] (55,0) circle (1);
\draw[fill=white,draw=gray,thick] (60,0) circle (1);
\draw[fill=white,draw=gray,thick] (66,0) circle (1);
\end{tikzpicture}
\end{center}
\caption{Chopping a useless bigon of a lamination}
\label{fig:bigons-1}
\end{figure}
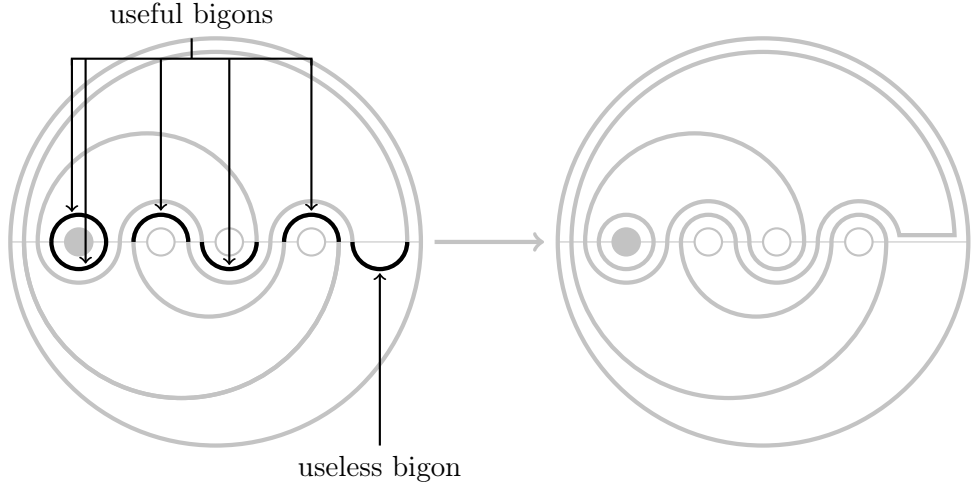

First, observe that a tight lamination $\calL$ should be \emph{transversal} to the real axis,
i.e. each intersection point between a curve of $\calL$ and the real axis
shall be an endpoint of one lower arc of $\calL$ and one upper arc of $\calL$.
Similarly, if some bigon of a lamination $\calL$ does not cover any (fixed or mobile) puncture
of $\calL$, then $\calL$ is certainly not tight, since chopping the useless bigon in question
would reduce the complexity of $\calL$ without changing its isotopy class,
as illustrated in Figure~\ref{fig:bigons-1}.

It turns out that the converse statement holds, and that bigons provide an
intrinsic and easy characterization of tight laminations
(see~\cite{dehornoy2008ordering, Fenn_orderingthe} for details).

\begin{pro}
A lamination $\calL$ is tight if and only if each of its bigons covers at at least one puncture of $\calL$.
\label{pro:tight-bigons}
\end{pro}

In particular, observe that every lamination $\calL$ that is already transversal to the real axis
and whose complexity is finite can be made tight by chopping recursively all its useless bigons.
Hence, using this tightening procedure, we present the principle of relaxation algorithms,
which also relies on the following results,
whose proof is postponed to section~\ref{section:proofs}.

\begin{lem}
Let $\calL$ be a non-trivial tight lamination. At least one mobile puncture $p$ of $\calL$
is covered by a bigon $\calB$ of $\calL$.
Once that puncture is selected, the bigon $\calB$ is unique.
Furthermore, both endpoints of $\calB$ also belong to arcs $\calA_1$ and $\calA_2$ of $\calL$,
distinct from each other, and such that $\calA_1$ and $\calA_2$ do not both cover the fixed puncture $-1$.
\label{lem:easy-stuff}
\end{lem}

As mentioned in the introduction, the class of relaxation algorithms that we present below also
relies on the notion of \emph{sliding braid}.
Sliding braids are defined as the braids of on of the following forms:
\begin{center}
\begin{tabular}{lll}
$\displaystyle[k \curvearrowright \ell] = \sigma_k \ldots \sigma_{\ell-1}$ & \hspace*{20pt} & $\displaystyle[k \curvearrowleft \ell] = \sigma_{\ell-1}^{-1} \ldots \sigma_k^{-1}$ \\
$\displaystyle[k \curvearrowbotright \ell] = \sigma_k^{-1} \ldots \sigma_{\ell-1}^{-1}$ & & $\displaystyle[k \curvearrowbotleft \ell] = \sigma_{\ell-1} \ldots \sigma_k$
\end{tabular}
\end{center}
for $k \leq \ell$. In particular, observe that the trivial braid $\varepsilon$ is a sliding braid (consider the case $k = \ell$ in the above definitions)
and that the relations $[k \curvearrowleft \ell] = [k \curvearrowright \ell]^{-1}$ and $[k \curvearrowbotleft \ell] = [k \curvearrowbotright \ell]^{-1}$ hold for all $k \leq \ell$.

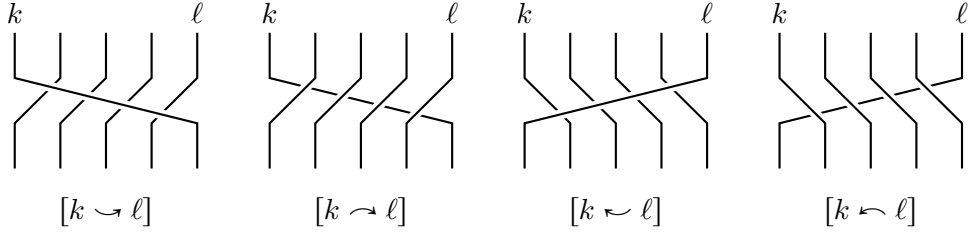
\begin{figure}[!ht]
\begin{center}
\begin{tikzpicture}[scale=0.12]
\draw[draw=black,thick] (0,0) -- (0,5) -- (5,10) -- (5,15);
\draw[draw=black,thick] (5,0) -- (5,5) -- (10,10) -- (10,15);
\draw[draw=black,thick] (10,0) -- (10,5) -- (15,10) -- (15,15);
\draw[draw=black,thick] (15,0) -- (15,5) -- (20,10) -- (20,15);
\draw[draw=white,line width=1.3mm] (20,0) -- (20,5) -- (0,10) -- (0,15);
\draw[draw=black,thick] (20,0) -- (20,5) -- (0,10) -- (0,15);
\node[anchor=south] at (0,15) {$k$};
\node[anchor=south] at (20,15) {$\ell$};
\node[anchor=north] at (10,-2) {$[k \curvearrowbotright \ell]$};
\end{tikzpicture}
~~
\begin{tikzpicture}[scale=0.12]
\draw[draw=black,thick] (20,0) -- (20,5) -- (0,10) -- (0,15);
\draw[draw=white,line width=1.3mm] (0,0) -- (0,5) -- (5,10) -- (5,15);
\draw[draw=white,line width=1.3mm] (5,0) -- (5,5) -- (10,10) -- (10,15);
\draw[draw=white,line width=1.3mm] (10,0) -- (10,5) -- (15,10) -- (15,15);
\draw[draw=white,line width=1.3mm] (15,0) -- (15,5) -- (20,10) -- (20,15);
\draw[draw=black,thick] (0,0) -- (0,5) -- (5,10) -- (5,15);
\draw[draw=black,thick] (5,0) -- (5,5) -- (10,10) -- (10,15);
\draw[draw=black,thick] (10,0) -- (10,5) -- (15,10) -- (15,15);
\draw[draw=black,thick] (15,0) -- (15,5) -- (20,10) -- (20,15);
\node[anchor=south] at (0,15) {$k$};
\node[anchor=south] at (20,15) {$\ell$};
\node[anchor=north] at (10,-2) {$[k \curvearrowright \ell]$};
\end{tikzpicture}
~~
\begin{tikzpicture}[scale=0.12]
\draw[draw=black,thick] (20,0) -- (20,5) -- (15,10) -- (15,15);
\draw[draw=black,thick] (15,0) -- (15,5) -- (10,10) -- (10,15);
\draw[draw=black,thick] (10,0) -- (10,5) -- (5,10) -- (5,15);
\draw[draw=black,thick] (5,0) -- (5,5) -- (0,10) -- (0,15);
\draw[draw=white,line width=1.3mm] (0,0) -- (0,5) -- (20,10) -- (20,15);
\draw[draw=black,thick] (0,0) -- (0,5) -- (20,10) -- (20,15);
\node[anchor=south] at (0,15) {$k$};
\node[anchor=south] at (20,15) {$\ell$};
\node[anchor=north] at (10,-2) {$[k \curvearrowbotleft \ell]$};
\end{tikzpicture}
~~
\begin{tikzpicture}[scale=0.12]
\draw[draw=black,thick] (0,0) -- (0,5) -- (20,10) -- (20,15);
\draw[draw=white,line width=1.3mm] (20,0) -- (20,5) -- (15,10) -- (15,15);
\draw[draw=white,line width=1.3mm] (15,0) -- (15,5) -- (10,10) -- (10,15);
\draw[draw=white,line width=1.3mm] (10,0) -- (10,5) -- (5,10) -- (5,15);
\draw[draw=white,line width=1.3mm] (5,0) -- (5,5) -- (0,10) -- (0,15);
\draw[draw=black,thick] (20,0) -- (20,5) -- (15,10) -- (15,15);
\draw[draw=black,thick] (15,0) -- (15,5) -- (10,10) -- (10,15);
\draw[draw=black,thick] (10,0) -- (10,5) -- (5,10) -- (5,15);
\draw[draw=black,thick] (5,0) -- (5,5) -- (0,10) -- (0,15);
\node[anchor=south] at (0,15) {$k$};
\node[anchor=south] at (20,15) {$\ell$};
\node[anchor=north] at (10,-2) {$[k \curvearrowleft \ell]$};
\end{tikzpicture}
\end{center}
\caption{Braid diagrams of sliding braids}
\label{fig:sliding-braids}
\end{figure}

We show now how to relax a non-trivial tight lamination $\calL$.
First, choose some mobile puncture $p_k$ that belongs to a bigon $\calB$ (which is unique once $p_k$ is chosen).
Then, choose some endpoint $e$ of $\calB$ that also belongs to an arc $\calA$ of $\calL$,
such that $\calA$ does not contain the fixed puncture $-1$, and let $E$ be the other endpoint of $\calA$.
Now, let us slide the puncture $p_k$ along the arc $\calA$,
so that it lands next to the endpoint $E$, and let $\calL'$ be the lamination obtained after $p_k$ has been slid.

Such sliding operations can be performed by letting a \emph{sliding braid} act on the lamination $\calL$.
More precisely, let $i \in \{0,\ldots,n\}$ be the unique integer such that $E$ belongs to the interval $(p_i,p_{i+1})$, with the convention that $p_{n+1} = 1$.
Then $\calL'$ belongs to the isotopy class defined by $\calL \cdot \alpha$, where $\alpha$ is the sliding braid defined by:
\[\alpha = \begin{cases}
[i+1 \curvearrowleft k] & \text{if } i \leq k-1 \text{ and } \calA \text{ is an upper arc} \\
[i+1 \curvearrowbotleft k] & \text{if } i \leq k-1 \text{ and } \calA \text{ is a lower arc} \\
[k \curvearrowright i] & \text{if } i \geq k \text{ and } \calA \text{ is an upper arc} \\
[k \curvearrowbotright i] & \text{if } i \geq k \text{ and } \calA \text{ is a lower arc}
\end{cases}\]

Hence, let $\beta$ be the braid represented by $\calL$, so that $\calL'$ represents the braid $\beta \alpha$.
We say that $\alpha$ is a \emph{relaxing braid} of $\beta$.
Since $\calB$ has become a useless bigon of $\calL'$, it follows that $\calL'$ is not tight, and that
$\|\beta\| = \|\calL\| = \|\calL'\| > \|\beta \alpha\|$.
A natural idea is then to tighten the lamination $\calL'$, relax it, tighten it again, and so on, until we obtain the trivial lamination $\bL$, if we ever do.
This amounts to choosing sliding braids $\alpha_1, \alpha_2, \ldots$ such that
each braid $\alpha_j$ is a relaxing braid of $\beta \alpha_1 \ldots \alpha_{j-1}$.

Due to the strict inequalities $\|\beta\| > \|\beta \alpha_1\| > \|\beta \alpha_1 \alpha_2\| > \ldots$,
it is certain that we will indeed obtain the trivial lamination after a finite number of tightenings and relaxations.
At that point, we will have chosen relaxing braids $\alpha_1, \ldots, \alpha_k$, so that
$\varepsilon = \beta \alpha_1 \ldots \alpha_k$.
It follows that $\beta$ can be factored as the product $\alpha_k^{-1} \ldots \alpha_2^{-1} \alpha_1^{-1}$ of sliding braids.

Of course, the puncture $p_k$ and the endpoint $e$ (and therefore the braid $\alpha_1$) picked at each step of the relaxation process
might have been chosen in numerous ways.
We describe now how the puncture $p_k$ and the endpoint $e$ are chosen in order to obtain the relaxation normal form.

\begin{dfn}{Relaxation normal form}
Let $\calL$ be a non-trivial tight lamination, and let $\beta$ be the braid represented by $\calL$.
Let $p_k$ be the rightmost mobile puncture of $\calL$ covered by some bigon $\calB$ of $\calL$.
We call this bigon the \emph{rightmost bigon} of $\calL$.
Let also $e_1 < e_2$ be the two endpoints of $\calB$, and let $\calA_1$ and $\calA_2$ be the
arcs of $\calL$ with which the bigon $\calB$ shares its endpoints $e_1$ and $e_2$, respectively.

If $\calA_2$ does not cover the fixed puncture $-1$, then we slide $p_k$ along the arc $\calA_2$;
otherwise, we slide $p_k$ along the arc $\calA_1$. Let $\calL'$ be the lamination obtained after $p_k$ has been slid.
We call \emph{right-relaxing braid} of $\beta$, and denote by $\bR(\beta)$, the braid $\sigma$ such that $\calL' = \calL \cdot \sigma$.

Furthermore, let $\beta_1, \ldots, \beta_k$ be the unique sequence such that
$\beta_1 = \beta$, $\beta_k = \varepsilon$ and, for all $i \leq k-1$,
$\beta_i$ is a non-trivial braid such that $\beta_{i+1} = \beta_i \bR(\beta_i)$.
We call \emph{relaxation normal form} of $\beta$, and denote by $\RNF(\beta)$,
the word
\[\bR(\beta_{k-1})^{-1} \cdot \ldots \cdot \bR(\beta_1)^{-1},\] where
$\cdot$ is the concatenation symbol.
\label{dfn:RNF}
\end{dfn}

Figure~\ref{fig:right-relax} illustrates the right-relaxation procedure in the case of the braid
$\beta = \sigma_2 \sigma_1 \sigma_3^{-1}$, whose relaxation normal form turns out to be
$\RNF(\beta) = [2 \curvearrowright 3] \cdot [1 \curvearrowright 3] \cdot [2 \curvearrowbotright 4] = \sigma_2 \cdot \sigma_1 \sigma_2 \cdot \sigma_2^{-1} \sigma_3^{-1}$.

\begin{figure}[!ht]
\begin{center}
\begin{tikzpicture}[scale=0.17]
\draw[draw=gray] (0,0) -- (34,0);

\draw[draw=gray,ultra thick] (4,0) arc (180:-180:2);
\draw[draw=gray,ultra thick] (3,0) arc (180:0:9) arc (360:180:2) arc (0:180:4) arc (360:180:3);
\draw[draw=gray,ultra thick] (2,0) arc (180:0:10) arc (180:360:3) arc (180:0:2) arc (360:180:9) arc (0:180:2) arc (360:180:4);
\draw[draw=gray,ultra thick] (1,0) arc (180:0:11) arc (180:360:2) arc (180:0:3) arc (360:180:16);
\draw[draw=gray,ultra thick] (0,0) arc (180:-180:17);

\draw[draw=black,ultra thick,->] (30,0) arc (360:180:7.25);

\draw[fill=black,draw=black,thick] (6,0) circle (1);
\draw[fill=black!15,draw=black,thick] (12,0) circle (1);
\draw[fill=black!60,draw=black,thick] (19,0) circle (1);
\draw[fill=white,draw=black,thick] (25,0) circle (1);
\draw[fill=black!35,draw=black,thick] (30,0) circle (1);

\draw[draw=black,ultra thick] (28,0) arc (180:0:2);

\draw[->,draw=black,thick] (30,17) -- (30,2.3);

\node[anchor=south] at (22.75,-6) {sliding};
\node[anchor=south] at (22.75,-7.3) {move};
\node[anchor=south] at (28,16.5) {rightmost bigon};

\draw[->,draw=gray,ultra thick] (35,0) -- (42,0);
\node[anchor=north] at (38.5,0) {$\alpha_1$};
\node[anchor=north] at (38.5,-2) {=};
\node[anchor=north] at (38.5,-3.3) {$[2 \curvearrowbotleft 4]$};

\draw[draw=gray] (43,0) -- (73,0);

\draw[draw=gray,ultra thick] (47,0) arc (180:-180:2);
\draw[draw=gray,ultra thick] (46,0) arc (180:0:10.5) arc (360:180:2) arc (0:180:4.5) arc (360:180:4);
\draw[draw=gray,ultra thick] (45,0) arc (180:0:11.5) arc (360:180:4.5) arc (0:180:2) arc (360:180:5);
\draw[draw=gray,ultra thick] (44,0) arc (180:-180:12.5);
\draw[draw=gray,ultra thick] (43,0) arc (180:-180:15);

\draw[draw=black,ultra thick,->] (65,0) arc (0:180:6.25);

\draw[fill=black,draw=black,thick] (49,0) circle (1);
\draw[fill=black!15,draw=black,thick] (57,0) circle (1);
\draw[fill=black!35,draw=black,thick] (61,0) circle (1);
\draw[fill=black!60,draw=black,thick] (65,0) circle (1);
\draw[fill=white,draw=black,thick] (71,0) circle (1);

\draw[draw=black,ultra thick] (63,0) arc (180:360:2);

\draw[->,draw=gray,ultra thick] (58,-16) -- (58,-23);
\node[anchor=west] at (58,-19.5) {$\alpha_2 = [1 \curvearrowleft 3]$};

\draw[draw=gray] (44,-38) -- (72,-38);

\draw[draw=gray,ultra thick] (48,-38) arc (180:-180:2);
\draw[draw=gray,ultra thick] (47,-38) arc (180:-180:4.5);
\draw[draw=gray,ultra thick] (46,-38) arc (180:0:10.5) arc (360:180:2) arc (0:180:2) arc (360:180:6.5);
\draw[draw=gray,ultra thick] (45,-38) arc (180:-180:11.5);
\draw[draw=gray,ultra thick] (44,-38) arc (180:-180:14);

\draw[draw=black,ultra thick,->] (65,-38) arc (0:180:3.75);

\draw[fill=black,draw=black,thick] (50,-38) circle (1);
\draw[fill=black!60,draw=black,thick] (54,-38) circle (1);
\draw[fill=black!15,draw=black,thick] (61,-38) circle (1);
\draw[fill=black!35,draw=black,thick] (65,-38) circle (1);
\draw[fill=white,draw=black,thick] (70,-38) circle (1);

\draw[draw=black,ultra thick] (63,-38) arc (180:360:2);

\draw[<-,draw=gray,ultra thick] (36,-38) -- (43,-38);
\node[anchor=north] at (39.5,-38) {$\alpha_3$};
\node[anchor=north] at (39.5,-40) {=};
\node[anchor=north] at (39.5,-41.3) {$[2 \curvearrowleft 3]$};

\draw[draw=gray] (11,-38) -- (35,-38);

\draw[draw=gray,ultra thick] (15,-38) arc (180:-180:2);
\draw[draw=gray,ultra thick] (14,-38) arc (180:-180:4.5);
\draw[draw=gray,ultra thick] (13,-38) arc (180:-180:7);
\draw[draw=gray,ultra thick] (12,-38) arc (180:-180:9.5);
\draw[draw=gray,ultra thick] (11,-38) arc (180:-180:12);

\draw[fill=black,draw=black,thick] (17,-38) circle (1);
\draw[fill=black!60,draw=black,thick] (21,-38) circle (1);
\draw[fill=black!35,draw=black,thick] (25,-38) circle (1);
\draw[fill=black!15,draw=black,thick] (29,-38) circle (1);
\draw[fill=white,draw=black,thick] (33,-38) circle (1);
\end{tikzpicture}
\end{center}
\caption{Applying the right-relaxation procedure on the braid $\sigma_2 \sigma_1 \sigma_3^{-1}$}
\label{fig:right-relax}
\end{figure}
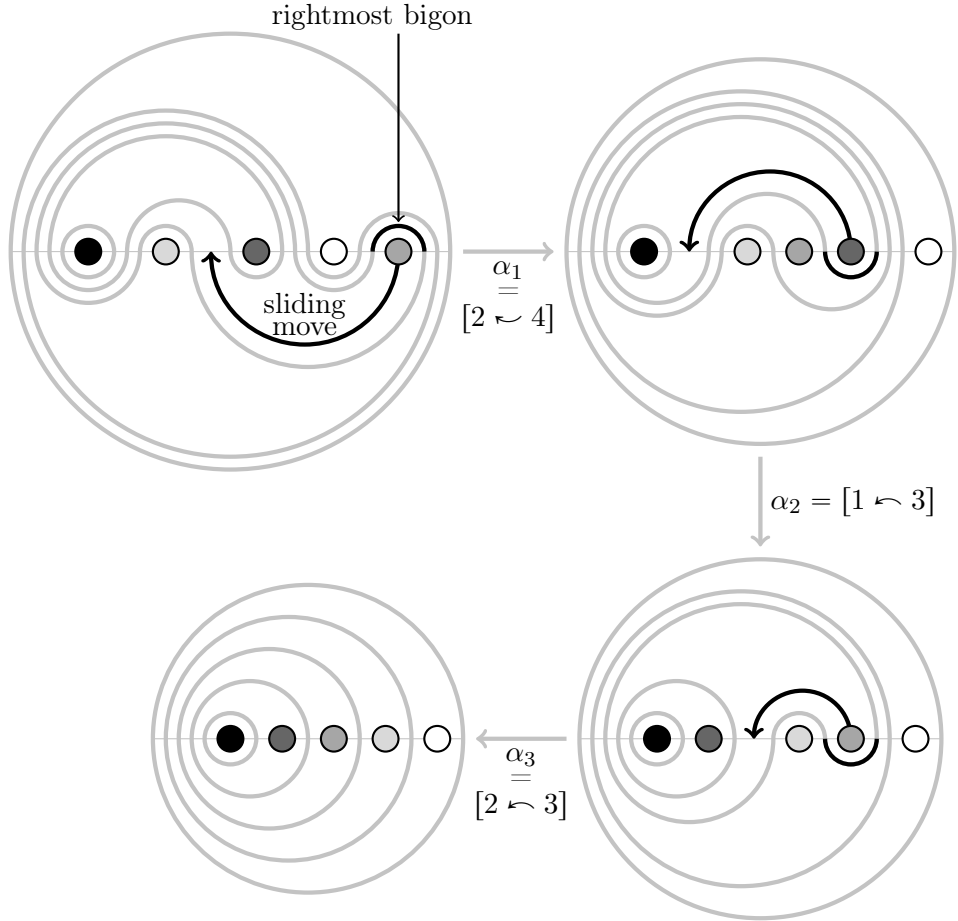

A crucial point is that, due to Theorem~\ref{thm:braid=tight},
the puncture $p_k$ and the arc $\calA_1$ or $\calA_2$
along which we slide $p_k$ do \emph{not} depend on which tight lamination $\calL$
representing $\beta$ we started with.
It follows that the language $\{\RNF(\beta) : \beta \in B_n\}$
of the relaxation normal form is prefix-closed.

This prefix-closure property offers numerous possibilities.
The relaxation normal form induces a tree, whose nodes are the words in relaxation normal form,
and where the children of a word $\bw$ are the words of the type $\bw \cdot \lambda$
(for some sliding braid $\lambda$) that are in normal form.
Hence, this tree is a sub-graph of the oriented Cayley graph of $B_n$ for the sliding braids.

This is useful for studying random processes:
for instance, we may define a random walk by jumping from one word in relaxation normal form
to one of its children that we choose at random.
Another example is testing if a word is in relaxation normal form:
it is possible to proceed by induction,
checking only whether, for some relaxation normal word $\bw$
and some sliding braid $\lambda$, the word $\bw \cdot \lambda$ is in relaxation normal form.
We will use the latter property when proving that the set $\{\RNF(\beta) : \beta \in B_n\}$ is regular.

\subsection{Recognizing the Relaxation Normal Form}
\label{subsection:sketch}

Based on the previous discussion, our goal is now to devise a simple criterion
for deciding, given a braid $\beta \in B_n$ and a sliding braid $\lambda$,
whether the relation $\lambda^{-1} = \bR(\beta \lambda)$ holds.
More precisely, since we aim at constructing an automaton recognizing the relaxation normal form,
we would like to extract from the braid $\beta$ some amount of information $\iota(\beta)$
such that knowing $\iota(\beta)$ and $\lambda$ would be sufficient to decide whether $\lambda^{-1} = \bR(\beta \lambda)$
and, if yes, to compute the information $\iota(\beta \lambda)$.

Since we are only interseted in finite-state automata, the information $\iota(\beta)$
should be \emph{succinct}, i.e. the set $\{\iota(\beta) : \beta \in B_n\}$ should be finite.
Indeed, this set would then be the state set of our automaton, whose edges
would be of the form $\iota(\beta) \xrightarrow{\lambda} \iota(\beta \lambda)$
for all those braids $\beta \in B_n$ and sliding braids $\lambda$ such that $\lambda^{-1} = \bR(\beta \lambda)$.

We describe now an adequately chosen amount of information,
and state the main results that pave the way to constructing
an automaton recognizing the relaxation normal form.
All the results stated below will be proved in section~\ref{section:proofs}.

\begin{dfn}{Neighbor points and neighbor arcs}
Let $\calL$ be a tight lamination, and let $\calL_\RR$ denote the set of all intersection points between
the real axis and the curves of $\calL$.
Let $p$ be a point on the open real interval $(\min \calL_\RR, \max \calL_\RR)$.
We call \emph{left neighbor} of $p$ in $\calL$ the point $p_\calL^- := \max \{z \in \calL_\RR: z < p\}$ and
\emph{right neighbor} of $p$ in $\calL$ the point $p_\calL^+ := \min \{z \in \calL_\RR: z > p\}$.

The point $p_\calL^-$ belongs to two arcs of $\calL$.
We denote the upper one by $\upA_-(p,\calL)$, and we call it \emph{left upper arc} of $p$ in $\calL$.
We denote the lower one by $\lowA_-(p,\calL)$, and we call it \emph{left lower arc} of $p$ in $\calL$.
Similarly, the point $p_\calL^+$ belongs to two arcs of $\calL$.
We denote the upper one by $\upA_+(p,\calL)$, and we call it \emph{right upper arc} of $p$ in $\calL$.
We denote the lower one by $\lowA_+(p,\calL)$, and we call it \emph{right lower arc} of $p$ in $\calL$.
These four arcs are called \emph{neighbor arcs} of $p$ in $\calL$.
\label{dfn:neighbor-arcs}
\end{dfn}

Figure~\ref{fig:neighbor-arcs} shows a tight lamination
in which a puncture $p$, the neighbor points of $p$ and the neighbor arcs of $p$ have been highlighted.

\begin{figure}[!ht]
\begin{center}
\begin{tikzpicture}[scale=0.22]
\draw[draw=gray] (4,0) -- (27,0);

\draw[fill=gray,draw=gray,thick] (9,0) circle (1);

\draw[draw=gray,ultra thick] (11,0) arc (0:360:2);
\draw[draw=gray,ultra thick] (25,0) arc (0:180:9.5);
\draw[draw=gray,ultra thick] (26,0) arc (0:180:10.5);
\draw[draw=gray,ultra thick] (27,0) arc (0:360:11.5);

\draw[draw=gray,ultra thick] (20,0) arc (180:360:3);
\draw[draw=gray,ultra thick] (21,0) arc (180:360:2);

\draw[fill=white,draw=gray,thick] (14,0) circle (1);
\draw[fill=white,draw=gray,thick] (18,0) circle (1);
\draw[fill=white,draw=gray,thick] (23,0) circle (1);

\draw[draw=black,thick] (28,8) -- (17,8) -- (17,4.5);
\draw[draw=black,ultra thick] (12,0) arc (180:0:4.5);
\draw[draw=black,thick] (28,3.5) -- (18,3.5) -- (18,2);
\draw[draw=black,ultra thick] (16,0) arc (180:0:2);
\draw[draw=black,thick] (28,-8) -- (10.5,-8) -- (10.5,-5.5);
\draw[draw=black,ultra thick] (16,0) arc (360:180:5.5);
\draw[draw=black,thick] (28,-4) -- (9,-4) -- (9,-3);
\draw[draw=black,ultra thick] (12,0) arc (360:180:3);

\draw[draw=black,thick] (3,2.5) -- (12,2.5) -- (12,0);
\draw[draw=black,thick] (3,6) -- (14,6) -- (14,0);
\draw[draw=black,thick] (3,9.5) -- (16,9.5) -- (16,0);

\dvertex{12}
\dvertex{14}
\dvertex{16}
\node[anchor=west] at (0,2.5) {$p_\calL^-$};
\node[anchor=west] at (0,6) {$p$};
\node[anchor=west] at (0,9.5) {$p_\calL^+$};
\node[anchor=west] at (28,8) {$\upA_-(p,\calL)$};
\node[anchor=west] at (28,3.5) {$\upA_+(p,\calL)$};
\node[anchor=west] at (28,-4) {$\lowA_-(p,\calL)$};
\node[anchor=west] at (28,-8) {$\lowA_+(p,\calL)$};

\end{tikzpicture}
\end{center}
\caption{A puncture, its neighbor points and its neighbor arcs}
\label{fig:neighbor-arcs}
\end{figure}
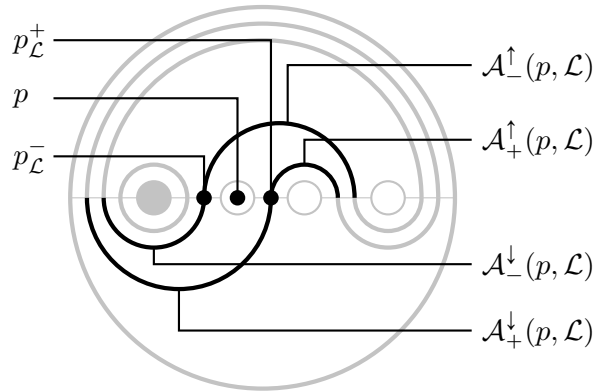

\begin{dfn}{Shadow and extended shadow}
Let $\calL$ be a tight lamination that represents some braid $\beta$, and let
$p_0, \ldots, p_n$ be the (fixed and mobile) punctures of $\calL$.
Let $p_k$ be the rightmost mobile puncture of $\calL$ covered by a bigon.
Finally, let $\calA$ be an arc of $\calL$, and let $\calB$ be the arc with which $\calA$ shares its rightmost endpoint.
We call \emph{shadow} of $\calA$, and denote by $\pi_\calL(\calA)$, the set $\{i : p_i$ is covered by the arc $\calA\}$.
Note that the shadow $\pi_\calL(\calA)$ does indeed depend on $\calL$, since it indicates which punctures $p_i$ of $\calL$ are covered by $\calA$.
We also call \emph{extended shadow} of $\calA$, and denote by $\pi_\calL^2(\calA)$, the pair defined by:
\[\pi_\calL^2(\calA) = \begin{cases}
(\pi_\calL(\calA),\pi_\calL(\calB)) & \text{if } k \in \pi_\calL(\calA) \\
(\pi_\calL(\calA),\emptyset) & \text{if } k \notin \pi_\calL(\calA)
\end{cases}\]

Then, for $1 \leq i \leq n$, $\diamond \in \{+,-\}$ and $\vartheta \in \{\downarrow,\uparrow\}$,
we denote by $\pi_\beta(i,\diamond,\vartheta)$ the shadow of the arc $\calA^\vartheta_\diamond(p_i,\calL)$ in $\calL$,
and we denote by $\pi_\beta^2(i,\diamond,\vartheta)$ the extended shadow of the arc $\calA^\vartheta_\diamond(p_i,\calL)$ in $\calL$.
By abuse of notation, we define the \emph{shadow} of $\beta$ as the mapping
\begin{center}
\begin{tabular}{lllll}
$\pi_\beta$ & $:$ & $\{1,\ldots,n\} \times \{+,-\} \times \{\downarrow,\uparrow\}$ & $\mapsto$ & $2^{\{0,\ldots,n\}}$ \\
& & $(i,\diamond,\vartheta)$ & $\mapsto$ & $\pi_\beta(i,\diamond,\vartheta)$
\end{tabular}
\end{center}
and the \emph{extended shadow} of $\beta$ as the mapping
\begin{center}
\begin{tabular}{lllll}
$\pi_\beta^2$ & $:$ & $\{1,\ldots,n\} \times \{+,-\} \times \{\downarrow,\uparrow\}$ & $\mapsto$ & $2^{\{0,\ldots,n\}} \times 2^{\{0,\ldots,n\}}$. \\
& & $(i,\diamond,\vartheta)$ & $\mapsto$ & $\pi_\beta^2(i,\diamond,\vartheta)$
\end{tabular}
\end{center}
\end{dfn}

Note that the above definitions of the neighbor points and arcs of a point,
as well as the shadow and the extended shadow of an arc, depend on which lamination $\calL$ is considered.
Hence, we prefer mentioning explicitly which lamination $\calL$ is considered, although
this choice leads to heavier notation. In particular, in the proofs of section~\ref{section:proofs},
we may consider some arc or point relatively to \emph{several} laminations,
and in that case mentioning which lamination we consider will be required.

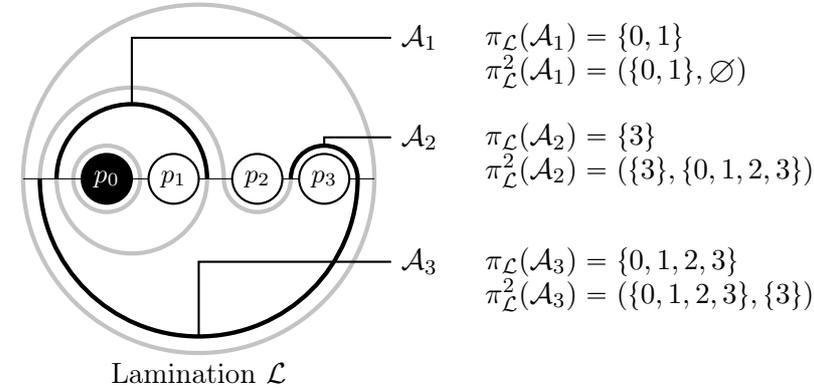
\begin{figure}[!ht]
\begin{center}
\begin{tikzpicture}[scale=0.22]

\draw[draw=gray,ultra thick] (3,0) arc (180:-180:2);
\draw[draw=gray,ultra thick] (2,0) arc (180:-180:4.5);
\draw[draw=gray,ultra thick] (1,0) arc (180:0:5.5) arc (180:360:2) arc (180:0:2) arc (360:180:9.5);
\draw[draw=gray,ultra thick] (0,0) arc (180:-180:10.5);

\draw[draw=black,ultra thick] (2,0) arc (180:0:4.5);
\draw[draw=black,ultra thick] (16,0) arc (180:0:2) arc (360:180:9.5);

\draw[draw=black] (0,0) -- (21,0);

\PUNCTURE{5}{0}
\puncture{9}{1}
\puncture{14}{2}
\puncture{18}{3}

\node[anchor=north] at (10.5,-10.5) {Lamination $\calL$};

\draw[draw=black,thick] (6.5,4.5) -- (6.5,8.5) -- (22,8.5);
\draw[draw=black,thick] (18,2) -- (18,2.5) -- (22,2.5);
\draw[draw=black,thick] (10.5,-9.5) -- (10.5,-5) -- (22,-5);

\node[anchor=west] at (22,8.5) {$\calA_1$};
\node[anchor=west] at (22,2.5) {$\calA_2$};
\node[anchor=west] at (22,-5) {$\calA_3$};

\node[anchor=west] at (27,8.5) {$\pi_\calL(\calA_1) = \{0,1\}$};
\node[anchor=west] at (27,2.5) {$\pi_\calL(\calA_2) = \{3\}$};
\node[anchor=west] at (27,-5) {$\pi_\calL(\calA_3) = \{0,1,2,3\}$};
\node[anchor=west] at (27,6.5) {$\pi_\calL^2(\calA_1) = (\{0,1\},\emptyset)$};
\node[anchor=west] at (27,0.5) {$\pi_\calL^2(\calA_2) = (\{3\},\{0,1,2,3\})$};
\node[anchor=west] at (27,-7) {$\pi_\calL^2(\calA_3) = (\{0,1,2,3\},\{3\})$};

\node[anchor=north] at (22.5,-13) {\small
\begin{tabular}{|c|c|c|c|c|c|c|}
\multicolumn{2}{c}{} & \multicolumn{3}{c}{$\pi_\beta(i,\diamond,\vartheta)$} & \multicolumn{2}{c}{} \\
\cline{3-5}
\multicolumn{2}{c|}{} & \multicolumn{3}{|c|}{$i$} & \multicolumn{2}{|c}{} \\
\cline{3-5}
\multicolumn{2}{c|}{} & 1 & 2 & 3 & \multicolumn{2}{|c}{} \\
\hline
\multirow{4}{*}{$\diamond$} & $-$ & $\{0\}$ & $\{0,1\}$ & $\{3\}$ & \multirow{2}{*}{$\uparrow$} & \multirow{4}{*}{$\vartheta$} \\
\cline{2-5}
& $+$ & $\{0,1\}$ & $\{3\}$ & $\{3\}$ & & \\
\cline{2-6}
& $-$ & $\{0\}$ & $\{2\}$ & $\{2\}$ & \multirow{2}{*}{$\downarrow$} & \\
\cline{2-5}
& $+$ & $\{0,1\}$ & $\{2\}$ & $\{0,1,2,3\}$ & & \\
\hline
\multicolumn{7}{c}{} \\
\multicolumn{2}{c}{} & \multicolumn{3}{c}{$\pi_\beta^2(i,\diamond,\vartheta)$} & \multicolumn{2}{c}{} \\
\cline{3-5}
\multicolumn{2}{c|}{} & \multicolumn{3}{|c|}{$i$} & \multicolumn{2}{|c}{} \\
\cline{3-5}
\multicolumn{2}{c|}{} & 1 & 2 & 3 & \multicolumn{2}{|c}{} \\
\hline
\multirow{4}{*}{$\diamond$} & $-$ & $(\{0\},\emptyset)$ & $(\{0,1\},\emptyset)$ & $(\{3\},\{0,1,2,3\})$ & \multirow{2}{*}{$\uparrow$} & \multirow{4}{*}{$\vartheta$} \\
\cline{2-5}
& $+$ & $(\{0,1\},\emptyset)$ & $(\{3\},\{0,1,2,3\})$ & $(\{3\},\{0,1,2,3\})$ & & \\
\cline{2-6}
& $-$ & $(\{0\},\emptyset)$ & $(\{2\},\emptyset)$ & $(\{2\},\emptyset)$ & \multirow{2}{*}{$\downarrow$} & \\
\cline{2-5}
& $+$ & $(\{0,1\},\emptyset)$ & $(\{2\},\emptyset)$ & $(\{0,1,2,3\},\{3\})$ & & \\
\hline
\end{tabular}
};
\end{tikzpicture}
\end{center}
\caption{A tight lamination and its (extended) shadow}
\label{fig:reduced-diagram}
\end{figure}

Figure~\ref{fig:reduced-diagram} represents the tight lamination
$\calL$ associated with the braid $\beta = \sigma_2^{-1}$.
The shadows and the extended shadows of three arcs $\calA_1$, $\calA_2$ and $\calA_3$ are indicated on the right.
The shadow and the extended shadow of $\beta$ are indicated at the bottom part of the picture.
For instance, it is indicated that $\pi^2_\beta(3,+,\downarrow) = (\{0,1,2,3\},\{3\})$.
Indeed, the rightmost index of $\calL$ is $3$, the shadow of $\lowA_+(p_3)$ in $\calL$ is $\{0,1,2,3\}$,
and $\lowA_+(p_3)$ shares its right endpoint with an arc of shadow $\{3\}$ in $\calL$.

The shadow $\pi_\beta$ of a braid $\beta$ might be a suitable candidate
for being the succinct amount of information $\iota(\beta)$ mentioned at the beginning of section~\ref{subsection:sketch}.
Indeed, the set $\{\iota(\beta) : \beta \in B_n\}$ is clearly finite,
and we prove below that, for all braids $\beta \in B_n$ and all sliding braids $\lambda$,
knowing $\pi_\beta$ and $\lambda$ is sufficient to determine whether the relation
$\lambda^{-1} = \bR(\beta \lambda)$ holds.

We first show that not all sliding braids may appear in the relaxation normal form.

\begin{lem}
For all non-trivial braids $\beta \in B_n$, the sliding braid $\bR(\beta)$ belongs to the set
of \emph{left-oriented sliding braids}, which defined as the set
\[\{[k \curvearrowleft \ell] : 1 \leq k < \ell \leq n\} \cup \{[k \curvearrowbotleft \ell] : 1 \leq k < \ell \leq n\}.\]
\label{lem:right-oriented}
\end{lem}

A direct consequence of Lemma~\ref{lem:right-oriented} is that
the relation $\lambda^{-1} = \bR(\beta \lambda)$ may hold only if $\lambda$ is a \emph{right-oriented sliding braid},
i.e. an element of the set
\[\{[k \curvearrowright \ell] : 1 \leq k < \ell \leq n\} \cup \{[k \curvearrowbotright \ell] : 1 \leq k < \ell \leq n\}.\]
Moreover, for all right-oriented sliding braids $\lambda$,
there exists a simple criterion for deciding whether $\lambda^{-1} = \bR(\beta \lambda)$.

\begin{pro}
Let $\beta \in B_n$ be some braid, and let $k$ and $\ell$ be integers such that $0 < k < \ell \leq n$.
The equality $\bR(\beta [k \curvearrowright \ell]) = [k \curvearrowleft \ell]$
holds if and only if all of the following conditions are fulfilled:
\begin{enumerate}
\item $\pi_\beta(k,+,\downarrow) \neq \{k\}$;
\item either $\pi_\beta(k,+,\uparrow) = \{0,\ldots,k\}$ or $\pi_\beta(k,-,\uparrow) \subseteq \{k,\ldots,\ell-1\}$;
\item for all $i \in \{\ell+2,\ldots,n\}$, $\ell+1 \in \pi_\beta(i,-,\uparrow) \cap \pi_\beta(i,-,\downarrow)$;
\item if $\ell < n$, then $k \in \pi_\beta(\ell+1,+,\uparrow)$;
\item if $\ell < n$, then either $\pi_\beta(\ell+1,+,\downarrow) \neq \{\ell+1\}$ or
$\pi_\beta(\ell+1,-,\uparrow) \subseteq \{k+1,\ldots,\ell\}$.
\end{enumerate}
Analogous conditions, where upper arcs are replaced by lower arcs and vice-versa, characterize whether the equality
$\bR(\beta [k \curvearrowbotright \ell]) = [k \curvearrowbotleft \ell]$ holds.
\label{pro:caracterisation-extension}
\end{pro}

These conditions are illustrated in Figure~\ref{fig:caracterisation-extension}.

\begin{figure}[!ht]
\begin{center}
\begin{tikzpicture}[scale=0.27]
\draw[draw=black,ultra thick] (-7,-30.5) -- (-7,4) -- (25,4) -- (25,-30.5) -- cycle;
\draw[draw=black,ultra thick] (-7,-1.5) -- (25,-1.5);
\draw[draw=black,ultra thick] (-7,-8.5) -- (25,-8.5);
\draw[draw=black,ultra thick] (-7,-17.5) -- (25,-17.5);
\draw[draw=black,ultra thick] (-7,-24.5) -- (25,-24.5);

\draw[draw=black,ultra thick] (-7,4) -- (-7,2.5) -- (-5.5,2.5) -- (-5.5,4) -- cycle;
\node at (-6.25,3.25) {1.};
\draw[draw=black,ultra thick] (-7,-1.5) -- (-7,-3) -- (-5.5,-3) -- (-5.5,-1.5) -- cycle;
\node at (-6.25,-2.25) {2.};
\draw[draw=black,ultra thick] (-7,-10) -- (-7,-8.5) -- (-5.5,-8.5) -- (-5.5,-10) -- cycle;
\node at (-6.25,-9.25) {3.};
\draw[draw=black,ultra thick] (-7,-17.5) -- (-7,-19) -- (-5.5,-19) -- (-5.5,-17.5) -- cycle;
\node at (-6.25,-18.25) {4.};
\draw[draw=black,ultra thick] (-7,-24.5) -- (-7,-26) -- (-5.5,-26) -- (-5.5,-24.5) -- cycle;
\node at (-6.25,-25.25) {5.};

\node at (5,1.5) {\Huge$\neg$};

\draw[draw=black] (7,1.5) -- (13,1.5);
\draw[fill=white,draw=black,thick] (10,1.5) circle (1.5);
\node at (10,1.5) {\small$p_k$};
\draw[draw=black,ultra thick] (8,1.5) arc (180:360:2);

\draw[draw=black] (-5.5,-6) -- (8.5,-6);
\draw[fill=white,draw=black,thick] (-1.5,-6) circle (1.5);
\node at (-1.5,-6) {\small$p_0$};
\draw[fill=white,draw=black,thick] (5.5,-6) circle (1.5);
\node at (5.5,-6) {\small$p_k$};
\draw[draw=black,ultra thick] (-4.5,-6) arc (180:90:3.5) -- (4,-2.5) arc (90:0:3.5);

\node at (10.5,-6) {\Huge$\vee$};

\draw[draw=black] (12.5,-6) -- (23.5,-6);
\draw[fill=white,draw=black,thick] (15.5,-6) circle (1.5);
\node at (15.5,-6) {\small$p_k$};
\draw[fill=white,draw=black,thick] (21.5,-6) circle (1.5);
\node at (21.5,-6) {\small$p_\ell$};
\draw[draw=black,ultra thick] (13.5,-6) arc (180:0:2.5);

\draw[draw=black] (-2.5,-13) -- (20.5,-13);
\draw[fill=white,draw=black,thick] (-0.5,-13) circle (1.5);
\node at (-0.5,-13) {\small$p_{\ell+1}$};
\draw[fill=white,draw=black,thick] (4.5,-13) circle (1.5);
\node at (4.5,-13) {\small$p_{\ell+2}$};
\draw[fill=white,draw=black,thick] (9.5,-13) circle (1.5);
\node at (9.5,-13) {\small$p_{\ell+3}$};
\draw[fill=white,draw=black,thick] (18.5,-13) circle (1.5);
\node at (18.5,-13) {\small$p_n$};
\node[anchor=south] at (13.75,-13.25) {$\cdots$};
\draw[draw=black,ultra thick] (-2.5,-11) -- (0.5,-11) arc (90:-90:2) -- (-2.5,-15);
\draw[draw=black,ultra thick] (-2.5,-10.5) -- (5,-10.5) arc (90:-90:2.5) -- (-2.5,-15.5);
\draw[draw=black,ultra thick] (-2.5,-9.5) -- (13,-9.5) arc (90:-90:3.5) -- (-2.5,-16.5);

\draw[draw=black] (1,-22) -- (17,-22);
\draw[fill=white,draw=black,thick] (5,-22) circle (1.5);
\node at (5,-22) {\small$p_k$};
\draw[fill=white,draw=black,thick] (14,-22) circle (1.5);
\node at (14,-22) {\small$p_{\ell+1}$};
\draw[draw=black,ultra thick] (2,-22) arc (180:90:3.5) -- (12.5,-18.5) arc (90:0:3.5);

\draw[draw=black] (-3.5,-27.5) -- (2.5,-27.5);
\draw[fill=white,draw=black,thick] (-0.5,-27.5) circle (1.5);
\node at (-0.5,-27.5) {\small$p_{\ell+1}$};
\draw[draw=black,ultra thick] (-2.5,-27.5) arc (180:360:2);

\node at (5.5,-27.5) {\Huge$\Rightarrow$};

\draw[draw=black] (8.5,-27.5) -- (21.5,-27.5);
\draw[fill=white,draw=black,thick] (10.5,-27.5) circle (1.5);
\node at (10.5,-27.5) {\small$p_k$};
\draw[fill=white,draw=black,thick] (19.5,-27.5) circle (1.5);
\node at (19.5,-27.5) {\small$p_{\ell+1}$};
\draw[draw=black,ultra thick] (13.5,-27.5) arc (180:0:2);
\end{tikzpicture}
\end{center}
\caption{Conditions in Proposition~\ref{pro:caracterisation-extension}}
\label{fig:caracterisation-extension}
\end{figure}
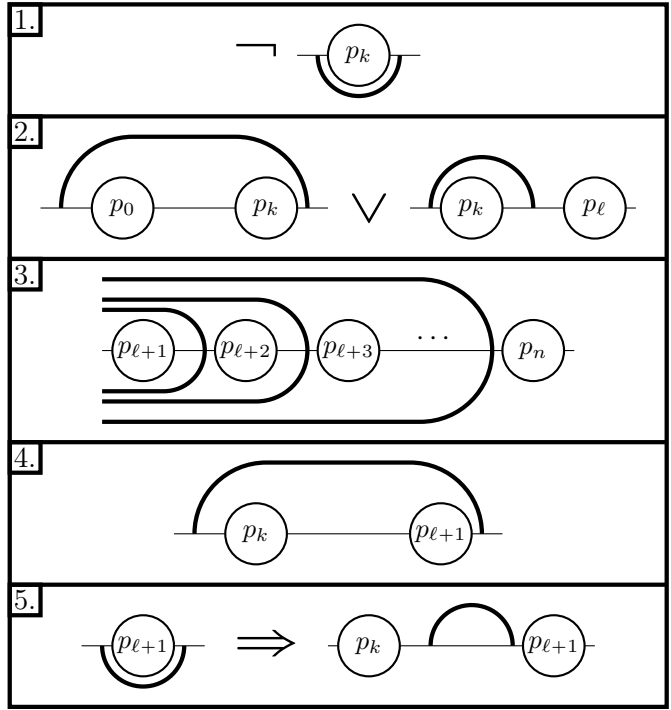

Requiring that five different conditions hold simultaneously may seem a heavy burden
in view of having to prove Proposition~\ref{pro:caracterisation-extension}.
Yet, considering only any four of these conditions would not be sufficient
to ensure the equality $\bR(\beta [k \curvearrowright \ell]) = [k \curvearrowleft \ell]$.
This is illustrated by Figure~\ref{fig:counter-example},
which presents tight laminations of five braids $\beta_1,\ldots,\beta_5$, such that, for all $i \in \{1,2,3,4,5\}$,
only the condition $(i)$ prevents the equality
$\bR(\beta_i [1 \curvearrowright 3]) = [1 \curvearrowleft 3]$ from being true.

\def\arraystretch{0.7}
\begin{figure}[!ht]
\begin{center}
\begin{tikzpicture}[scale=0.0899275]
\draw[draw=gray] (0,0) -- (23,0);

\draw[draw=gray,thick] (3,0) arc (180:-180:2);
\draw[draw=gray,thick] (2,0) arc (180:0:3) arc (180:360:4.5) arc (0:180:2) arc (360:180:2) arc (180:0:4.5) arc (360:180:8);
\draw[draw=gray,thick] (1,0) arc (180:-180:9);
\draw[draw=gray,thick] (0,0) arc (180:-180:11.5);

\draw[fill=gray,draw=gray,thick] (5,0) circle (1);
\draw[fill=white,draw=gray,thick] (11,0) circle (1);
\draw[fill=white,draw=gray,thick] (15,0) circle (1);
\draw[fill=white,draw=gray,thick] (21,0) circle (1);

\node[anchor=south] at (11.5,11.5) {$\sss\beta_1 = \sigma_1^{-2}$};
\node[anchor=north] at (11.5,-11.5) {\begin{tabular}{c}$\sss\bR(\beta_1[1 \curvearrowright 3])$\\$\sss=$\\$\sss[2 \curvearrowleft 3]$\end{tabular}};
\end{tikzpicture}
{\tiny~}
\begin{tikzpicture}[scale=0.0766049]
\draw[draw=gray] (0,0) -- (27,0);

\draw[draw=gray,thick] (3,0) arc (180:-180:2);
\draw[draw=gray,thick] (2,0) arc (180:0:3) arc (180:360:8) arc (0:180:3) arc (360:180:2) arc (180:0:5.5) arc (360:180:11.5);
\draw[draw=gray,thick] (1,0) arc (180:0:4) arc (180:360:7) arc (0:180:2) arc (360:180:4.5) arc (180:0:8) arc (360:180:12.5);
\draw[draw=gray,thick] (0,0) arc (180:-180:13.5);

\draw[fill=gray,draw=gray,thick] (5,0) circle (1);
\draw[fill=white,draw=gray,thick] (12,0) circle (1);
\draw[fill=white,draw=gray,thick] (16,0) circle (1);
\draw[fill=white,draw=gray,thick] (21,0) circle (1);

\node[anchor=south] at (13.5,13.5) {$\sss\beta_2 = \Delta_3^{-1} \sigma_1^{-1} \sigma_2^{-1}$};
\node[anchor=north] at (13.5,-13.5) {\begin{tabular}{c}$\sss\bR(\beta_2[1 \curvearrowright 3])$\\$\sss=$\\$\sss[1 \curvearrowbotleft 2]$\end{tabular}};
\end{tikzpicture}
{\tiny~}
\begin{tikzpicture}[scale=0.0590952]
\draw[draw=gray] (0,0) -- (35,0);

\draw[draw=gray,thick] (5,0) arc (180:-180:2);
\draw[draw=gray,thick] (4,0) arc (180:-180:4.5);
\draw[draw=gray,thick] (3,0) arc (180:-180:7);
\draw[draw=gray,thick] (2,0) arc (180:0:13) arc (360:180:2) arc (0:180:3) arc (360:180:8);
\draw[draw=gray,thick] (1,0) arc (180:0:14) arc (360:180:3) arc (0:180:2) arc (180:360:5.5) arc (180:2:2) arc (360:180:16.5);
\draw[draw=gray,thick] (0,0) arc (180:-180:17.5);

\draw[fill=gray,draw=gray,thick] (7,0) circle (1);
\draw[fill=white,draw=gray,thick] (11,0) circle (1);
\draw[fill=white,draw=gray,thick] (15,0) circle (1);
\draw[fill=white,draw=gray,thick] (21,0) circle (1);
\draw[fill=white,draw=gray,thick] (26,0) circle (1);
\draw[fill=white,draw=gray,thick] (32,0) circle (1);

\node[anchor=south] at (17.5,17.5) {$\sss\beta_3 = \sigma_4^{-1} \sigma_3$};
\node[anchor=north] at (17.5,-17.5) {\begin{tabular}{c}$\sss\bR(\beta_3[1 \curvearrowright 3])$\\$\sss=$\\$\sss[2 \curvearrowbotleft 5]$\end{tabular}};
\end{tikzpicture}
{\tiny~}
\begin{tikzpicture}[scale=0.0795513]
\draw[draw=gray] (0,0) -- (26,0);

\draw[draw=gray,thick] (4,0) arc (180:-180:2);
\draw[draw=gray,thick] (3,0) arc (180:-180:4.5);
\draw[draw=gray,thick] (2,0) arc (180:-180:7);
\draw[draw=gray,thick] (1,0) arc (180:0:8) arc (180:360:2) arc (180:0:2) arc (360:180:12);
\draw[draw=gray,thick] (0,0) arc (180:-180:13);

\draw[fill=gray,draw=gray,thick] (6,0) circle (1);
\draw[fill=white,draw=gray,thick] (10,0) circle (1);
\draw[fill=white,draw=gray,thick] (14,0) circle (1);
\draw[fill=white,draw=gray,thick] (19,0) circle (1);
\draw[fill=white,draw=gray,thick] (23,0) circle (1);

\node[anchor=south] at (13,13) {$\sss\beta_4 = \sigma_3^{-1}$};
\node[anchor=north] at (13,-13) {\begin{tabular}{c}$\sss\bR(\beta_4[1 \curvearrowright 3])$\\$\sss=$\\$\sss[2 \curvearrowbotleft 4]$\end{tabular}};
\end{tikzpicture}
{\tiny~}
\begin{tikzpicture}[scale=0.0689444]
\draw[draw=gray] (0,0) -- (30,0);

\draw[draw=gray,thick] (4,0) arc (180:-180:2);
\draw[draw=gray,thick] (3,0) arc (180:0:12) arc (360:180:2) arc (0:180:7) arc (360:180:3);
\draw[draw=gray,thick] (2,0) arc (180:0:13) arc (360:180:4.5) arc (0:180:4.5) arc (360:180:4);
\draw[draw=gray,thick] (1,0) arc (180:0:14) arc (360:180:7) arc (0:180:2) arc (360:180:5);
\draw[draw=gray,thick] (0,0) arc (180:-180:15);

\draw[fill=gray,draw=gray,thick] (6,0) circle (1);
\draw[fill=white,draw=gray,thick] (13,0) circle (1);
\draw[fill=white,draw=gray,thick] (17,0) circle (1);
\draw[fill=white,draw=gray,thick] (21,0) circle (1);
\draw[fill=white,draw=gray,thick] (25,0) circle (1);

\node[anchor=south] at (15,15) {$\sss\beta_5 = \Delta_4$};
\node[anchor=north] at (15,-15) {\begin{tabular}{c}$\sss\bR(\beta_5[1 \curvearrowright 3])$\\$\sss=$\\$\sss[1 \curvearrowleft 4]$\end{tabular}};
\end{tikzpicture}
\end{center}
\caption{Falsifying conditions of Proposition~\ref{pro:caracterisation-extension} one by one}
\label{fig:counter-example}
\end{figure}
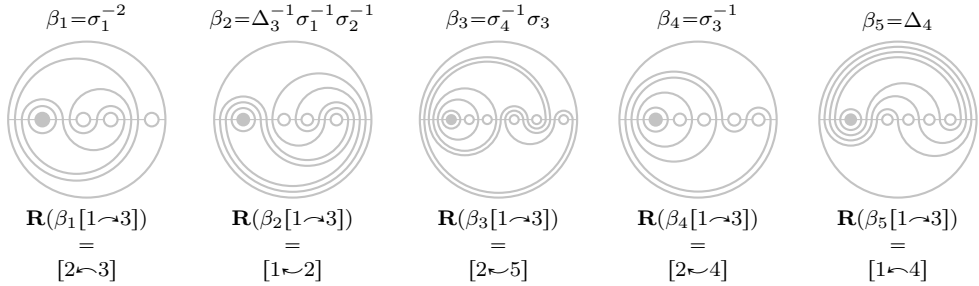

Since knowing $\pi_\beta$ is sufficient
to decide whether the equality $\lambda^{-1} = \bR(\beta \lambda)$ holds for a given
sliding braid $\lambda$, it is tempting to imagine that $\pi_\beta$ is the required
amount of information $\iota(\beta)$ mentioned above.
This is \emph{not} the case. Indeed, a second requirement about $\iota(\beta)$ is
that, when $\lambda^{-1} = \bR(\beta \lambda)$, knowing $\iota(\beta)$ and $\lambda$
is sufficient to compute $\iota(\beta \lambda)$.
However, knowing $\pi_\beta$ and $\lambda$ is not
sufficient to compute $\pi_{\beta \lambda}$, and the following example supports an even stronger statement.

{\def\arraystretch{0.7}
\begin{figure}[!ht]
\begin{center}
\begin{tikzpicture}[scale=0.13]
\draw[draw=gray] (0,0) -- (44,0);

\draw[draw=gray,thick] (4,0) arc (180:-180:2);
\draw[draw=gray,thick] (3,0) arc (180:360:3) arc (180:0:15.5) arc (360:180:11) arc (0:180:2) arc (180:360:13.5) arc (0:180:19);
\draw[draw=gray,thick] (2,0) arc (180:360:4) arc (180:0:14.5) arc (360:180:10) arc (180:0:8) arc (360:180:3) arc (0:180:2) arc (180:360:5.5) arc (0:180:11.5) arc (180:360:14.5) arc (0:180:20);
\draw[draw=gray,thick] (1,0) arc (180:360:5) arc (180:0:13.5) arc (360:180:9) arc (180:0:7) arc (360:180:2) arc (0:180:4.5) arc (180:360:8) arc (0:180:12.5) arc (180:360:15.5) arc (0:180:21);
\draw[draw=gray,thick] (0,0) arc (180:-180:22);

\draw[fill=gray,draw=gray,thick] (6,0) circle (1);
\draw[fill=white,draw=gray,thick] (16,0) circle (1);
\draw[fill=white,draw=gray,thick] (23,0) circle (1);
\draw[fill=white,draw=gray,thick] (27,0) circle (1);
\draw[fill=white,draw=gray,thick] (32,0) circle (1);

\node[anchor=south] at (22,22) {$\alpha$};
\end{tikzpicture}
~~~~
\begin{tikzpicture}[scale=0.11]
\draw[draw=gray] (0,0) -- (52,0);

\draw[draw=gray,thick] (4,0) arc (180:-180:2);
\draw[draw=gray,thick] (3,0) arc (180:360:3) arc (180:0:19.5) arc (360:180:15) arc (0:180:2) arc (180:360:17.5) arc (0:180:23);
\draw[draw=gray,thick] (2,0) arc (180:360:4) arc (180:0:18.5) arc (360:180:14) arc (180:0:12) arc (360:180:10) arc (180:0:8) arc (360:180:3) arc (0:180:2) arc (180:360:5.5) arc (0:180:9) arc (180:360:11) arc (0:180:15.5) arc (180:360:18.5) arc (0:180:24);
\draw[draw=gray,thick] (1,0) arc (180:360:5) arc (180:0:17.5) arc (360:180:13) arc (180:0:11) arc (360:180:9) arc (180:0:7) arc (360:180:2) arc (0:180:4.5) arc (180:360:8) arc (0:180:10) arc (180:360:12) arc (0:180:16.5) arc (180:360:19.5) arc (0:180:25);
\draw[draw=gray,thick] (0,0) arc (180:-180:26);

\draw[fill=gray,draw=gray,thick] (6,0) circle (1);
\draw[fill=white,draw=gray,thick] (16,0) circle (1);
\draw[fill=white,draw=gray,thick] (27,0) circle (1);
\draw[fill=white,draw=gray,thick] (31,0) circle (1);
\draw[fill=white,draw=gray,thick] (36,0) circle (1);

\node[anchor=south] at (26,26) {$\alpha'$};
\end{tikzpicture}

~\smallskip

{\def\arraystretch{1}\small
\begin{tabular}{|c|c|c|c|c|c|c|c|}
\multicolumn{2}{c}{} & \multicolumn{4}{c}{$\pi_\alpha(i,\diamond,\vartheta) = \pi_{\alpha'}(i,\diamond,\vartheta)$} & \multicolumn{2}{c}{} \\
\cline{3-6}
\multicolumn{2}{c|}{} & \multicolumn{4}{|c|}{$i$} & \multicolumn{2}{|c}{} \\
\cline{3-6}
\multicolumn{2}{c|}{} & 1 & 2 & 3 & 4 & \multicolumn{2}{|c}{} \\
\hline
\multirow{4}{*}{$\diamond$} & $-$ & $\{1\}$ & $\{2,3\}$ & $\{3\}$ & $\{3\}$ & \multirow{2}{*}{$\uparrow$} & \multirow{4}{*}{$\vartheta$} \\
\cline{2-6}
& $+$ & $\{1\}$ & $\{3\}$ & $\{3\}$ & $\{2,3,4\}$ & & \\
\cline{2-7}
& $-$ & $\{1,2,3,4\}$ & $\{2,3,4\}$ & $\{3,4\}$ & $\{4\}$ & \multirow{2}{*}{$\downarrow$} & \\
\cline{2-6}
& $+$ & $\{2,3,4\}$ & $\{3,4\}$ & $\{4\}$ & $\{4\}$ & & \\
\hline
\end{tabular}
}
\end{center}
\caption{Choosing $\iota(\beta) = \pi_\beta$ is not enough}
\label{fig:alpha-beta}
\end{figure}
}

The words $\ba = [3 \curvearrowright 4] \cdot [2 \curvearrowright 4] \cdot [1 \curvearrowright 4]^4 \cdot [2 \curvearrowright 4]$ and
$\ba' = \ba \cdot [2 \curvearrowright 4]^3$ are the respective relaxation normal forms of the braids 
$\alpha = \Delta_4^2 \sigma_3 \sigma_2 \sigma_3 \sigma_2 \sigma_3$ and
$\alpha' = \alpha (\sigma_2 \sigma_3)^3$.
The braids $\alpha$ and $\alpha'$ have the same shadow, i.e. $\pi_\alpha = \pi_{\alpha'}$, as illustrated by Figure~\ref{fig:alpha-beta}.
Hence, if $\pi_\beta$ were an adequate candidate for the amount of information $\iota(\beta)$,
the sets of braids $\{\gamma : \ba \text{ is a prefix of } \RNF(\alpha \gamma)\}$ and
$\{\gamma : \ba' \text{ is a prefix of } \RNF(\alpha' \gamma)\}$ would be equal to each other.
Yet, the braid $\gamma = \sigma_1 \sigma_2 \sigma_3 \sigma_1^{-1} \sigma_2^{-1} \sigma_3^{-1}$ belongs only to the latter set,
as shown the relations $\RNF(\alpha \gamma) = [3 \curvearrowright 4] \cdot [2 \curvearrowright 4] \cdot
[1 \curvearrowright 4]^3 \cdot [1 \curvearrowright 3] \cdot [2 \curvearrowright 3] \cdot [1 \curvearrowright 3]$ and
$\RNF(\alpha' \gamma) = \ba' \cdot [1 \curvearrowright 4] \cdot [1 \curvearrowbotright 4]$.

It is therefore legitimate to look for an amount of information $\iota(\beta)$ that would be richer than the only shadow $\pi_\beta$.
Choosing the extended shadow $\pi^2_\beta$ to be amount of information $\iota(\beta)$ that we shall memorise about $\beta$
is an adequate choice.

\begin{pro}
Let $\Sigma$ be the set of all right-oriented sliding braids
$[k \curvearrowright \ell]$ or $[k \curvearrowbotright \ell]$, where
$1 \leq k < \ell \leq n$.
There exists two functions $\dec$ and $\comp$ that take
as inputs the extended shadow $\pi^2_\beta$ of some braid $\beta$ and
a braid $\lambda \in \Sigma$, and such that
\begin{enumerate}
\item $\dec(\pi^2_\beta,\lambda) = \mathbf{true}$ if $\RNF(\beta) \cdot \lambda = \RNF(\beta \lambda)$, and $\mathbf{false}$ otherwise;
\item $\comp(\pi^2_\beta,\lambda) = \pi^2_{\beta \lambda}$ if $\RNF(\beta) \cdot \lambda = \RNF(\beta \lambda)$.
\end{enumerate}
\label{pro:automata-state}
\end{pro}

From Proposition~\ref{pro:automata-state} follow immediately the main results of this paper.

\begin{thm}
Let $\calA = (\Sigma, Q, i, \delta, Q)$ be a deterministic automaton, with
\begin{itemize}
\item alphabet $\Sigma = \{[k \curvearrowright \ell]: 1 \leq k < \ell \leq n\} \cup
\{[k \curvearrowbotright \ell]: 1 \leq k < \ell \leq n\}$;
\item state set $Q = \{\pi^2_\beta: \beta \in B_n\}$;
\item initial state $i = \pi^2_\varepsilon$;
\item transition function $\delta = \{(\pi^2_\beta,\lambda,\pi^2_{\beta \lambda}) : \bR(\beta \lambda) = \lambda^{-1}\}$;
\item set of accepting states $Q$.
\end{itemize}
The automaton $\calA$ accepts exactly the language of all relaxation normal words.
\label{thm:here-is-the-automaton}
\end{thm}

\begin{thm}
The relaxation normal form is regular.
\label{thm:master-thm}
\end{thm}

Figure~\ref{fig:automaton} presents the minimal automaton accepting the language $\RNF(B_3)$.

This minimal automaton is obtained by merging states of the above-defined automaton $\calA$.
Each state $\mathbf{s}$ of the minimal automaton is a subset of $Q$, and is represented in Figure~\ref{fig:automaton}
by some braid $\beta$ such that $\pi^2_\beta \in \mathbf{s}$.
The initial state is the state $\{\pi^2_\varepsilon\}$, and each state is accepting.
Moreover, for the sake of readability of Figure~\ref{fig:automaton},
we chose to denote by $\ol{\beta}$ the braid $\beta^{-1}$.

\begin{figure}[!ht]
\begin{center}
\begin{tikzpicture}[scale=0.74,rotate=90]
\draw[draw=black,fill=white,very thick] (-11,-7) -- (-13,-7) -- (-13,-9) -- (-11,-9) -- cycle;
\node at (-12,-8) {$\sigma_1$};
\draw[draw=black,fill=white,very thick] (-11,-3) -- (-13,-3) -- (-13,-5) -- (-11,-5) -- cycle;
\node at (-12,-4) {$\Delta\sigma_1\sigma_2^2$};
\draw[draw=black,fill=white,very thick] (-7,-3) -- (-9,-3) -- (-9,-5) -- (-7,-5) -- cycle;
\node at (-8,-4) {$\Delta\sigma_1\sigma_2$};
\draw[draw=black,fill=white,very thick] (-11,1) -- (-13,1) -- (-13,-1) -- (-11,-1) -- cycle;
\node at (-12,0) {$\Delta$};
\draw[draw=black,fill=white,very thick] (-11,5) -- (-13,5) -- (-13,3) -- (-11,3) -- cycle;
\node at (-12,4) {$\sigma_2$};
\draw[draw=black,fill=white,very thick] (-11,7) -- (-13,7) -- (-13,9) -- (-11,9) -- cycle;
\node at (-12,8) {$\sigma_2^2$};
\draw[draw=black,fill=white,very thick] (-7,1) -- (-9,1) -- (-9,-1) -- (-7,-1) -- cycle;
\node at (-8,0) {$\Delta\sigma_2$};
\draw[draw=black,fill=white,very thick] (-7,5) -- (-9,5) -- (-9,3) -- (-7,3) -- cycle;
\node at (-8,4) {$\Delta\sigma_2^2$};
\draw[draw=black,fill=white,very thick] (-3,1) -- (-5,1) -- (-5,-1) -- (-3,-1) -- cycle;
\node at (-4,0) {$\sigma_1\sigma_2$};
\draw[draw=black,fill=white,very thick] (2,-5) -- (-2,-5) -- (-2,-9) -- (2,-9) -- cycle;
\node at (0,-7) {$\ol{\sigma_1}\sigma_2$};

\draw[draw=black,fill=white,very thick] (11,7) -- (13,7) -- (13,9) -- (11,9) -- cycle;
\node at (12,8) {$\ol{\sigma_1}$};
\draw[draw=black,fill=white,very thick] (11,3) -- (13,3) -- (13,5) -- (11,5) -- cycle;
\node at (12,4) {$\ol{\Delta\sigma_1\sigma_2^2}$};
\draw[draw=black,fill=white,very thick] (7,3) -- (9,3) -- (9,5) -- (7,5) -- cycle;
\node at (8,4) {$\ol{\Delta\sigma_1\sigma_2}$};
\draw[draw=black,fill=white,very thick] (11,-1) -- (13,-1) -- (13,1) -- (11,1) -- cycle;
\node at (12,-0) {$\ol{\Delta}$};
\draw[draw=black,fill=white,very thick] (11,-5) -- (13,-5) -- (13,-3) -- (11,-3) -- cycle;
\node at (12,-4) {$\ol{\sigma_2}$};
\draw[draw=black,fill=white,very thick] (11,-7) -- (13,-7) -- (13,-9) -- (11,-9) -- cycle;
\node at (12,-8) {$\ol{\sigma_2}^2$};
\draw[draw=black,fill=white,very thick] (7,-1) -- (9,-1) -- (9,1) -- (7,1) -- cycle;
\node at (8,0) {$\ol{\Delta\sigma_2}$};
\draw[draw=black,fill=white,very thick] (7,-5) -- (9,-5) -- (9,-3) -- (7,-3) -- cycle;
\node at (8,-4) {$\ol{\Delta\sigma_2}^2$};
\draw[draw=black,fill=white,very thick] (3,-1) -- (5,-1) -- (5,1) -- (3,1) -- cycle;
\node at (4,-0) {$\ol{\sigma_1\sigma_2}$};
\draw[draw=black,fill=white,very thick] (-2,5) -- (2,5) -- (2,9) -- (-2,9) -- cycle;
\node at (0,7) {$\sigma_1\ol{\sigma_2}$};

\draw[draw=black,fill=white,very thick] (-1,-1) -- (1,-1) -- (1,1) -- (-1,1) -- cycle;
\node at (0,0) {$\varepsilon$};

\node[anchor=west] at (-2.5,0.1) {\scriptsize$\sigma_1\sigma_2$};
\node[anchor=west] at (-6.5,0.1) {\scriptsize$\sigma_1\sigma_2$};
\node[anchor=west] at (-10.5,0.1) {\scriptsize$\sigma_1\sigma_2$};
\node[anchor=east] at (-9.5,-2.5) {\scriptsize$\sigma_1\sigma_2$};
\node[anchor=south west] at (-8.1,-1.2) {\scriptsize$\sigma_1\sigma_2$};
\node[anchor=south east] at (-12.1,1.2) {\scriptsize$\sigma_1\sigma_2$};
\node[anchor=west] at (-3,5.6) {\scriptsize$\sigma_1\sigma_2$};
\node[anchor=east] at (-2.5,-9.5) {\scriptsize$\sigma_1\sigma_2$};
\node[anchor=south east] at (-0.6,-4.8) {\scriptsize$\sigma_1\sigma_2$};
\node[anchor=west] at (-3,7.1) {\scriptsize$\ol{\sigma_1\sigma_2}$};
\node[anchor=west] at (-2.5,4.5) {\scriptsize$\ol{\sigma_1\sigma_2}$};
\node[anchor=west] at (-3,8.6) {\scriptsize$\ol{\sigma_1\sigma_2}$,};
\node[anchor=west] at (-3.4,8.6) {\scriptsize$\sigma_1\sigma_2$};
\node[anchor=south west] at (-8.1,2.8) {\scriptsize$\sigma_2$};
\node[anchor=south west] at (-12.1,6.8) {\scriptsize$\sigma_2$};
\node[anchor=east] at (-10.5,2.5) {\scriptsize$\sigma_2$};
\node[anchor=east] at (-8,5.9) {\scriptsize$\sigma_2$};
\node[anchor=east] at (-12,-2.1) {\scriptsize$\sigma_2$};
\node[anchor=north] at (-13.9,8) {\scriptsize$\sigma_2$};
\node[anchor=west] at (-2.5,-6.4) {\scriptsize$\sigma_2$};
\node[anchor=west] at (-10.5,-3.9) {\scriptsize$\sigma_2$};
\node[anchor=east] at (-13.5,-7.6) {\scriptsize$\sigma_1$};
\node[anchor=west] at (-10.5,-8.4) {\scriptsize$\sigma_1$};
\node[anchor=west] at (-10.5,-7.4) {\scriptsize$\sigma_1$};
\node[anchor=west] at (-12,-9.9) {\scriptsize$\sigma_1$};
\node[anchor=east] at (-2.5,9.5) {\scriptsize$\ol{\sigma_2}$};
\node[anchor=east] at (-0,9.9) {\scriptsize$\ol{\sigma_2}$};

\node[anchor=west] at (2.5,0.1) {\scriptsize$\ol{\sigma_1\sigma_2}$};
\node[anchor=west] at (6.5,0.1) {\scriptsize$\ol{\sigma_1\sigma_2}$};
\node[anchor=west] at (10.5,0.1) {\scriptsize$\ol{\sigma_1\sigma_2}$};
\node[anchor=west] at (9.5,2.5) {\scriptsize$\ol{\sigma_1\sigma_2}$};
\node[anchor=south east] at (7.9,1.2) {\scriptsize$\ol{\sigma_1\sigma_2}$};
\node[anchor=south west] at (11.9,-1.2) {\scriptsize$\ol{\sigma_1\sigma_2}$};
\node[anchor=west] at (3,-5.4) {\scriptsize$\ol{\sigma_1\sigma_2}$};
\node[anchor=west] at (2.5,9.5) {\scriptsize$\ol{\sigma_1\sigma_2}$};
\node[anchor=north west] at (0.6,4.8) {\scriptsize$\ol{\sigma_1\sigma_2}$};
\node[anchor=west] at (3,-6.9) {\scriptsize$\sigma_1\sigma_2$};
\node[anchor=east] at (2.5,-4.5) {\scriptsize$\sigma_1\sigma_2$};
\node[anchor=west] at (3,-8.4) {\scriptsize$\sigma_1\sigma_2$};
\node[anchor=west] at (3.4,-8.4) {\scriptsize$\ol{\sigma_1\sigma_2}$,};
\node[anchor=south east] at (7.9,-2.8) {\scriptsize$\ol{\sigma_2}$};
\node[anchor=south east] at (11.9,-6.8) {\scriptsize$\ol{\sigma_2}$};
\node[anchor=east] at (10.5,-2.5) {\scriptsize$\ol{\sigma_2}$};
\node[anchor=west] at (8,-5.9) {\scriptsize$\ol{\sigma_2}$};
\node[anchor=west] at (12,2.1) {\scriptsize$\ol{\sigma_2}$};
\node[anchor=south] at (13.9,-8) {\scriptsize$\ol{\sigma_2}$};
\node[anchor=west] at (2.5,6.6) {\scriptsize$\ol{\sigma_2}$};
\node[anchor=west] at (10.5,4.1) {\scriptsize$\ol{\sigma_2}$};
\node[anchor=west] at (13.5,7.6) {\scriptsize$\ol{\sigma_1}$};
\node[anchor=west] at (10.5,8.6) {\scriptsize$\ol{\sigma_1}$};
\node[anchor=west] at (10.5,7.6) {\scriptsize$\ol{\sigma_1}$};
\node[anchor=east] at (12,9.9) {\scriptsize$\ol{\sigma_1}$};
\node[anchor=west] at (2.5,-9.5) {\scriptsize$\sigma_2$};
\node[anchor=west] at (0,-9.9) {\scriptsize$\sigma_2$};

\draw[->,>=stealth',draw=black,thick] (-7,-1) -- (-6,-2) -- (-6,-5.5) -- (-8,-7.5) -- (-11,-7.5);
\draw[->,>=stealth',draw=black,thick] (-13,4) -- (-14,4) -- (-14,-7.5) -- (-13,-7.5);
\draw[->,>=stealth',draw=black,thick] (-11,-9) -- (-10,-10) -- (-3,-10) -- (-2,-9);
\draw[->,>=stealth',draw=black,thick] (-13,-8.5) -- (-15,-8.5) -- (-15,10) -- (-3,10) -- (-2,9);
\draw[->,>=stealth',draw=black,thick] (-3.5,-1) -- (-3.5,-6.5) -- (-2,-6.5);
\draw[->,>=stealth',draw=black,thick] (-5,0) -- (-7,0);
\draw[->,>=stealth',draw=black,thick] (-3.5,1) -- (-3.5,3.5) -- (-2,5);
\draw[->,>=stealth',draw=black,thick] (-0.5,5) -- (-0.5,4) -- (-1.5,3) -- (-1.5,-3) -- (-0.5,-4) -- (-0.5,-5);
\draw[->,>=stealth',draw=black,thick] (0.5,-9) -- (0.5,-10) -- (-0.5,-10) -- (-0.5,-9);
\draw[->,>=stealth',draw=black,thick] (-12,3) -- (-12,1);
\draw[->,>=stealth',draw=black,thick] (-11,-1) -- (-9,-3);
\draw[->,>=stealth',draw=black,thick] (-8,-3) -- (-8,-1);
\draw[->,>=stealth',draw=black,thick] (-9,0) -- (-11,0);
\draw[->,>=stealth',draw=black,thick] (-9,-4) -- (-11,-4);
\draw[->,>=stealth',draw=black,thick] (-12.5,-3) -- (-12.5,-2) -- (-11.5,-2) -- (-11.5,-3);
\draw[->,>=stealth',draw=black,thick] (-8,1) -- (-8,3);
\draw[->,>=stealth',draw=black,thick] (-8.5,5) -- (-8.5,6) -- (-7.5,6) -- (-7.5,5);
\draw[->,>=stealth',draw=black,thick] (-7,4) -- (-6,4) -- (-4.5,5.5) -- (-2,5.5);
\draw[->,>=stealth',draw=black,thick] (-12,5) -- (-12,7);
\draw[->,>=stealth',draw=black,thick] (-11,8.5) -- (-2,8.5);
\draw[->,>=stealth',draw=black,thick] (-13,7.5) -- (-14,7.5) -- (-14,8.5) -- (-13,8.5);
\draw[->,>=stealth',draw=black,thick] (-11.5,-9) -- (-11.5,-10) -- (-12.5,-10) -- (-12.5,-9);
\draw[->,>=stealth',draw=black,thick] (-11,4.5) -- (-10,4.5) -- (-10,7) -- (-2,7);
\draw[->,>=stealth',draw=black,thick] (-1,0) -- (-3,0);
\draw[->,>=stealth',draw=black,thick] (-0.5,1) -- (-0.5,1.5) -- (-9.5,1.5) -- (-11,3);
\draw[->,>=stealth',draw=black,thick] (-0.5,-1) -- (-0.5,-1.5) -- (-7.5,-8.5) -- (-11,-8.5);

\draw[->,>=stealth',draw=black,thick] (7,1) -- (6,2) -- (6,5.5) -- (8,7.5) -- (11,7.5);
\draw[->,>=stealth',draw=black,thick] (13,-4) -- (14,-4) -- (14,7.5) -- (13,7.5);
\draw[->,>=stealth',draw=black,thick] (11,9) -- (10,10) -- (3,10) -- (2,9);
\draw[->,>=stealth',draw=black,thick] (13,8.5) -- (15,8.5) -- (15,-10) -- (3,-10) -- (2,-9);
\draw[->,>=stealth',draw=black,thick] (3.5,1) -- (3.5,6.5) -- (2,6.5);
\draw[->,>=stealth',draw=black,thick] (5,-0) -- (7,-0);
\draw[->,>=stealth',draw=black,thick] (3.5,-1) -- (3.5,-3.5) -- (2,-5);
\draw[->,>=stealth',draw=black,thick] (0.5,-5) -- (0.5,-4) -- (1.5,-3) -- (1.5,3) -- (0.5,4) -- (0.5,5);
\draw[->,>=stealth',draw=black,thick] (-0.5,9) -- (-0.5,10) -- (0.5,10) -- (0.5,9);
\draw[->,>=stealth',draw=black,thick] (12,-3) -- (12,-1);
\draw[->,>=stealth',draw=black,thick] (11,1) -- (9,3);
\draw[->,>=stealth',draw=black,thick] (8,3) -- (8,1);
\draw[->,>=stealth',draw=black,thick] (9,-0) -- (11,-0);
\draw[->,>=stealth',draw=black,thick] (9,4) -- (11,4);
\draw[->,>=stealth',draw=black,thick] (12.5,3) -- (12.5,2) -- (11.5,2) -- (11.5,3);
\draw[->,>=stealth',draw=black,thick] (8,-1) -- (8,-3);
\draw[->,>=stealth',draw=black,thick] (8.5,-5) -- (8.5,-6) -- (7.5,-6) -- (7.5,-5);
\draw[->,>=stealth',draw=black,thick] (7,-4) -- (6,-4) -- (4.5,-5.5) -- (2,-5.5);
\draw[->,>=stealth',draw=black,thick] (12,-5) -- (12,-7);
\draw[->,>=stealth',draw=black,thick] (11,-8.5) -- (2,-8.5);
\draw[->,>=stealth',draw=black,thick] (13,-7.5) -- (14,-7.5) -- (14,-8.5) -- (13,-8.5);
\draw[->,>=stealth',draw=black,thick] (11.5,9) -- (11.5,10) -- (12.5,10) -- (12.5,9);
\draw[->,>=stealth',draw=black,thick] (11,-4.5) -- (10,-4.5) -- (10,-7) -- (2,-7);
\draw[->,>=stealth',draw=black,thick] (1,-0) -- (3,-0);
\draw[->,>=stealth',draw=black,thick] (0.5,-1) -- (0.5,-1.5) -- (9.5,-1.5) -- (11,-3);
\draw[->,>=stealth',draw=black,thick] (0.5,1) -- (0.5,1.5) -- (7.5,8.5) -- (11,8.5);

\draw[->,>=stealth',draw=black,thick] (0,3) -- (0,1);
\end{tikzpicture}
\end{center}
\caption{Minimal automaton accepting the language $\RNF(B_3)$}
\label{fig:automaton}
\end{figure}
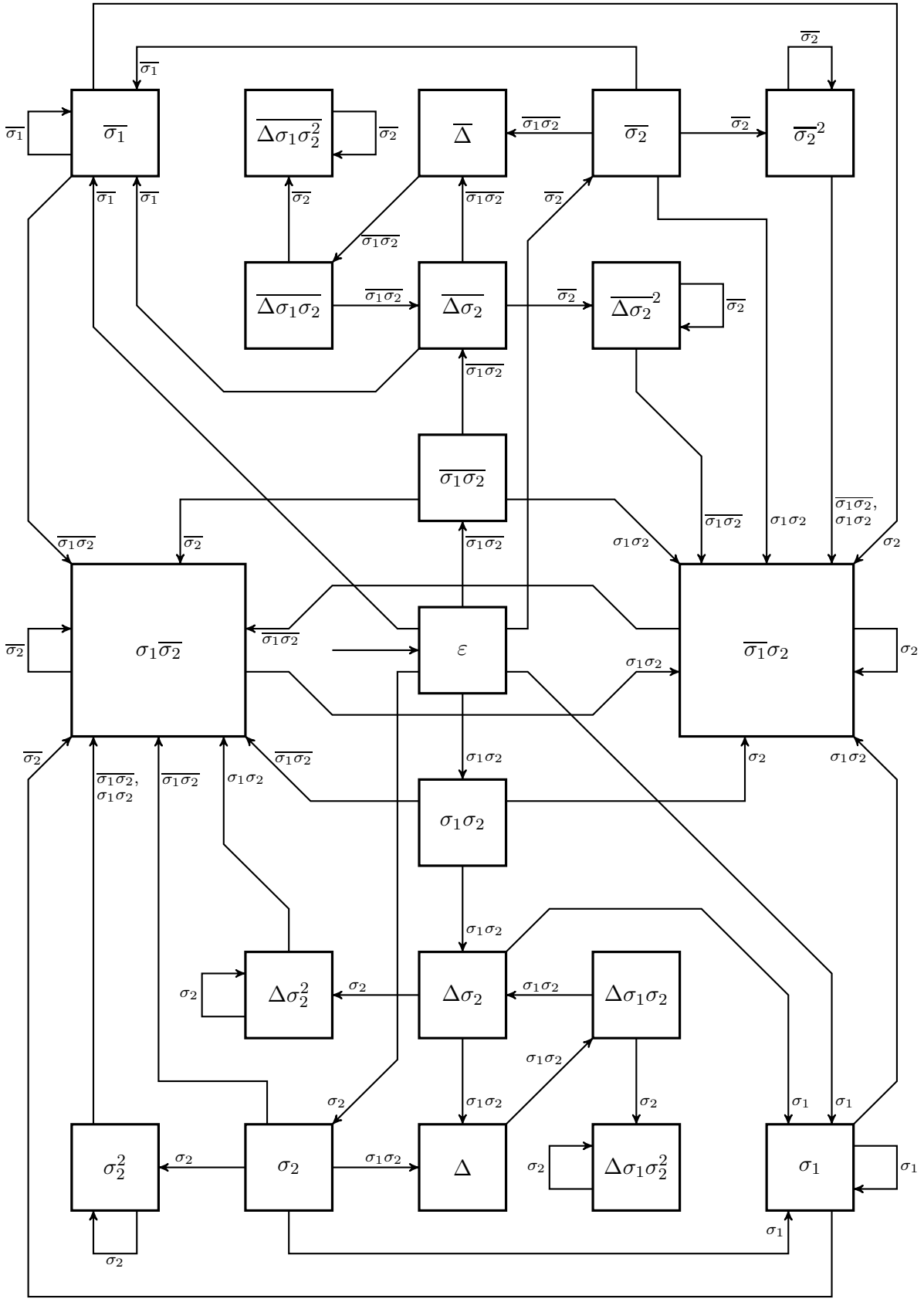

\pagebreak

\section{The Relaxation Normal Form is Regular: Rigorous Proofs}
\label{section:proofs}

Having exposed in section~\ref{section:braid-laminations} the main claims that led to the construction
of automata recognizing the relaxation normal form,
it remains to prove these claims.

\subsection{First Proofs}

As a warm-up, let us write full proofs of Lemmas~\ref{lem:easy-stuff} and~\ref{lem:right-oriented}.
The tools and ideas used in these proofs will occur in most subsequent proofs,
albeit in more complex sequences of arguments.

\begin{lem}
Let $\calL$ be a tight lamination.
The bigons of $\calL$ are the arcs $\calA$ of $\calL$, with endpoints $e_\calC < E_\calC$,
and such that the interval $(e_\calC,E_\calC)$ is minimal for the inclusion,
among all such intervals associated with arcs of $\calL$ lying in the same (upper or lower) half-plane as $\calA$.
\label{lem:small-is-bigon}
\end{lem}

\begin{proof}
First, it is clear that all bigons fall in the class of arcs described in Lemma~\ref{lem:small-is-bigon}.
Conversely, let $\calA$ be an arc of $\calL$ that is not a bigon.
Without loss of generality, we assume that $\calA$ is an upper arc.
There exists an endpoint $p$ of an arc of $\calL$ such that $p \in (e_\calC,E_\calC)$,
and there exists an upper arc $\calD$ of $\calL$, with endpoints
$e_\calD < E_\calD$, and such that $p \in \{e_\calD,E_\calD\}$.
Since $\calC$ and $\calD$ cannot cross each other, it follows that $e_\calC < e_\calD < E_\calD < E_\calC$.
\end{proof}

In particular, when $\calA$ is an arc of $\calL$ with endpoints $e_\calA < E_\calA$,
there exists a bigon $\calC$, with endpoints $e_\calC < E_\calC$ such that
(i) $(e_\calC,E_\calC)$ is a subset of $(e_\calA,E_\calA)$ and
(ii) $\calA$ and $\calC$ lie in the same (upper of lower) half-plane as each other.
Reusing the notions of Definition~\ref{dfn:arcs-bigons}, we say that such a bigon is \emph{covered} by $\calA$.
The notion of covered bigon leads directly to the following result.

\begin{cor}
A lamination $\calL$ is tight if and only if every arc of $\calL$ covers at least one puncture.
\label{cor:all-cover}
\end{cor}

\begin{proof}[Proof of Lemma~\ref{lem:easy-stuff}]
Let us first assume that $\calL$ is a tight lamination such that no mobile puncture of $\calL$ is covered by a bigon.
Let $\calA$ be an arc of $\calL$, and let $\calC$ be a bigon of $\calL$ covered by $\calA$.
Since $\calL$ is tight, $\calC$ must cover the fixed puncture $p_0 = -1$, hence $\calA$ also covers the puncture $p_0 = -1$,
i.e. $\calA$ touches the real interval $(-\infty,-1)$.
Every curve $\calL_i$ of $\calL$ contains only one lower arc and one upper arc
that touch the interval $(-\infty,-1)$, and we just proved that $\calL_i$ cannot contain other arcs.
It follows that $\calL$ is trivial.

From now on, we drop the assumption that no mobile puncture of $\calL$ may be covered by a bigon.
Let us assume that two distinct bigons $\calB_1$ and $\calB_2$ cover a puncture $p$ of $\calL$.
The bigons $\calB_1$ and $\calB_2$ must share their endpoints, hence the union $\calB_1 \cup \calB_2$
is one of the curves $\calL_i$. The puncture $p$ is the only puncture that lies within
the finite area delimited by $\calL_i$, hence $p = p_0$ and $\calL_i = \calL_0$.

Finally, let us assume that $p$ is a mobile puncture covered by a bigon $\calB$,
which is necessarily unique.
The arcs $\calA_1$ and $\calA_2$ must be distinct,
unless what the arc $\calA_1 = \calA_2$, sharing both endpoints of $\calB$, would itself be a bigon.
Moreover, the three arcs $\calA_1$, $\calA_2$ and $\calB$ belong to the same curve $\calL_j$,
and it is known that $\calL_j$ crosses exactly once the real interval $(-\infty,-1)$.
Consequently, at least one of the arcs $\calA_1$ and $\calA_2$ cannot cross the interval $(-\infty,-1)$,
and therefore cannot cover the fixed puncture $-1$.
\end{proof}

We continue by proving the following result, from which Lemma~\ref{lem:right-oriented} will follow.

\begin{lem}
Let $\calL$ be a non-trivial tight lamination and let $p_k$ be the rightmost puncture
covered by some bigon $\calB$ of the lamination $\calL$.
In addition, let $e_\calB < E_\calB$ be the endpoints of $\calB$.
For all arcs $\calA$ of $\calL$ with endpoints $e_\calA < E_\calA$,
we have $e_\calA \leq e_\calB$, with equality if and only if $\calA = \calB$.
\label{lem:left-endpoint-is-left}
\end{lem}

\begin{proof}
Let us assume that the inequality $e_\calB \leq e_\calA$ holds.
All bigons cover exactly one puncture of $\calL$, and therefore $p_{k-1} < e_\calB < p_k < E_\calB < p_{k+1}$
(with the convention that $p_{n+1} = 1$).
Furthermore, if $\calA$ is a bigon, then $\calA$ must cover some pucture $p_j$ such that
$p_{k-1} < e_\calB \leq e_\calA < p_j$, and therefore $p_k \leq p_j$.
Since $p_j$ cannot lie to the right of $p_k$, it follows that $p_k = p_j$, and since only one bigon covers $p_k$, it follows that $\calA = \calB$,
whence $e_\calA \leq e_\calB$.

Hence, we assume that $\calA$ is not a bigon.
Let $\calC$ be a bigon of $\calL$ covered by $\calA$, with endpoints $e_\calC < E_\calC$.
Since $\calA \neq \calC$, we have $e_\calB \leq e_\calA < e_\calC$,
which was just proven to be impossible. This completes the proof.
\end{proof}

\begin{proof}[Proof of Lemma~\ref{lem:right-oriented}]
Let $\calL$ be a non-trivial tight lamination representing a braid $\beta$, and let $p_k$ be the rightmost puncture
covered by some bigon $\calB$ of the lamination $\calL$.
In addition, let $\calA$ be the arc of $\calB$ along which $p_k$ is slid when $\calL$ is relaxed.
Let $e_\calA < E_\calA$ be the endpoints of $\calA$, and let $e_\calB < E_\calB$ be the endpoints of $\calB$.

Since $\calA \neq \calB$, it follows from Lemma~\ref{lem:left-endpoint-is-left} that $e_\calA < e_\calB < p_k < E_\calB$.
Hence, the arc $\calA$ does not share its left endpoint $e_\calA$ with $\calB$, and therefore
the puncture $p_k$ is slid until it arrives next to $e_\calA$.
Since $e_\calA$ lies in some interval $(p_i,p_{i+1})$ with $i \leq k-1$,
it follows directly that $\bR(\beta)$ is a left-oriented sliding braid.
\end{proof}

\subsection{Proving Proposition~\ref{pro:caracterisation-extension}}

We prove now the first challenging result of this paper, which is Proposition~\ref{pro:caracterisation-extension}.
We do it in two steps, first proving that the conditions enumerated in Proposition~\ref{pro:caracterisation-extension} are necessary,
then that they are sufficient, thereby proving Proposition~\ref{pro:caracterisation-extension} itself.

\begin{lem}
Let $\beta \in B_n$ be some braid, and let $k$ and $\ell$ be integers such that $0 < k < \ell \leq n$.
If the equality $\bR(\beta [k \curvearrowright \ell]) = [k \curvearrowleft \ell]$ holds, then all of the following conditions are fulfilled:
\begin{enumerate}
\item $\pi_\beta(k,+,\downarrow) \neq \{k\}$;
\item either $\pi_\beta(k,+,\uparrow) = \{0,\ldots,k\}$ or $\pi_\beta(k,-,\uparrow) \subseteq \{k,\ldots,\ell-1\}$;
\item for all $i \in \{\ell+2,\ldots,n\}$, $\ell+1 \in \pi_\beta(i,-,\uparrow) \cap \pi_\beta(i,-,\downarrow)$;
\item if $\ell < n$, then $k \in \pi_\beta(\ell+1,+,\uparrow)$;
\item if $\ell < n$, then either $\pi_\beta(\ell+1,+,\downarrow) \neq \{\ell+1\}$ or
$\pi_\beta(\ell+1,-,\uparrow) \subseteq \{k+1,\ldots,\ell\}$.
\end{enumerate}
\label{lem:condition-1}
\end{lem}

\begin{proof}
In what follows, we denote by $\lambda$ the sliding braid $[k \curvearrowright \ell]$,
and let us assume that the equality $\bR(\beta \lambda) = \lambda^{-1}$ holds.
Let $\calL$ be a tight lamination of $\beta$, and let $\oL$ be a tight lamination of $\beta \lambda$.
We will also denote by $p_0,\ldots,p_n$ the punctures of $\calL$, and by $\op_0,\ldots,\op_n$ the punctures of $\oL$.

We first show how to draw $\calL$ by modifying $\oL$.
Since $\bR(\beta \lambda) = \lambda^{-1}$, the puncture $\op_\ell$ of $\oL$ belongs to a lower bigon of $\oL$.
When relaxing $\oL$, the puncture $\op_\ell$ is slid along one of its upper neighboring arcs in $\oL$,
and arrives at some position $p_k \in (\op_{k-1},\op_k)$.
Observe that no upper arc of $\oL$ has shadow $\{\ell\}$,
unless what that arc would cover an upper bigon which covering $\op_\ell$ itself, which is impossible since $\op_\ell$ already belongs to a lower bigon.

Hence, we obtain a lamination $\calL'$ isotopic to $\calL$ by appyling the following steps.
\begin{enumerate}[(i)]
\item We replace the puncture $\op_\ell$ by the puncture $p_k$.
\item For each (lower) arc $A$ of $\oL$ with shadow $\{\ell\}$ in $\oL$,
let $E_1 < E_2$ be the endpoints of $A$, and let $B_1$ and $B_2$ be the upper arcs that respectively share the endpoints $E_1$ and $E_2$.
Lemma~\ref{lem:left-endpoint-is-left} shows that $B_1$ and $B_2$ have endpoints $e_1 < E_1$ and $e_2 < E_2$.
Since $B_1$ must cover some puncture of $\oL$ and since $B_1$ and $B_2$ cannot cross each other,
we have $e_2 < p_k < e_1 < \op_{\ell-1} < E_1 < \op_\ell < E_2 < \op_{\ell+1}$,
regardless of whether $\op_\ell$ is slid along $B_1$ or along $B_2$.
Hence, we merge the arcs $B_1$, $A$ and $B_2$ into one upper arc $C$ whose endpoints are $e_2 < e_1$,
as illustrated in Figure~\ref{fig:proof-step-0}.
\item We do not modify other arcs or punctures of $\oL$.
\end{enumerate}

\begin{figure}[!ht]
\begin{center}
\begin{tikzpicture}[scale=0.35]
\draw[draw=black] (0,0) -- (14,0);

\draw[fill=white,draw=black,thick] (2,0) circle (1.5);
\node at (2,0) {\small$p_k$};
\draw[fill=white,draw=black,thick] (6,0) circle (1.5);
\node at (6,0.6) {\small$\op_{\ell-1}$};
\node at (6,-0.2) {\small$=$};
\node at (6,-0.8) {\small$p_\ell$};
\draw[fill=white,draw=black,thick] (10,0) circle (1.5);
\node at (10,0) {\small$\op_\ell$};
\draw[fill=white,draw=black,thick] (14,0) circle (1.5);
\node at (14,0.6) {\small$\op_{\ell+1}$};
\node at (14,-0.2) {\small$=$};
\node at (14,-0.8) {\small$p_{\ell+1}$};

\draw[draw=gray,ultra thick] (0,0) arc (180:90:4) -- (8,4) arc (90:0:4) arc (360:180:2) arc (0:180:2);

\draw[draw=black,fill=black,ultra thick] (0,0) circle (0.25);
\draw[draw=black,fill=black,ultra thick] (4,0) circle (0.25);
\draw[draw=black,fill=black,ultra thick] (8,0) circle (0.25);
\draw[draw=black,fill=black,ultra thick] (12,0) circle (0.25);
\draw[draw=black,densely dotted] (0,0) -- (0,-3);
\draw[draw=black,densely dotted] (4,0) -- (4,-3);
\draw[draw=black,densely dotted] (8,0) -- (8,-3);
\draw[draw=black,densely dotted] (12,0) -- (12,-3);
\node[anchor=north] at (0,-3) {$e_2$};
\node[anchor=north] at (4,-3) {$e_1$};
\node[anchor=north] at (8,-3) {$E_1$};
\node[anchor=north] at (12,-3) {$E_2$};

\node[anchor=south] at (6,3.85) {$B_2$};
\node[anchor=south] at (6,1.85) {$B_1$};
\node[anchor=north] at (10,-1.85) {$A$};

\draw[->,draw=black,ultra thick] (16.25,0) -- (18.5,0);

\draw[draw=black] (19.5,0) -- (33.5,0);

\draw[fill=white,draw=black,thick] (21.5,0) circle (1.5);
\node at (21.5,0) {\small$p_k$};
\draw[fill=white,draw=black,thick] (25.5,0) circle (1.5);
\node at (25.5,0.6) {\small$\op_{\ell-1}$};
\node at (25.5,-0.2) {\small$=$};
\node at (25.5,-0.8) {\small$p_\ell$};
\draw[fill=white,draw=black,thick] (29.5,0) circle (1.5);
\node at (29.5,0) {\small$\op_\ell$};
\draw[fill=white,draw=black,thick] (33.5,0) circle (1.5);
\node at (33.5,0.6) {\small$\op_{\ell+1}$};
\node at (33.5,-0.2) {\small$=$};
\node at (33.5,-0.8) {\small$p_{\ell+1}$};

\draw[draw=gray,ultra thick] (19.5,0) arc (180:0:2);

\draw[draw=black,fill=black,ultra thick] (19.5,0) circle (0.25);
\draw[draw=black,fill=black,ultra thick] (23.5,0) circle (0.25);
\draw[draw=black,densely dotted] (19.5,0) -- (19.5,-3);
\draw[draw=black,densely dotted] (23.5,0) -- (23.5,-3);
\node[anchor=north] at (19.5,-3) {$e_2$};
\node[anchor=north] at (23.5,-3) {$e_1$};

\node[anchor=south] at (21.5,1.85) {$C$};
\end{tikzpicture}
\end{center}
\caption{Merging arcs of the lamination $\oL$}
\label{fig:proof-step-0}
\end{figure}
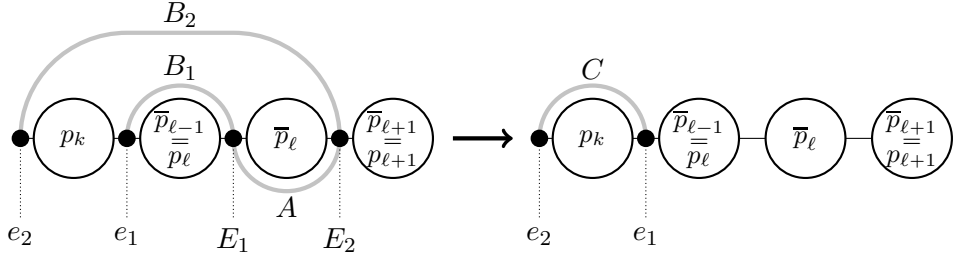

The lamination $\calL'$ obtained in that way is tight.
Indeed, let $\calA$ be an arc of $\calL'$.
If $\calA$ is also an arc of $\oL$, then $\calC$ must cover some puncture $\op_j$ with $j \neq \ell$, and $\op_j$ is also a puncture of $\calL'$.
If $\calA$ is not an arc of $\oL$, then its endpoints are of the form $e_2 < e_1$, with $e_2 < p_k < e_1$, and
therefore $\calA$ covers the puncture $p_k$ ot $\calL'$.
Hence, we may assume that $\calL = \calL'$, and we prove now that
$\beta$ satisfies the conditions enumerated in Proposition~\ref{pro:caracterisation-extension}.

In the enumeration below, we call $A$ the lower bigon of $\oL$ that covers $\op_\ell$ and,
reusing the notations of the step (ii),
we denote by $B_1$ and $B_2$ the arcs of $\oL$ that share an endpoint with $A$, and so on.

\begin{enumerate}
\item Assume here that $\pi_\beta(k,+,\downarrow) = \{k\}$.
Then, the puncture $p_k$ is covered by some lower bigon $\calD$ of $\calL$ and,
like all lower arcs of $\calL$, $\calD$ is also an arc of $\oL$.
Since $\calD$ covers $p_k$ but does not cover any puncture $\op_j$ of $\oL$ with $j \neq \ell$, and
since $p_k < \op_{\ell-1} < \op_\ell$, the arc $\calD$ cannot cover the puncture $\op_\ell$ either.
Hence, $\calD$ covers no puncture of $\oL$, contradicting the tightness of $\oL$.
This contradiction shows that $\pi_\beta(k,+,\downarrow) \neq \{k\}$.

\item If $e_2 < -1$, then $\op_\ell$ was slid again $B_1$, and therefore $C$
is the right upper arc of $p_k$ in the lamination $\calL$, so that
$\pi_\beta(k,+,\uparrow) = \pi_\calL(C) = \{0,\ldots,k\}$.
If $-1 < e_2$, then $\op_\ell$ was slid again $B_1$, and therefore $C$
is the left upper arc of $p_k$ in $\calL$, so that
$\pi_\beta(k,-,\uparrow) = \pi_\calL(C) \subseteq \{k,\ldots,\ell\}$.

\item Let $\calD$ be a left neighbor arc of a puncture $p_j$ in $\calL$, with $j \geq \ell+2$,
and let $e < E$ be the endpoints of $\calA$.
If $\calD$ belongs to $\oL$, then Lemma~\ref{lem:left-endpoint-is-left} proves that
$e \leq \op_\ell < \op_{\ell+1} = p_{\ell+1}$.
If $\calD$ is one of the arcs created during the step (ii),
then we also know that $e < p_\ell < p_{\ell+1}$.
Hence, in both cases, we have $e < p_{\ell+1} \leq p_{j-1} < E$, and therefore $\ell+1 \in \pi_\beta(\calD)$.

\item Let $\calD$ be the right upper neighbor arc of the puncture $p_{\ell+1}$ in $\calL$,
with endpoints $e_\calD < E_\calD$ and $p_{\ell+1} < E_\calD$.
Arcs created during the step (ii) have both their endpoints in the interval $(-\infty,p_\ell)$,
hence $\calD$ is also an arc of $\oL$, and Lemma~\ref{lem:left-endpoint-is-left} proves that
$e_\calD \leq E_1 < E_2 < \op_{\ell+1} = p_{\ell+1} < E_\calD$.
Both $C_2$ and $\calD$ are upper arcs of $\calL$, hence they do not cross each other,
and therefore $e_\calD < e_2 < p_k < p_{\ell+1} < E_\calD$, which is why
$k \in \pi_\calL(\calD) = \pi_\beta(\ell+1,+,\uparrow)$.

\item Assume that $p_{\ell+1}$ belongs to some lower bigon $\calD$ in $\calL$, with endpoints
$e_\calD < E_\calD$. Like all lower arcs of $\calL$, $\calD$ is also an arc of $\oL$.
Since $\calD$ covers the puncture $p_{\ell+1} = \op_{\ell+1}$ but not the puncture $p_\ell = \op_{\ell-1}$,
we have $\op_{\ell-1} < e_\calD < \op_{\ell+1}$. Moreover, $\calD$ is distinct from $A$,
hence Lemma~\ref{lem:left-endpoint-is-left} proves that $e_1 < \op_{\ell-1} < e_\calD < E_1 < \op_\ell$.

Let $\calE$ be the upper arc of $\oL$ that shares the endpoint $e_\calD$,
and let $e_\calE < E_\calE$ be the endpoints of $\calE$.
The upper arcs $B_1$ and $\calE$ of the lamination $\oL$ cannot cross each other,
and no puncture lies on the interval $(e_\calD,E_1)$, hence
$e_\calD$ must be the rightmost puncture of $\calE$, and we have
\[p_k < e_1 < e_\calE < \op_{\ell-1} < E_\calE = e_\calD < \op_{\ell+1} = p_{\ell+1},\]
as illustrated in Figure~\ref{fig:proof-step-5}.
Moreover, the arc $\calE$ was certainly not among the upper arcs of $\oL$ deleted during the phase (ii),
hence $\calE$ is still an arc of $\calL$, and is even the left upper arc of $p_{\ell+1}$.
This proves that $\pi_\beta(\ell+1,-,\uparrow) = \pi_\calL(\calE) \subseteq \{k+1,\ldots,\ell\}$.
\end{enumerate}

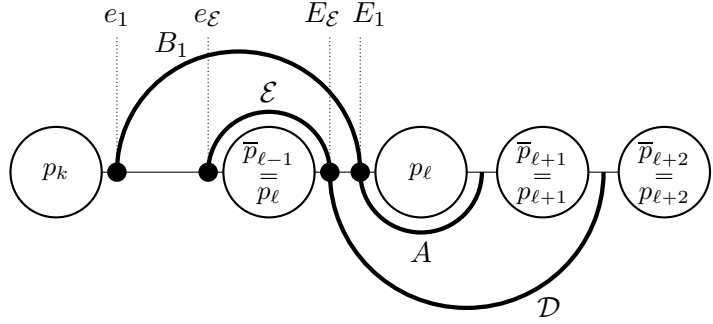
\begin{figure}[!ht]
\begin{center}
\begin{tikzpicture}[scale=0.4]
\draw[draw=black] (-2,0) -- (18,0);

\draw[fill=white,draw=black,thick] (-2,0) circle (1.5);
\node at (-2,0) {\small$p_k$};
\draw[fill=white,draw=black,thick] (5,0) circle (1.5);
\node at (5,0.6) {\small$\op_{\ell-1}$};
\node at (5,-0.2) {\small$=$};
\node at (5,-0.8) {\small$p_\ell$};
\draw[fill=white,draw=black,thick] (10,0) circle (1.5);
\node at (10,0) {\small$p_\ell$};
\draw[fill=white,draw=black,thick] (14,0) circle (1.5);
\node at (14,0.6) {\small$\op_{\ell+1}$};
\node at (14,-0.2) {\small$=$};
\node at (14,-0.8) {\small$p_{\ell+1}$};
\draw[fill=white,draw=black,thick] (18,0) circle (1.5);
\node at (18,0.6) {\small$\op_{\ell+2}$};
\node at (18,-0.2) {\small$=$};
\node at (18,-0.8) {\small$p_{\ell+2}$};

\draw[draw=black,fill=black,ultra thick] (0,0) circle (0.25);
\draw[draw=black,fill=black,ultra thick] (3,0) circle (0.25);
\draw[draw=black,fill=black,ultra thick] (7,0) circle (0.25);
\draw[draw=black,fill=black,ultra thick] (8,0) circle (0.25);
\draw[draw=black,densely dotted] (0,0) -- (0,4.5);
\draw[draw=black,densely dotted] (3,0) -- (3,4.5);
\draw[draw=black,densely dotted] (7,0) -- (7,4.5);
\draw[draw=black,densely dotted] (8,0) -- (8,4.5);
\node[anchor=south] at (0,4.5) {$e_1$};
\node[anchor=south] at (3,4.5) {$e_\calE$};
\node[anchor=south east] at (7.7,4.5) {$E_\calE$};
\node[anchor=south west] at (7.4,4.5) {$E_1$};

\draw[draw=black,ultra thick] (0,0) arc (180:0:4) arc (180:360:2);
\draw[draw=black,ultra thick] (3,0) arc (180:0:2) arc (180:360:4.5);

\node[anchor=south] at (1.8,3.5) {$B_1$};
\node[anchor=south] at (5,2) {$\calE$};
\node[anchor=north] at (10,-2) {$A$};
\node[anchor=north] at (14.2,-3.8) {$\calD$};
\end{tikzpicture}
\end{center}
\caption{A fragment of the lamination $\oL$}
\label{fig:proof-step-5}
\end{figure}
\end{proof}

\begin{lem}
Let $\beta \in B_n$ be some braid, and let $k$ and $\ell$ be integers such that $0 < k < \ell \leq n$.
Let us assume that the following conditions are fulfilled:
\begin{enumerate}
\item $\pi_\beta(k,+,\downarrow) \neq \{k\}$;
\item either $\pi_\beta(k,+,\uparrow) = \{0,\ldots,k\}$ or $\pi_\beta(k,-,\uparrow) \subseteq \{k,\ldots,\ell-1\}$;
\item for all $i \in \{\ell+2,\ldots,n\}$, $\ell+1 \in \pi_\beta(i,-,\uparrow) \cap \pi_\beta(i,-,\downarrow)$;
\item if $\ell < n$, then $k \in \pi_\beta(\ell+1,+,\uparrow)$;
\item if $\ell < n$, then either $\pi_\beta(\ell+1,+,\downarrow) \neq \{\ell+1\}$ or
$\pi_\beta(\ell+1,-,\uparrow) \subseteq \{k+1,\ldots,\ell\}$.
\end{enumerate}
Then, the equality $\bR(\beta [k \curvearrowright \ell]) = [k \curvearrowleft \ell]$ holds.
\label{lem:condition-2}
\end{lem}

\begin{proof}
Again, we denote by $\lambda$ the sliding braid $[k \curvearrowright \ell]$,
and we assume that the conditions (1) to (5) hold.
Let $\calL$ be a tight lamination of $\beta$, with punctures $p_0,\ldots,p_n$, and
let $\oL$ be a tight lamination of $\beta \lambda$, with punctures $\op_0,\ldots,\op_n$.

\begin{figure}[!ht]
\begin{center}
\begin{tikzpicture}[scale=0.135]
\draw[draw=black,ultra thick] (-2,7) -- (88,7) -- (88,-23) -- (-2,-23) -- cycle;
\draw[draw=black,ultra thick] (-2,-3) -- (88,-3);
\draw[draw=black,ultra thick] (-2,-13) -- (88,-13);
\draw[draw=black,ultra thick] (28,-23) -- (28,7);
\draw[draw=black,ultra thick] (58,-23) -- (58,7);

\node[anchor=north west] at (-1.5,6.5) {$\Omega_1$};
\node[anchor=north west] at (-1.5,-3.5) {$\Omega_2$};
\node[anchor=north west] at (-1.5,-13.5) {$\Omega_3$};
\node[anchor=north west] at (28.5,6.5) {$\Omega_4$};
\node[anchor=north west] at (28.5,-3.5) {$\Omega_5$};
\node[anchor=north west] at (28.5,-13.5) {$\Omega_6$};
\node[anchor=north west] at (58.5,6.5) {$\Omega_7$};
\node[anchor=north west] at (58.5,-3.5) {$\Omega_8$};
\node[anchor=north west] at (58.5,-13.5) {$\Omega_9$};

\draw[draw=black] (0,0) -- (26,0);
\draw[draw=black] (30,0) -- (56,0);
\draw[draw=black] (60,0) -- (86,0);
\draw[draw=black] (0,-10) -- (26,-10);
\draw[draw=black] (30,-10) -- (56,-10);
\draw[draw=black] (60,-10) -- (86,-10);
\draw[draw=black] (0,-20) -- (26,-20);
\draw[draw=black] (30,-20) -- (56,-20);
\draw[draw=black] (60,-20) -- (86,-20);

\draw[draw=black,thick,densely dotted] (7,10) -- (7,-20);
\draw[draw=black,thick,densely dotted] (15,10) -- (15,-20);
\draw[draw=black,thick,densely dotted] (19,10) -- (19,-20);
\draw[draw=black,thick,densely dotted] (37,10) -- (37,-20);
\draw[draw=black,thick,densely dotted] (45,10) -- (45,-20);
\draw[draw=black,thick,densely dotted] (49,10) -- (49,-20);
\draw[draw=black,thick,densely dotted] (67,10) -- (67,-20);
\draw[draw=black,thick,densely dotted] (75,10) -- (75,-20);
\draw[draw=black,thick,densely dotted] (79,10) -- (79,-20);
\node[anchor=south] at (7,10) {$p_k$};
\node[anchor=south] at (15,10) {$p_\ell$};
\node[anchor=south] at (19,10) {$p_{\ell+1}$};
\node[anchor=south] at (37,10) {$p_k$};
\node[anchor=south] at (45,10) {$p_\ell$};
\node[anchor=south] at (49,10) {$p_{\ell+1}$};
\node[anchor=south] at (67,10) {$p_k$};
\node[anchor=south] at (75,10) {$p_\ell$};
\node[anchor=south] at (79,10) {$p_{\ell+1}$};

\draw[draw=black,fill=white,thick] (7,0) circle (0.75);
\draw[draw=black,fill=white,thick] (15,0) circle (0.75);
\draw[draw=black,fill=white,thick] (19,0) circle (0.75);
\draw[draw=black,fill=white,thick] (37,0) circle (0.75);
\draw[draw=black,fill=white,thick] (45,0) circle (0.75);
\draw[draw=black,fill=white,thick] (49,0) circle (0.75);
\draw[draw=black,fill=white,thick] (67,0) circle (0.75);
\draw[draw=black,fill=white,thick] (75,0) circle (0.75);
\draw[draw=black,fill=white,thick] (79,0) circle (0.75);
\draw[draw=black,fill=white,thick] (7,-10) circle (0.75);
\draw[draw=black,fill=white,thick] (15,-10) circle (0.75);
\draw[draw=black,fill=white,thick] (19,-10) circle (0.75);
\draw[draw=black,fill=white,thick] (37,-10) circle (0.75);
\draw[draw=black,fill=white,thick] (45,-10) circle (0.75);
\draw[draw=black,fill=white,thick] (49,-10) circle (0.75);
\draw[draw=black,fill=white,thick] (67,-10) circle (0.75);
\draw[draw=black,fill=white,thick] (75,-10) circle (0.75);
\draw[draw=black,fill=white,thick] (79,-10) circle (0.75);
\draw[draw=black,fill=white,thick] (7,-20) circle (0.75);
\draw[draw=black,fill=white,thick] (15,-20) circle (0.75);
\draw[draw=black,fill=white,thick] (19,-20) circle (0.75);
\draw[draw=black,fill=white,thick] (37,-20) circle (0.75);
\draw[draw=black,fill=white,thick] (45,-20) circle (0.75);
\draw[draw=black,fill=white,thick] (49,-20) circle (0.75);
\draw[draw=black,fill=white,thick] (67,-20) circle (0.75);
\draw[draw=black,fill=white,thick] (75,-20) circle (0.75);
\draw[draw=black,fill=white,thick] (79,-20) circle (0.75);

\draw[draw=black,ultra thick] (5,0) arc (180:0:2);
\draw[draw=black,ultra thick] (13,-10) arc (180:0:2);
\draw[draw=black,ultra thick] (17,-20) arc (180:0:2);
\draw[draw=black,ultra thick] (35,0) arc (180:90:2) -- (45,2) arc (90:0:2);
\draw[draw=black,ultra thick] (43,-10) arc (180:90:2) -- (49,-8) arc (90:0:2);
\draw[draw=black,ultra thick] (35,-20) arc (180:90:2) -- (49,-18) arc (90:0:2);
\draw[draw=black,ultra thick] (61,0) arc (180:0:2);
\draw[draw=black,ultra thick] (69,-10) arc (180:0:2);
\draw[draw=black,ultra thick] (81,-20) arc (180:0:2);
\end{tikzpicture}
\end{center}
\caption{Nine classes of upper arcs}
\label{fig:omega}
\end{figure}
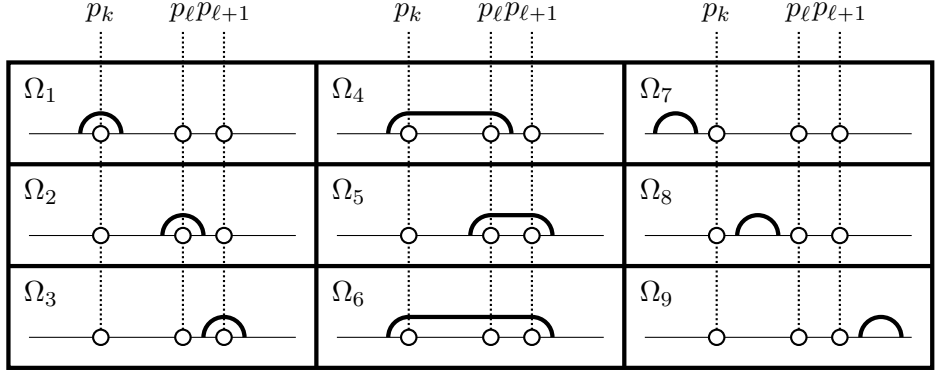

Let $\upbA$ denote the set of the upper arcs of $\calL$.
We partition $\upbA$ in several subsets, as illustrated in Figure~\ref{fig:omega}:
\begin{itemize}
\item $\Omega_1 = \{A \in \upbA : k \in \pi_\calL(A) \text{ and } \ell,\ell+1 \notin \pi_\calL(A)\}$,
\item $\Omega_2 = \{A \in \upbA : \ell \in \pi_\calL(A) \text{ and } k, \ell+1 \notin \pi_\calL(A)\}$,
\item $\Omega_3 = \{A \in \upbA : \ell+1 \in \pi_\calL(A) \text{ and } k, \ell \notin \pi_\calL(A)\}$,
\item $\Omega_4 = \{A \in \upbA : k,\ell \in \pi_\calL(A) \text{ and } \ell+1 \notin \pi_\calL(A)\}$,
\item $\Omega_5 = \{A \in \upbA : \ell,\ell+1 \in \pi_\calL(A) \text{ and } k \notin \pi_\calL(A)\}$,
\item $\Omega_6 = \{A \in \upbA : k,\ell,\ell+1 \in \pi_\calL(A)\}$,
\item $\Omega_7 = \{A \in \upbA : \pi_\calL(A) \subseteq \{0,\ldots,k-1\}\}$,
\item $\Omega_8 = \{A \in \upbA : \pi_\calL(A) \subseteq \{k+1,\ldots,\ell-1\}\}$, and
\item $\Omega_9 = \{A \in \upbA : \pi_\calL(A) \subseteq \{\ell+2,\ldots,n\}\}$.
\end{itemize}

We first show that the sets $\Omega_3$, $\Omega_5$ and $\Omega_9$ are empty.
If $\ell = n$, then of course $\Omega_5 = \Omega_3 = \emptyset$.
If $\ell < n$, then the requirement $4$ indicates that $k \in \pi_\calL(\upA_+(p_{\ell+1},\calL))$.
This proves that all upper arcs that cover $p_{\ell+1}$ also cover $p_k$, whence $\Omega_5 = \Omega_3 = \emptyset$.

Then, the requirement $3$ indicates that no puncture $p_i$ with $i \geq \ell+2$ belongs to a bigon of $\calL$.
Hence, for all arcs $\calA$ of $\calL$, Lemma~\ref{lem:left-endpoint-is-left} proves that
$e_A < p_{\ell+1}$, where $e_A$ is the leftmost endpoint of $\calA$. This proves that
$\calA$ cannot belong to the set $\Omega_9$, i.e. that $\Omega_9 = \emptyset$.

In the following diagram, we use notations defined as follows.
Let $X$ and $Y$ be subsets of $\upbA$.
If a real point $p \in \RR$ is covered by all arcs $\calA \in X$, then we write $p \prec X$.
Similarly, if, for all arcs $\calA \in X$ and $\calB \in Y$, the arc $\calA$ is covered by $\calB$, then we write $X \prec Y$.
Using these notations, the relations shown in Figure~\ref{fig:omega-relations} are clear.

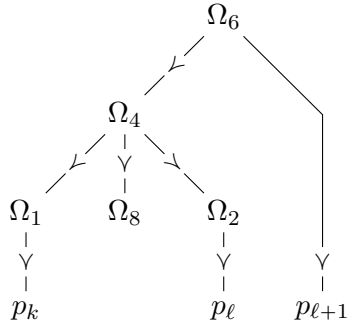
\begin{figure}[!ht]
\begin{center}
\begin{tikzpicture}[scale=1.3]
\node at (1,0) {$\Omega_6$};
\node at (0,-1) {$\Omega_4$};
\node at (-1,-2) {$\Omega_1$};
\node at (0,-2) {$\Omega_8$};
\node at (1,-2) {$\Omega_2$};
\node at (-1,-3) {$p_k$};
\node at (1,-3) {$p_\ell$};
\node at (2,-3) {$p_{\ell+1}$};

\draw[draw=black] (0.2,-0.8) -- (0.4,-0.6);
\draw[draw=black] (0.6,-0.4) -- (0.8,-0.2);
\node[rotate=45] at (0.5,-0.5) {\small $\prec$};

\draw[draw=black] (-0.8,-1.8) -- (-0.6,-1.6);
\draw[draw=black] (-0.4,-1.4) -- (-0.2,-1.2);
\node[rotate=45] at (-0.5,-1.5) {\small $\prec$};

\draw[draw=black] (0,-1.8) -- (0,-1.65);
\draw[draw=black] (0,-1.35) -- (0,-1.2);
\node[rotate=90] at (-0,-1.5) {\small $\prec$};

\draw[draw=black] (0.8,-1.8) -- (0.6,-1.6);
\draw[draw=black] (0.4,-1.4) -- (0.2,-1.2);
\node[rotate=135] at (0.5,-1.5) {\small $\prec$};

\draw[draw=black] (-1,-2.8) -- (-1,-2.65);
\draw[draw=black] (-1,-2.35) -- (-1,-2.2);
\node[rotate=90] at (-1,-2.5) {\small $\prec$};

\draw[draw=black] (1,-2.8) -- (1,-2.65);
\draw[draw=black] (1,-2.35) -- (1,-2.2);
\node[rotate=90] at (1,-2.5) {\small $\prec$};

\draw[draw=black] (2,-2.8) -- (2,-2.65);
\draw[draw=black] (2,-2.35) -- (2,-1) -- (1.2,-0.2);
\node[rotate=90] at (2,-2.5) {\small $\prec$};
\end{tikzpicture}
\end{center}
\caption{Blinding relations between sets $\Omega_i$}
\label{fig:omega-relations}
\end{figure}

Furthermore, note that, for each point $p \in \RR$, the coverage relation induces a total order
on the set of the upper arcs of $\calL$ that cover the point $p$.
Consequently, each of the sets $\Omega_1$, $\Omega_2$, $\Omega_4$ and $\Omega_6$ is totally ordered by the coverage relation.
In what follows, we set $\omega_i = \#\Omega_i$, and we denote by $A_1^i, \ldots, A_{\omega_i}^i$ the elements of $\Omega_i$, with respective endpoints $e_j^i < E_j^j$,
such that the arc $A_j^i$ is covered by the arc $A_k^i$ if and only if $j < k$.

It is then straightforward to see that the elements of $(p_\ell,p_{\ell+1}) \cap \calL$ are
\[p_\ell < E_1^2 < \ldots < E_{\omega_2}^2 < E_1^4 < \ldots < E_{\omega_4}^4 < p_{\ell+1}.\]
Hence, let $a_j^-$ and $a_j^+$ (for $1 \leq j \leq \omega_1$) and $\op_\ell$ be real numbers such that
\[E_{\omega_2}^2 < a_{\omega_1}^- < \ldots < a_1^- < \op_\ell <
a_1^+ < \ldots < a_{\omega_1}^+ < E_1^4,\]
as illustrated in the top picture of Figure~\ref{fig:order-endpoints-omega}.

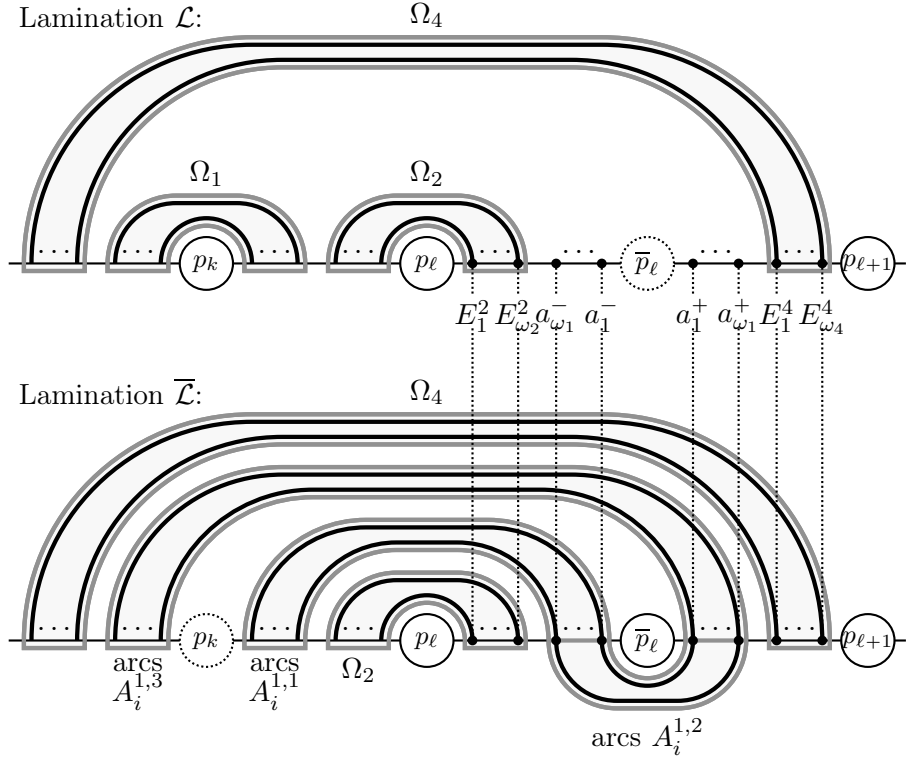
\begin{figure}[!ht]
\begin{center}
\begin{tikzpicture}[scale=0.2]
\draw[draw=blackgray,fill=verypalegray,ultra thick] (-3.5,15) arc (180:90:15) -- (34.5,30) arc (90:0:15) -- (49.5,14.5) -- (45.5,14.5) --
(45.5,15) arc (0:90:13) -- (13.5,28) arc (90:180:13) -- (0.5,14.5) -- (-3.5,14.5) -- cycle;
\draw[draw=blackgray,fill=verypalegray,ultra thick] (2,15) arc (180:90:4.5) -- (10.5,19.5) arc (90:0:4.5) -- (15,14.5) -- (11,14.5) --
(11,15) arc (0:180:2.5) -- (6,14.5) -- (2,14.5) -- cycle;
\draw[draw=blackgray,fill=verypalegray,ultra thick] (16.5,15) arc (180:90:4.5) -- (25,19.5) arc (90:0:4.5) -- (29.5,14.5) -- (25.5,14.5) --
(25.5,15) arc (0:180:2.5) -- (20.5,14.5) -- (16.5,14.5) -- cycle;

\draw[draw=black,thick] (-4.5,15) -- (54.5,15);

\draw[draw=blackgray,ultra thick] (-3.5,15) arc (180:90:15) -- (34.5,30) arc (90:0:15) -- (49.5,14.5) -- (45.5,14.5) --
(45.5,15) arc (0:90:13) -- (13.5,28) arc (90:180:13) -- (0.5,14.5) -- (-3.5,14.5) -- cycle;
\draw[draw=blackgray,ultra thick] (2,15) arc (180:90:4.5) -- (10.5,19.5) arc (90:0:4.5) -- (15,14.5) -- (11,14.5) --
(11,15) arc (0:180:2.5) -- (6,14.5) -- (2,14.5) -- cycle;
\draw[draw=blackgray,ultra thick] (16.5,15) arc (180:90:4.5) -- (25,19.5) arc (90:0:4.5) -- (29.5,14.5) -- (25.5,14.5) --
(25.5,15) arc (0:180:2.5) -- (20.5,14.5) -- (16.5,14.5) -- cycle;

\draw[draw=black,ultra thick] (20,15) arc (180:0:3);
\draw[draw=black,ultra thick] (17,15) arc (180:90:4) -- (25,19) arc (90:0:4);
\draw[draw=black,ultra thick] (5.5,15) arc (180:0:3);
\draw[draw=black,ultra thick] (2.5,15) arc (180:90:4) -- (10.5,19) arc (90:0:4);
\draw[draw=black,ultra thick] (0,15) arc (180:90:13.5) -- (32.5,28.5) arc (90:0:13.5);
\draw[draw=black,ultra thick] (-3,15) arc (180:90:14.5) -- (34.5,29.5) arc (90:0:14.5);

\draw[draw=blackgray,fill=verypalegray,ultra thick] (-3.5,-10) arc (180:90:15) -- (34.5,5) arc (90:0:15) -- (49.5,-10.5) -- (45.5,-10.5) --
(45.5,-10) arc (0:90:13) -- (13.5,3) arc (90:180:13) -- (0.5,-10.5) -- (-3.5,-10.5) -- cycle;
\draw[draw=blackgray,fill=verypalegray,ultra thick] (15,-10) arc (180:90:6) -- (25,-4) arc (90:0:6) arc (180:270:4.5) -- (39.5,-14.5) arc (270:360:4.5)
arc (0:90:11.5) -- (13.5,1.5) arc (90:180:11.5) -- (2,-10.5) -- (6,-10.5) -- (6,-10) arc (180:90:9.5) -- (30.5,-0.5) arc (90:0:9.5) arc (360:180:2.5)
arc (0:90:8) -- (19,-2) arc (90:180:8) -- (11,-10.5) -- (15,-10.5) -- cycle;
\draw[draw=blackgray,fill=verypalegray,ultra thick] (16.5,-10) arc (180:90:4.5) -- (25,-5.5) arc (90:0:4.5) -- (29.5,-10.5) -- (25.5,-10.5) --
(25.5,-10) arc (0:180:2.5) -- (20.5,-10.5) -- (16.5,-10.5) -- cycle;

\draw[draw=black,thick] (-4.5,-10) -- (54.5,-10);

\draw[draw=blackgray,ultra thick] (-3.5,-10) arc (180:90:15) -- (34.5,5) arc (90:0:15) -- (49.5,-10.5) -- (45.5,-10.5) --
(45.5,-10) arc (0:90:13) -- (13.5,3) arc (90:180:13) -- (0.5,-10.5) -- (-3.5,-10.5) -- cycle;
\draw[draw=blackgray,ultra thick] (15,-10) arc (180:90:6) -- (25,-4) arc (90:0:6) arc (180:270:4.5) -- (39.5,-14.5) arc (270:360:4.5)
arc (0:90:11.5) -- (13.5,1.5) arc (90:180:11.5) -- (2,-10.5) -- (6,-10.5) -- (6,-10) arc (180:90:9.5) -- (30.5,-0.5) arc (90:0:9.5) arc (360:180:2.5)
arc (0:90:8) -- (19,-2) arc (90:180:8) -- (11,-10.5) -- (15,-10.5) -- cycle;
\draw[draw=blackgray,ultra thick] (16.5,-10) arc (180:90:4.5) -- (25,-5.5) arc (90:0:4.5) -- (29.5,-10.5) -- (25.5,-10.5) --
(25.5,-10) arc (0:180:2.5) -- (20.5,-10.5) -- (16.5,-10.5) -- cycle;
\draw[draw=blackgray,ultra thick] (31,-10) -- (35,-10);
\draw[draw=blackgray,ultra thick] (40,-10) -- (44,-10);

\draw[draw=black,ultra thick] (20,-10) arc (180:0:3);
\draw[draw=black,ultra thick] (17,-10) arc (180:90:4) -- (25,-6) arc (90:0:4);
\draw[draw=black,ultra thick] (14.5,-10) arc (180:90:6.5) -- (25,-3.5) arc (90:0:6.5) arc (180:270:4) -- (39.5,-14) arc
(270:360:4) arc (0:90:11) -- (13.5,1) arc (90:180:11);
\draw[draw=black,ultra thick] (11.5,-10) arc (180:90:7.5) -- (27,-2.5) arc (90:0:7.5) arc (180:360:3) arc (0:90:10) -- (15.5,0) arc (90:180:10);
\draw[draw=black,ultra thick] (0,-10) arc (180:90:13.5) -- (32.5,3.5) arc (90:0:13.5);
\draw[draw=black,ultra thick] (-3,-10) arc (180:90:14.5) -- (34.5,4.5) arc (90:0:14.5);

\draw[draw=black,fill=black,thick] (26,15) circle (0.25);
\draw[draw=black,fill=black,thick] (29,15) circle (0.25);
\draw[draw=black,fill=black,thick] (31.5,15) circle (0.25);
\draw[draw=black,fill=black,thick] (34.5,15) circle (0.25);
\draw[draw=black,fill=black,thick] (40.5,15) circle (0.25);
\draw[draw=black,fill=black,thick] (43.5,15) circle (0.25);
\draw[draw=black,fill=black,thick] (46,15) circle (0.25);
\draw[draw=black,fill=black,thick] (49,15) circle (0.25);

\draw[draw=black,fill=black,thick] (26,-10) circle (0.25);
\draw[draw=black,fill=black,thick] (29,-10) circle (0.25);
\draw[draw=black,fill=black,thick] (31.5,-10) circle (0.25);
\draw[draw=black,fill=black,thick] (34.5,-10) circle (0.25);
\draw[draw=black,fill=black,thick] (40.5,-10) circle (0.25);
\draw[draw=black,fill=black,thick] (43.5,-10) circle (0.25);
\draw[draw=black,fill=black,thick] (46,-10) circle (0.25);
\draw[draw=black,fill=black,thick] (49,-10) circle (0.25);

\draw[draw=black,densely dotted,thick] (26,13) -- (26,15);
\draw[draw=black,densely dotted,thick] (29,13) -- (29,15);
\draw[draw=black,densely dotted,thick] (31.5,13) -- (31.5,15);
\draw[draw=black,densely dotted,thick] (34.5,13) -- (34.5,15);
\draw[draw=black,densely dotted,thick] (40.5,13) -- (40.5,15);
\draw[draw=black,densely dotted,thick] (43.5,13) -- (43.5,15);
\draw[draw=black,densely dotted,thick] (46,13) -- (46,15);
\draw[draw=black,densely dotted,thick] (49,13) -- (49,15);

\draw[draw=black,densely dotted,thick] (26,10.7) -- (26,-10);
\draw[draw=black,densely dotted,thick] (29,10.7) -- (29,-10);
\draw[draw=black,densely dotted,thick] (31.5,10.7) -- (31.5,-10);
\draw[draw=black,densely dotted,thick] (34.5,10.7) -- (34.5,-10);
\draw[draw=black,densely dotted,thick] (40.5,10.7) -- (40.5,-10);
\draw[draw=black,densely dotted,thick] (43.5,10.7) -- (43.5,-10);
\draw[draw=black,densely dotted,thick] (46,10.7) -- (46,-10);
\draw[draw=black,densely dotted,thick] (49,10.7) -- (49,-10);

\node[anchor=north] at (26,13.25) {$E_1^2$};
\node[anchor=north] at (29,13.25) {$E_{\omega_2}^2$};
\node[anchor=north] at (31.5,13.35) {$a_{\omega_1}^-$};
\node[anchor=north] at (34.5,13.35) {$a_1^-$};
\node[anchor=north] at (40.5,13.35) {$a_1^+$};
\node[anchor=north] at (43.5,13.35) {$a_{\omega_1}^+$};
\node[anchor=north] at (46,13.25) {$E_1^4$};
\node[anchor=north] at (49,13.25) {$E_{\omega_4}^4$};

\node[anchor=south] at (-1.5,15) {$\ldots$};
\node[anchor=south] at (4,15) {$\ldots$};
\node[anchor=south] at (13,15) {$\ldots$};
\node[anchor=south] at (18.5,15) {$\ldots$};
\node[anchor=south] at (27.5,15) {$\ldots$};
\node[anchor=south] at (33,15) {$\ldots$};
\node[anchor=south] at (42,15) {$\ldots$};
\node[anchor=south] at (47.5,15) {$\ldots$};

\node[anchor=south] at (8.5,19.5) {$\Omega_1$};
\node[anchor=south] at (23,19.5) {$\Omega_2$};
\node[anchor=south] at (23,30) {$\Omega_4$};

\node[anchor=north] at (18.5,-10.5) {$\Omega_2$};
\node[anchor=south] at (23,5) {$\Omega_4$};
\node[anchor=north] at (4,-10.5) {arcs};
\node[anchor=north] at (4,-11.5) {$A_i^{1,3}$};
\node[anchor=north] at (13,-10.5) {arcs};
\node[anchor=north] at (13,-11.5) {$A_i^{1,1}$};
\node[anchor=north] at (37.5,-14.5) {arcs $A_i^{1,2}$};

\node[anchor=south] at (-1.5,-10) {$\ldots$};
\node[anchor=south] at (4,-10) {$\ldots$};
\node[anchor=south] at (13,-10) {$\ldots$};
\node[anchor=south] at (18.5,-10) {$\ldots$};
\node[anchor=south] at (27.5,-10) {$\ldots$};
\node[anchor=south] at (33,-10) {$\ldots$};
\node[anchor=south] at (42,-10) {$\ldots$};
\node[anchor=south] at (47.5,-10) {$\ldots$};

\draw[fill=white,draw=black,thick] (8.5,15) circle (1.75);
\node at (8.5,15) {\small$p_k$};
\draw[fill=white,draw=black,thick] (23,15) circle (1.75);
\node at (23,15) {\small$p_\ell$};
\draw[fill=white,draw=black,thick,densely dotted] (37.5,15) circle (1.75);
\node at (37.5,15) {\small$\op_\ell$};
\draw[fill=white,draw=black,thick] (52,15) circle (1.75);
\node at (52,15) {\small$p_{\ell+1}$};

\draw[fill=white,draw=black,thick,densely dotted] (8.5,-10) circle (1.75);
\node at (8.5,-10) {\small$p_k$};
\draw[fill=white,draw=black,thick] (23,-10) circle (1.75);
\node at (23,-10) {\small$p_\ell$};
\draw[fill=white,draw=black,thick] (37.5,-10) circle (1.75);
\node at (37.5,-10) {\small$\op_\ell$};
\draw[fill=white,draw=black,thick] (52,-10) circle (1.75);
\node at (52,-10) {\small$p_{\ell+1}$};

\node[anchor=south west] at (-4.5,30) {Lamination $\calL$:};
\node[anchor=south west] at (-4.5,5) {Lamination $\oL$:};
\end{tikzpicture}
\end{center}
\caption{Ordering $(p_\ell,p_{\ell+1}) \cap \calL$ -- Adding points $a_i^\pm$ and $\op_\ell$ -- Going from $\calL$ to $\oL$}
\label{fig:order-endpoints-omega}
\end{figure}

Hence, like in the proof of Lemma~\ref{lem:condition-1}, and as illustrated in
Figure~\ref{fig:order-endpoints-omega}, we obtain a lamination $\calL'$ isotopic to $\oL$ by appyling the following steps.
\begin{enumerate}[(i)]
\item We replace the puncture $p_k$ by the puncture $\op_\ell$.
\item We replace each (upper) arc $A_j^1$ in $\Omega_1$ by three arcs:
one upper arc $A_j^{1,1}$ with endpoints $E_j^1 < a_j^-$,
one lower arc $A_j^{1,2}$ with endpoints $a_j^- < a_j^+$, and
one upper arc $A_j^{1,3}$ with endpoints $e_j^1 < a_j^+$.
\item We do not modify other arcs or punctures of $\calL$.
\end{enumerate}

This lamination $\calL'$ is tight. Indeed, let $\calA$ be an arc of $\calL'$.
If $\calA$ is one of the arcs $A_j^{1,1}$, $A_j^{1,2}$ or $A_j^{1,3}$, then
$\calA$ covers either $p_\ell$ or $\op_\ell$ (or both).
If $\calA$ is also an upper arc of $\calL$, then $\calA \notin \Omega_1$, hence
$\calA$ covers some puncture $p_j$ with $j \neq k$.
Finally, if $\calA$ is also a lower arc of $\calL$, the requirement $1$
proves that $p_k$ does not belong to a lower bigon of $\calL$.
Hence, $\calA$ cannot cover such a (non-existent) bigon,
and therefore $\calA$ must also cover some puncture $p_j$ with $j \neq k$.
Consequently, we may assume that $\oL = \calL'$, and we prove now that
$\bR(\beta \lambda) = \lambda^{-1}$.

Let $\calB$ be the rightmost bigon of $\oL$, and let $\op_i$ be the puncture of $\oL$ that is covered by $\calB$,
and let us assume that $\calB \neq A_1^{1,2}$, i.e. that $i \geq \ell+1$.
If $i \geq \ell+1$, then $\calB$ cannot be one of the arcs $A_j^{1,1}$, $A_j^{1,2}$ or $A_j^{1,3}$,
hence $\calB$ is also a bigon of $\calL$.
Due to the requirement $3$, it follows that $i = \ell+1$, and
the requirement $4$ then shows that $\calB$ must be a lower bigon.

The requirement $5$ allows us to conclude that the left upper arc of $p_{\ell+1}$
covers the puncture $p_\ell$ but not the puncture $p_k$, which proves that $\Omega_4 = \emptyset$.
The requirement $2$ proves that $\Omega_1 \neq \emptyset$. Consequently, like all lower arcs of $\oL$
that cover the puncture $p_{\ell+1}$, the arc $\calB$ also covers $\op_\ell$,
which contradicts that it is a bigon.
Informally, this means that $\calB$ was a bigon of $\calL$ and that the new arc $A_1^{1,2}$
was ``inserted'' inside $\calB$, as illustrated in Figure~\ref{fig:i=ell+1}.
Hence, $A_1^{1,2}$ is the rightmost bigon of $\oL$.

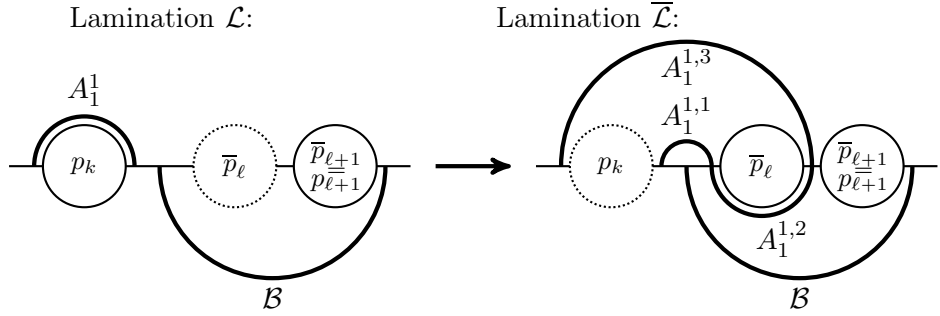
\begin{figure}[!ht]
\begin{center}
\begin{tikzpicture}[scale=0.33]
\draw[draw=black,thick] (4,0) -- (20,0);

\draw[draw=black,ultra thick] (5,0) arc (180:0:2);
\draw[draw=black,ultra thick] (10,0) arc (180:360:4.5);
\draw[fill=white,draw=black,thick] (7,0) circle (1.65);
\node at (7,0) {\small$p_k$};
\draw[fill=white,draw=black,thick,densely dotted] (13,0) circle (1.65);
\node at (13,0) {\small$\op_\ell$};
\draw[fill=white,draw=black,thick] (17,0) circle (1.65);
\node at (17.05,0.5) {\small$\op_{\ell+1}$};
\node at (17.05,-0.25) {\small$=$};
\node at (17.05,-0.7) {\small$p_{\ell+1}$};

\draw[->,>=stealth',draw=black,ultra thick] (21,0) -- (24,0);

\draw[draw=black,thick] (25,0) -- (41,0);

\draw[draw=black,ultra thick] (26,0) arc (180:0:5) arc (360:180:2) arc (0:180:1);
\draw[draw=black,ultra thick] (31,0) arc (180:360:4.5);
\draw[fill=white,draw=black,thick,densely dotted] (28,0) circle (1.65);
\node at (28,0) {\small$p_k$};
\draw[fill=white,draw=black,thick] (34,0) circle (1.65);
\node at (34,0) {\small$\op_\ell$};
\draw[fill=white,draw=black,thick] (38,0) circle (1.65);
\node at (38.05,0.5) {\small$\op_{\ell+1}$};
\node at (38.05,-0.25) {\small$=$};
\node at (38.05,-0.7) {\small$p_{\ell+1}$};

\node[anchor=south] at (7,2) {$A_1^1$};
\node[anchor=north] at (14.5,-4.5) {$\calB$};
\node[anchor=north] at (31,5.1) {$A_1^{1,3}$};
\node[anchor=south] at (31,0.9) {$A_1^{1,1}$};
\node[anchor=north] at (34.8,-1.8) {$A_1^{1,2}$};
\node[anchor=north] at (35.5,-4.5) {$\calB$};

\node[anchor=south west] at (6,5.2) {Lamination $\calL$:};
\node[anchor=south west] at (23,5.2) {Lamination $\oL$:};
\end{tikzpicture}
\end{center}
\caption{From $\calL$ to $\oL$ when $p_{\ell+1}$ belongs to a lower bigon of $\calL$}
\label{fig:i=ell+1}
\end{figure}

Finally, according to the requirement $2$, two cases are possible.
\begin{itemize}
\item If $\pi_\beta(k,+,\uparrow) = \{0,\ldots,k\}$, then $A_1^1 = \upA_+(p_k,\calL)$,
and therefore both $A_1^1$ and $A_1^{1,3}$ cover the puncture $p_0 = \op_0$.
Moreover, the interval $(p_k,E_1^1)$ contains no endpoint of any arc of $\calL$ nor of $\oL$.
Hence, the lamination $\calL$ is obtained from $\oL$ by sliding the puncture $\op_\ell$ along the arc $A_1^{1,1} = \upA_-(\op_\ell,\oL)$,
then merging arcs of $\oL$.
\item If $\pi_\beta(k,-,\uparrow) \subseteq \{k,\ldots,\ell-1\}$, then 
$A_1^1 = \upA_-(p_k,\calL)$, and therefore
$A_{1,3}$ does not cover the puncture $p_0 = \op_0$.
Moreover, the interval $(e_1^1,p_k)$ contains no endpoint of any arc of $\calL$ nor of $\oL$.
The lamination $\calL$ is therefore obtained from $\oL$ by sliding the puncture $\op_\ell$
along the arc $A_1^{1,3} = \upA_+(\op_\ell,\oL)$, then merging arcs of $\oL$.
\end{itemize}
In both cases, it follows that 
$\bR(\beta \lambda) = \lambda^{-1}$, which completes the proof.
\end{proof}

\subsection{Proving Proposition~\ref{pro:automata-state}}

Having proved Proposition~\ref{pro:caracterisation-extension}
in two steps, we wish to prove Proposition~\ref{pro:automata-state}.
However, we first need to introduce new notions and results.

\begin{dfn}{Children of an arc}
Let $\calL$ be a lamination, let $\calA$ and $\calB$ be two arcs of $\calL$,
and let $p$ be a puncture of $\calL$.

We say that $\calA$ is a \emph{child} of $\calB$ if (i) $\calA$ is covered by $\calB$
and if (ii) for all arcs $\calC$ of $\calL$ that cover $\calA$,
either $\calB = \calC$ or $\calC$ covers $\calB$.
Similarly, we say that $p$ is a \emph{child} of $\calB$ if
(i) $p$ is covered by $\calB$ and if (ii) for all arcs $\calC$ of $\calL$ that cover $p$,
either $\calB = \calC$ or $\calC$ covers $\calB$.
\label{dfn:children}
\end{dfn}

\begin{dfn}{Cells and boundaries}
Let $\calL$ be a lamination.
We call \emph{cell} of $\calL$ each finite connected component of the set
$\CC \setminus (\calL \cup \RR)$.

In addition, we call \emph{arc boundary} of a cell $\calC$
each arc of $\calL$ that belongs to the boundary $\partial C$, and
\emph{real boundary} of $\calC$
each connected segment of the set $\RR \cap \partial C$.
Observe that one arc boundary of $\calC$ covers all the other boundaries of $\calC$:
we call that boundary \emph{parent boundary} of $\calC$, and the other
arc boundaries of $\calC$ are called \emph{children boundaries} of $\calC$.

Finally, we say that $\calC$ is an \emph{upper} cell it is contained in the upper half-plane
$\{z \in \CC : \mathrm{Im}(z) \geq 0\}$, and a \emph{lower} cell otherwise.
\end{dfn}

Figure~\ref{fig:cell-boundaries-1} shows a cell of some lamination, as well as its (arc and real) boundaries.

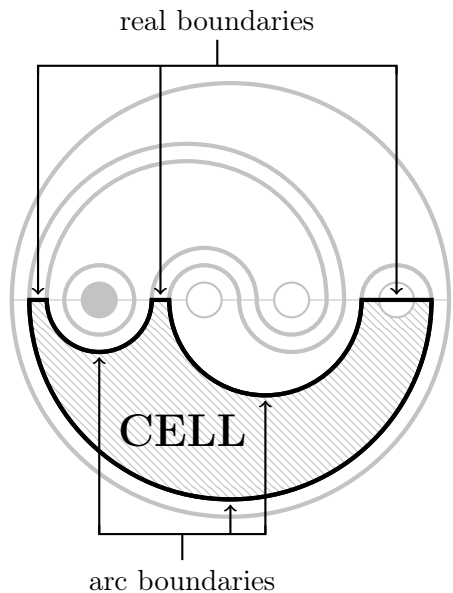
\begin{figure}[!ht]
\begin{center}
\begin{tikzpicture}[scale=0.23]
\draw[draw=black,pattern=north west lines, pattern color=gray,ultra thick] (28,0) arc (360:180:11.5);

\draw[draw=black,fill=white,ultra thick] (12,0) arc (360:180:3);
\draw[draw=black,fill=white,ultra thick] (13,0) arc (180:360:5.5);

\draw[draw=gray] (4,0) -- (29,0);

\draw[fill=gray,draw=gray,thick] (9,0) circle (1);

\draw[draw=gray,ultra thick] (11,0) arc (0:360:2);
\draw[draw=gray,ultra thick] (22,0) arc (0:180:8);
\draw[draw=gray,ultra thick] (23,0) arc (0:180:9);
\draw[draw=gray,ultra thick] (29,0) arc (0:360:12.5);

\draw[draw=gray,ultra thick] (12,0) arc (180:0:3);
\draw[draw=gray,ultra thick] (13,0) arc (180:0:2);
\draw[draw=gray,ultra thick] (24,0) arc (180:0:2);

\draw[draw=gray,ultra thick] (17,0) arc (180:360:3);
\draw[draw=gray,ultra thick] (18,0) arc (180:360:2);

\draw[fill=white,draw=gray,thick] (15,0) circle (1);
\draw[fill=white,draw=gray,thick] (20,0) circle (1);
\draw[fill=white,draw=gray,thick] (26,0) circle (1);

\draw[draw=black,ultra thick] (28,0) arc (360:180:11.5) -- (6,0) arc (180:360:3) -- (13,0) arc (180:360:5.5) -- cycle;

\draw[->,draw=black,thick] (5.5,13.5) -- (5.5,0.3);
\draw[->,draw=black,thick] (12.5,13.5) -- (12.5,0.3);
\draw[->,draw=black,thick] (26,13.5) -- (26,0.3);
\draw[draw=black,thick] (5.5,13) -- (5.5,13.5) -- (26,13.5) -- (26,13);
\draw[draw=black,thick] (15.75,13.5) -- (15.75,15);

\draw[->,draw=black,thick] (16.5,-13.5) -- (16.5,-11.8);
\draw[->,draw=black,thick] (9,-13.5) -- (9,-3.3);
\draw[->,draw=black,thick] (18.5,-13.5) -- (18.5,-5.8);
\draw[draw=black,thick] (9,-13) -- (9,-13.5) -- (18.5,-13.5) -- (18.5,-13);
\draw[draw=black,thick] (13.75,-13.5) -- (13.75,-15);

\node at (13.75,-7.5) {\LARGE \textbf{CELL}};
\node[anchor=south] at (15.75,15) {real boundaries};
\node[anchor=north] at (13.75,-15) {arc boundaries};
\end{tikzpicture}
\end{center}
\caption{A cell and its boundaries}
\label{fig:cell-boundaries-1}
\end{figure}

\begin{dfn}{Cell map}
Let $\calL$ be a tight lamination.
The \emph{cell map} of the lamination $\calL$
is the bipartite planar map (i.e. an embedding of a
planar graph into the plane) obtained as follows:
\begin{itemize}
\item inside each cell $\calC$ of $\calL$, we draw a vertex $v_\calC$;
\item for each cell $\calC$ of $\calL$ and each real boundary $B$ of $\calC$,
we draw, inside the cell $\calC$ itself, one half-edge between the vertex $v_\calC$ and the midpoint of the real boundary $B$,
so that the half-edges drawn inside of $\calC$ do not cross each other;
\item each real boundary $B$ belongs to one upper cell $\calC$ and one lower cell $\calC'$:
we merge the half-edges that link the midpoint of $B$ to the vertices $v_\calC$ and $v_{\calC'}$,
thereby obtaining one edge between $v_\calC$ and $v_{\calC'}$.
\end{itemize}
\end{dfn}

Note that the cell map is not supposed to be connected nor simple (and, actually,
is \emph{never} connected nor simple), although it does not contain loops.

Figure~\ref{fig:cell-map} shows the cell map of the lamination $\calL$,
whose edges are black lines and whose vertices are white circles.
We highlighted one of the connected components of the map
by drawing its edges with dotted lines (whereas the other edges are drawn with plain lines).

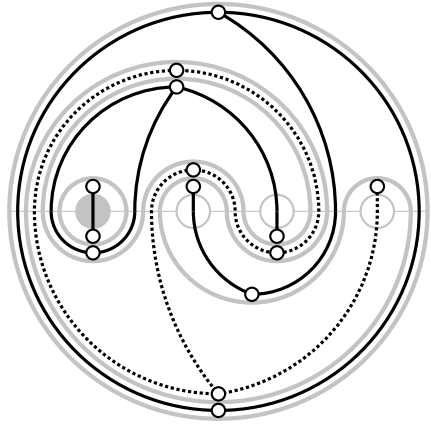
\begin{figure}[!ht]
\begin{center}
\begin{tikzpicture}[scale=0.22]
\draw[draw=gray] (4,0) -- (29,0);

\draw[fill=gray,draw=gray,thick] (9,0) circle (1);

\draw[draw=gray,ultra thick] (11,0) arc (0:360:2);
\draw[draw=gray,ultra thick] (22,0) arc (0:180:8);
\draw[draw=gray,ultra thick] (23,0) arc (0:180:9);
\draw[draw=gray,ultra thick] (29,0) arc (0:360:12.5);

\draw[draw=gray,ultra thick] (12,0) arc (180:0:3);
\draw[draw=gray,ultra thick] (13,0) arc (180:0:2);
\draw[draw=gray,ultra thick] (24,0) arc (180:0:2);

\draw[draw=gray,ultra thick] (12,0) arc (360:180:3);
\draw[draw=gray,ultra thick] (13,0) arc (180:360:5.5);
\draw[draw=gray,ultra thick] (17,0) arc (180:360:3);
\draw[draw=gray,ultra thick] (18,0) arc (180:360:2);
\draw[draw=gray,ultra thick] (28,0) arc (360:180:11.5);

\draw[fill=white,draw=gray,thick] (15,0) circle (1);
\draw[fill=white,draw=gray,thick] (20,0) circle (1);
\draw[fill=white,draw=gray,thick] (26,0) circle (1);

\draw[draw=black,very thick] (28.5,0) arc (0:360:12);
\draw[draw=black,very thick] (6.5,0) arc (180:360:2.5);
\draw[draw=black,very thick] (23.5,0) arc (360:270:5);
\draw[draw=black,very thick] (6.5,0) arc (180:90:7.5);
\draw[draw=black,very thick] (9,1.5) -- (9,-1.5);
\draw[draw=black,very thick] (15,1.5) -- (15,0);
\draw[draw=black,very thick] (20,-1.5) -- (20,0);
\draw[draw=black,very thick,densely dotted] (26,1.5) -- (26,0)
arc (0:-81.63:11.12) arc (270:180:11) arc (180:0:8.5) arc (360:180:2.5) arc (0:180:2.5) arc (-180:-140:17.13);

\draw[draw=black,very thick] (11.5,0) arc (180:143.13:12.5);
\draw[draw=black,very thick] (20,0) arc (0:77.32:7.69);
\draw[draw=black,very thick] (15,0) arc (-180:-110.02:5.32);
\draw[draw=black,very thick] (23.5,0) arc (0:60.51:13.79);

\bvertex{9}{1.5}
\bvertex{9}{-1.5}
\bvertex{14}{7.5}
\bvertex{14}{8.5}
\bvertex{16.5}{12}
\bvertex{16.5}{-12}
\bvertex{15}{2.5}
\bvertex{15}{1.5}
\bvertex{26}{1.5}
\bvertex{9}{-2.5}
\bvertex{18.5}{-5}
\bvertex{20}{-2.5}
\bvertex{20}{-1.5}
\bvertex{16.5}{-11}
\end{tikzpicture}
\end{center}
\caption{Cell map of a tight lamination}
\label{fig:cell-map}
\end{figure}

The cell map is not connected,
but the relative positions of its connected components is important.
Indeed, the topological properties of tight laminations are reflected on their cell maps and arc trees.

\begin{lem}
Let $\calL$ be a tight lamination.
No vertex of its cell map has degree $4$ or more.
\label{lem:lasso-map}
\end{lem}

\begin{proof}
Let $\calL_0,\ldots,\calL_n$ be the $n+1$ curves of the lamination $\calL$,
and let $\calM$ be the cell map of $\calL$.
Then, let $\calZ_0$ be the inner area defined by the curve $\calL_0$
and, for $1 \leq i \leq n$, let $\calZ_i$ be the area enclosed between the curves $\calL_{i-1}$ and $\calL_i$.
First, each area $\calZ_i$ is connected, hence the set $C_i := \{v_\calC : \calC \subseteq \calZ_i\}$ is a
connected subset of $\calM$.
Second, two cells $\calC$ and $\calC'$ that belong to two distinct areas $\calZ_i$ and $\calZ_j$ cannot have any common real boundary,
i.e. the vertices $v_\calC$ and $v_{\calC'}$ cannot be neighbors.
Therefore, the sets $C_i$ are the connected components of $\calM$.

Since $\|\calL\|$ is minimal, the curve $\calL_0$ must consist of two bigons with the same real projection,
hence $C_0$ consists of two vertices and one edge.
In addition, each component $C_i$ with $1 \leq i \leq n$ must contain one unique cycle,
which encloses the components $C_0,\ldots,C_{i-1}$:
indeed, by construction of the areas $\calZ_j$,
the set $\bigcup_{j \leq i-1} \calZ_j$ is the unique ``hole'' in the area $\calZ_i$.

Finally, for $1 \leq i \leq n$ and for each vertex $v_\calC \in C_i$ of degree $1$,
the cell $\calC$ has exactly one arc boundary, which must be a bigon, and exactly one real boundary,
which must contain the unique puncture of $\calZ_i$.
If $v_{\calC'}$ is another vertex of $C_i$ with degree $1$, then $\calC$ and $\calC'$ must therefore share the same real boundary.
Hence, $C_i$ consists of the two vertices $v_\calC$ and $v_{\calC'}$ of degree $1$, which contradicts the fact that $C_i$ must contain one cycle.
This proves that $C_i$ contains one unique cycle and at most one vertex with degree $1$,
and it follows that $C_i$ cannot contain any vertex with degree $4$ or more,
which completes the proof.
\end{proof}

\begin{cor}
Let $\calL$ be a tight lamination and let $\calA$ be an arc of $\calL$.
The arc $\calA$ has at most three children, including at most two arcs.
\label{cor:2-arc-children-or-less}
\end{cor}

\begin{proof}
First, since every two punctures of $\calL$ must be separated by some curve $\calL_j$ of $\calL$,
it comes at once that at most one puncture may be a child of $\calA$.
Then, let $\calC$ be the cell of $\calL$ whose parent boundary is $\calA$,
and let $\calB_1, \ldots, \calB_k$ be the arcs of $\calL$ that are children of $\calA$.
Since the cell $\calC$ is of degree $k+1$, Lemma~\ref{lem:lasso-map} proves that $k \leq 2$, which completes the proof.
\end{proof}

\begin{proof}[Proof of Proposition~\ref{pro:automata-state}]
Proposition~\ref{pro:caracterisation-extension} implies that
knowing $\pi_\beta$ is sufficient to check whether $\bR(\beta \lambda) = \lambda^{-1}$.
Hence, knowing $\pi^2_\beta$ is also sufficient, which proves the first part of Proposition~\ref{pro:automata-state}.

We prove now the second part of Proposition~\ref{pro:automata-state} when $\lambda = [k \curvearrowright \ell]$.
First, if $\beta \neq \varepsilon$, note that $\pi_\beta^2 \neq \pi_\varepsilon^2$.
Indeed, the rightmost index of $\beta$ is some positive integer $i$, which means that either
$\pi_\beta(i,+,\uparrow) = \{i\}$ or $\pi_\beta(i,+,\downarrow) = \{i\}$, whereas
$\pi_\varepsilon(i,+,\updownarrow) = \{0,\ldots,i\}$.
Therefore, we may already define the partial function $\comp(\pi^2_\varepsilon,\cdot) : \lambda \mapsto \pi^2_{\lambda}$,
and there remains to build an appropriate function $\comp$ on pairs ($\pi^2_\beta,\lambda)$ such that $\beta \neq \varepsilon$.

Since $\bR(\beta \lambda) = \lambda^{-1}$, we know that neither $\beta$ nor $\beta \lambda$ is trivial:
let $\calL$ and $\oL$ be the (non-trivial) tight laminations that represent respectively the braids $\beta$ and $\beta \lambda$.
Observe that $\calL$ satisfies the requirements $1$ to $5$ of Lemmas~\ref{lem:condition-1}.

In addition, consider the functions $\psi : i \mapsto i - \mathbf{1}_{k < i}$, $\ol{\psi} : i \mapsto i - \mathbf{1}_{k < i \leq \ell}$ and
\[\arraycolsep=0.1pt\def\arraystretch{1.2}\begin{array}{lll}
\Psi & : I & \mapsto \{x : \psi(\min I) \leq x \leq \max I\} \text{ if } I \neq \emptyset \text{, or } \emptyset \text{ if } I = \emptyset; \\
\Psi^\ast & : I & \mapsto \{x : \psi(\min I) \leq x \leq \psi(\max I)\} \text{ if } I \neq \emptyset \text{, or } \emptyset \text{ if } I = \emptyset; \\
\Theta^\uparrow & : (I,J) & \mapsto (\Psi(I),\Psi(J)) \text{ if } \{k,\ell\} \subseteq I \text{, or } (\Psi^\ast(I),\emptyset) \text{ if } \{k,\ell\} \not\subseteq I; \\
\Theta^\downarrow & : (I,J) & \mapsto (\Psi(I),\Psi(J)) \text{ if } \{k,\ell\} \subseteq J \text{, or } (\Psi^\ast(I),\emptyset) \text{ if } \{k,\ell\} \not\subseteq J,
\end{array}\]
where $I$ and $J$ are subintervals of $\{0,\ldots,n\}$.
The functions $\psi$, $\ol{\psi}$, $\Psi$, $\Psi^\ast$, $\Theta^\uparrow$ and $\Theta^\downarrow$ will play a crucial role in computing $\pi^2_{\beta\lambda}$.
Intuitively, the functions $\psi$ and $\ol{\psi}$ are meant to reflect the fact that some punctures of $\calL$ and $\oL$ have different names:
$p_i = \op_{\psi(i)} = \op_{\ol{\psi}(i)}$ if $i < k$ or if $k < i \leq \ell$, and $p_i = \op_i = \op_{\ol{\psi}(i)}$ if $\ell < i$.

Then, remember how the lamination $\oL$ was drawn in the proof of Lemma~\ref{lem:condition-2}.
We split the set $\upbA$ of upper arcs of $\calL$ into six subsets $\Omega_1$, $\Omega_2$, $\Omega_4$, $\Omega_6$, $\Omega_7$ and $\Omega_8$,
replaced each arc $A_j^1 \in \Omega_1$ by three arcs $A_j^{1,1}$, $A_j^{1,2}$ and $A_j^{1,3}$, then
replaced the puncture $p_k$ by a new puncture $\op_\ell$, and did not modify any other arc or puncture.

For each arc $A_i^1$ of $\calL$ that belongs to $\Omega_1$, we have $\pi_\calL(A_i^1) = \{u,\ldots,v\}$ for some $u$ and $v$ (that depend on $i$).
It comes immediately that $\pi_\oL(A_i^{1,1}) = \{v,\ldots,\ell-1\}$, $\pi_\oL(A_i^{1,2}) = \{\ell\}$ and $\pi_\oL(A_i^{1,3}) = \{u,\ldots,\ell\}$.

Now, let $\calA$ be an arc of both $\calL$ and $\oL$, i.e. an arc of $\calL$ that does not belong to $\Omega_1$,
and let $e < E$ be the endpoints of $\calA$.
Lemma~\ref{lem:left-endpoint-is-left} proves that $e < \op_\ell$.
It follows that $\pi^2_{\oL}(\calA) = (\Psi^\ast(\pi_\calL(\calA)),\emptyset)$ if $E < \op_\ell$, and
that $\pi_{\oL}(\calA) = \Psi(\pi_\calL(\calA))$ if $E > \op_\ell$.

Furthermore, if $E > \op_\ell$, let $I, J$ be subintervals of $\{0,\ldots,n\}$ such that $\pi^2_\calL(\calA) = (I,J)$,
and let $p_m$ be the rightmost puncture covered by a bigon of $\calL$. We prove now that $J$ is non-empty.
Due to the requirement $3$ of Lemma~\ref{lem:condition-1}, we know that $m \leq \ell+1$.
If $m = \ell+1$, then the interval $(\op_\ell,p_{\ell+1})$ contains no endpoint of any arc of $\calL$, and therefore $E > p_{\ell+1}$, so that $J \neq \emptyset$;
if $m \leq \ell$, then of course $J \neq \emptyset$ as well.

Finally, let $\calB$ be the arc of $\calL$ with which $\calA$ shares its right endpoint $E$.
If $\calA$ is an upper arc, then $E > \op_\ell$ if and only if $\calA \in \Omega_4 \cup \Omega_6$, i.e. if and only if $\{k,\ell\} \subseteq \pi_\calL(\calA) = I$.
Consequently, if $\calA$ is a lower arc, then $E > \op_\ell$ if and only if $\{k,\ell\} \subseteq \pi_\calL(\rext(\calA)) = J$.
Overall, it follows that $\pi^2_{\oL}(A) = \Theta^\uparrow(\pi^2_\calL(\calA))$ for all upper arcs $\calA \notin \Omega_1$ of $\calL$, and that
$\pi^2_{\oL}(\calA) = \Theta^\downarrow(\pi^2_\calL(\calA))$ for all lower arcs $\calA$ of $\calL$.

The last remaining challenge is to identify the neighboring arcs of the punctures of $\oL$.
We do it, following a long enumeration of cases,
and thereby we compute $\pi^2_{\beta\lambda}$ as a function of $\pi^2_{\beta}$ and of $\lambda$, as follows.

\begin{enumerate}
\item Let $(\diamond,u,v)$ the unique triple in $\{(+,0,k)\} \cup \{(-,k,z) : k \leq z < \ell\}$
such that $\pi_\beta(k,\diamond,\uparrow) = \{u,\ldots,v\}$.
The arc $A_1^1$ is such that $\pi_\oL(A_1^1) = \{u,\ldots,v\}$, whence

\begin{eqnarray*}
\pi^2_{\beta\lambda}(\ell,+,\uparrow) & = & (\{u,\ldots,\ell\},\{\ell\}); \\
\pi^2_{\beta\lambda}(\ell,-,\uparrow) & = & (\{v,\ldots,\ell-1\},\emptyset); \\
\pi^2_{\beta\lambda}(\ell,\downarrow,\pm) & = & (\{\ell\},\{u,\ldots,\ell\}).
\end{eqnarray*}

\item First, recall that $\upA_+(p_{\ell+1},\calL) \in \Omega_6$.
Hence, let $x$ be the integer such that ${\pi_\beta(\ell+1,+,\uparrow)} = \{x,\ldots,\ell+1\}$:
we have $x \leq k$.
If $k \in \pi_\beta(\ell+1,-,\uparrow)$, then $p_{\ell+1} = \op_{\ell+1}$ has the same neighbor arcs in $\calL$ and in $\oL$.

If $k \notin \pi_\beta(\ell+1,-,\uparrow)$, then Figure~\ref{fig:compute-hard-case-thm-4.8} illustrates the case where $x < k$
(the case where $x = k$ is analogous).
First, observe that $\Omega_4 = \emptyset$.
Since some arcs must separate the punctures $p_k$, $p_\ell$ and $p_{\ell+1}$ in $\calL$, it follows that $\Omega_1$ and $\Omega_2$ are non-empty.
Consequently, the neighbor arcs of the puncture $\op_{\ell+1} = p_{\ell+1}$ in $\oL$ are

\begin{center}
\begin{tabular}{ll}
$\upA_+(\op_{\ell+1},\oL) = \upA_+(p_{\ell+1},\calL) = A_1^6$, & $\lowA_+(\op_{\ell+1},\oL) = \lowA_+(p_{\ell+1},\calL)$, \\
$\upA_-(\op_{\ell+1},\oL) = A_{\omega_1}^{1,3}$, & $\lowA_-(\op_{\ell+1},\oL) = A_{\omega_1}^{1,2}$,
\end{tabular}
\end{center}

Moreover, since $\Omega_4 = \emptyset$, and using Corollary~\ref{cor:2-arc-children-or-less},
the children of the arc $A_1^6$ in $\calL$ must be, from left to right:
the arc $A_{\omega_1}^1$, the arc $A_{\omega_2}^2$, and the puncture $p_{\ell+1}$.
Since $x \leq k$, it follows that $x = \min \pi_\calL(A_1^6) = \min \pi_\calL(A_{\omega_1}^1) = \min \pi_{\oL}(A_{\omega_1}^{1,3})$.

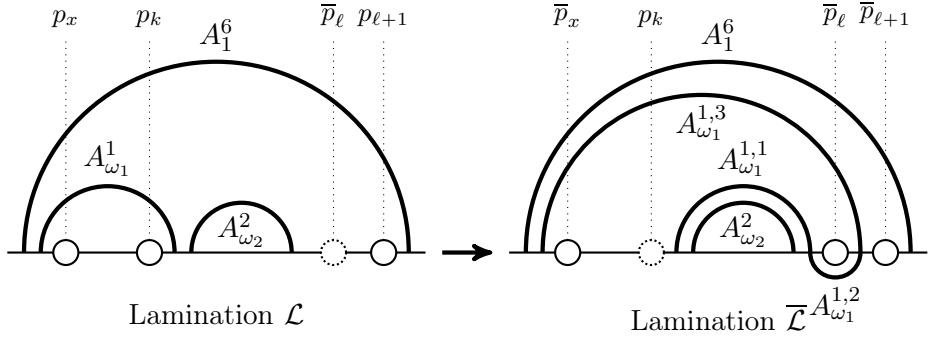
\begin{figure}[!ht]
\begin{center}
\begin{tikzpicture}[scale=0.22]
\draw[draw=black,thick] (0,0) -- (25,0);

\draw[draw=black,ultra thick] (1,0) arc (180:0:11.5);
\draw[draw=black,ultra thick] (2,0) arc (180:0:4);
\draw[draw=black,ultra thick] (11,0) arc (180:0:3);

\draw[draw=black,dotted] (3.5,0) -- (3.5,13);
\draw[draw=black,dotted] (8.5,0) -- (8.5,13);
\draw[draw=black,dotted] (19.5,0) -- (19.5,13);
\draw[draw=black,dotted] (22.5,0) -- (22.5,13);

\draw[fill=white,draw=black,thick] (3.5,0) circle (0.75);
\node[anchor=south] at (3.5,13) {\small$p_x$};
\draw[fill=white,draw=black,thick] (8.5,0) circle (0.75);
\node[anchor=south] at (8.5,13) {\small$p_k$};
\draw[fill=white,draw=black,thick,densely dotted] (19.5,0) circle (0.75);
\node[anchor=south] at (19.5,13) {\small$\op_\ell$};
\draw[fill=white,draw=black,thick] (22.5,0) circle (0.75);
\node[anchor=south] at (22.5,13) {\small$p_{\ell+1}$};

\draw[->,>=stealth',draw=black,ultra thick] (26,0) -- (29,0);

\draw[draw=black,thick] (30,0) -- (55,0);

\draw[draw=black,ultra thick] (31,0) arc (180:0:11.5);
\draw[draw=black,ultra thick] (32,0) arc (180:0:9.5) arc (360:180:1.5) arc (0:180:4);
\draw[draw=black,ultra thick] (41,0) arc (180:0:3);

\draw[draw=black,dotted] (33.5,0) -- (33.5,13);
\draw[draw=black,dotted] (38.5,0) -- (38.5,13);
\draw[draw=black,dotted] (49.5,0) -- (49.5,13);
\draw[draw=black,dotted] (52.5,0) -- (52.5,13);

\draw[fill=white,draw=black,thick] (33.5,0) circle (0.75);
\node[anchor=south] at (33.5,13) {\small$\op_x$};
\draw[fill=white,draw=black,thick,densely dotted] (38.5,0) circle (0.75);
\node[anchor=south] at (38.5,13) {\small$p_k$};
\draw[fill=white,draw=black,thick] (49.5,0) circle (0.75);
\node[anchor=south] at (49.5,13) {\small$\op_\ell$};
\draw[fill=white,draw=black,thick] (52.5,0) circle (0.75);
\node[anchor=south] at (52.5,13) {\small$\op_{\ell+1}$};

\node[anchor=south] at (12.5,11.5) {$A_1^6$};
\node[anchor=south] at (6,4) {$A_{\omega_1}^1$};
\node[anchor=north] at (14,3) {$A_{\omega_2}^2$};
\node[anchor=south] at (42.5,11.5) {$A_1^6$};
\node[anchor=south] at (44,4) {$A_{\omega_1}^{1,1}$};
\node[anchor=north] at (44,3) {$A_{\omega_2}^2$};
\node[anchor=north] at (49.5,-1.5) {$A_{\omega_1}^{1,2}$};
\node[anchor=north] at (41.5,9.5) {$A_{\omega_1}^{1,3}$};

\node[anchor=north] at (12.5,-2.5) {Lamination $\calL$};
\node[anchor=north] at (42.5,-2.5) {Lamination $\oL$};
\end{tikzpicture}
\end{center}
\caption{Computing $\pi_{\beta\lambda}^2(\ell+1,\pm,\updownarrow)$ when $k \notin \pi_\beta(\ell+1,-,\uparrow)$ --- assuming $z < k$}
\label{fig:compute-hard-case-thm-4.8}
\end{figure}

Adding these two cases, we obtain 
\vspace{-7mm}

\begin{eqnarray*}
\pi^2_{\beta\lambda}(\ell+1,+,\uparrow) & = & \Theta^\uparrow(\pi^2_\beta(\ell+1,+,\uparrow)); \\
\pi^2_{\beta\lambda}(\ell+1,+,\downarrow) & = & \Theta^\downarrow(\pi^2_\beta(\ell+1,+,\downarrow)); \\
\pi^2_{\beta\lambda}(\ell+1,-,\uparrow) & = & \Theta^\uparrow(\pi^2_\beta(\ell+1,-,\uparrow)) \text{ if } k \in \pi_\beta(\ell+1,-,\uparrow) \\
& & (\{z,\ldots,\ell\},\{\ell\}) \text{ if } k \notin \pi_\beta(\ell+1,-,\uparrow); \\
\pi^2_{\beta\lambda}(\ell+1,-,\downarrow) & = & \Theta^\downarrow(\pi^2_\beta(\ell+1,-,\downarrow)) \text{ if } k \in \pi_\beta(\ell+1,-,\uparrow) \\
& & (\{\ell\},\{z,\ldots,\ell\}) \text{ if } k \notin \pi_\beta(\ell+1,-,\uparrow).
\end{eqnarray*}

\item Observe that $\upA_+(p_\ell,\calL) \in \Omega_2 \cup \Omega_4$ and that
$\upA_-(p_\ell,\calL) \in \Omega_1 \cup \Omega_2 \cup \Omega_8$.
In addition, either $\upA_+(p_\ell,\calL) = \upA_-(p_\ell,\calL)$,
or $\upA_+(p_\ell,\calL)$ is the parent of $\upA_-(p_\ell,\calL)$.
Note that the former case arises if and only if $p_\ell$ belongs to an upper bigon of $\calL$ or,
equivalently, if $\upA_-(p_\ell,\calL) \in \Omega_2$.
Hence, let $y$ be the integer such that $\pi_\beta(\ell,+,\uparrow) = \{y,\ldots,\ell\}$.
In addition, if $\upA_-(p_\ell,\calL) \in \Omega_1 \cup \Omega_8$, let $z$ be the integer such that
$\pi_\beta(\ell,-,\uparrow) = \{z,\ldots,\ell-1\}$.

If $k \notin \pi_\beta(\ell,+,\uparrow)$, then $\upA_+(p_\ell,\calL) \in \Omega_2$ and
$\upA_-(p_\ell,\calL) \in \Omega_2 \cup \Omega_8$.
In this case, the puncture $p_\ell = \op_{\ell-1}$ has the same neighbor arcs in $\calL$ and in $\oL$.

If $k \in \pi_\beta(\ell,-,\uparrow)$, then $\upA_-(p_\ell,\calL) \in \Omega_1$ and
$\upA_+(p_\ell,\calL) \in \Omega_4$. It follows that
$\upA_-(p_\ell,\calL) = A_{\omega_1}^1$ and that $\Omega_2 = \emptyset$,
which shows that $\op_{\ell-1}$ belongs to an upper bigon of $\oL$.
Consequently, the neighbor arcs of the puncture $\op_{\ell-1} = p_\ell$ in $\oL$ are
\begin{center}
\begin{tabular}{ll}
$\upA_+(\op_{\ell-1},\oL) = A_{\omega_1}^{1,1}$, & $\lowA_+(\op_{\ell-1},\oL) = A_{\omega_1}^{1,2}$, \\
$\upA_-(\op_{\ell-1},\oL) = A_{\omega_1}^{1,1}$, & $\lowA_-(\op_{\ell-1},\oL) = \lowA_-(p_\ell,\calL)$.
\end{tabular}
\end{center}

Moreover, since $\upA_-(p_\ell,\calL) \in \Omega_1$, the integer $z$ is well-defined,
and satisfies the inequality $z \leq k$:
Figure~\ref{fig:compute-hard-case-thm-4.8-IIa} illustrates the case where $z < k$. It follows that
$z = \min \pi_\calL(A_{\omega_1}^1) = \min \pi_{\oL}(A_{\omega_1}^{1,3})$.

Finally, if $k \in \pi_\beta(\ell,+,\uparrow)$ and $k \notin \pi_\beta(\ell,-,\uparrow)$, then
$\upA_+(p_\ell,\calL) \in \Omega_4$ and $\upA_-(p_\ell,\calL) \in \Omega_8$.
It follows that $\Omega_2 = \emptyset$.
Consequently, the neighbor arcs of the puncture $\op_{\ell-1} = p_\ell$ in $\oL$ are
\begin{center}
\begin{tabular}{ll}
$\upA_+(\op_{\ell-1},\oL) = A_{\omega_1}^{1,1}$, & $\lowA_+(\op_{\ell-1},\oL) = A_{\omega_1}^{1,2}$, \\
$\upA_-(\op_{\ell-1},\oL) = \upA_-(p_\ell,\calL)$, & $\lowA_-(\op_{\ell-1},\oL) = \lowA_-(p_\ell,\calL)$,
\end{tabular}
\end{center}

Therefore, and using Corollary~\ref{cor:2-arc-children-or-less}, the children of the arc $\upA_+(p_\ell,\calL) = A_1^4$ in $\calL$
must be, from left to right: the arc $A_{\omega_1}^1$, the arc $\upA_-(p_\ell,\calL)$, and the puncture $p_\ell$.
Hence, both integers $y$ and $z$ are well-defined, and they satisfy the inequalities $y \leq k < z$:
Figure~\ref{fig:compute-hard-case-thm-4.8-IIb} illustrates the case where $y < k$.
It follows that
\begin{center}
\begin{tabular}{l}
$y = \min \pi_\calL(A_1^4) = \min \pi_\calL(A_{\omega_1}^1) = \min \pi_{\oL}(A_{\omega_1}^{1,3})$ and that \\
$z-1 = \min \pi_\calL(\upA_-(p_\ell,\calL))-1 = \max \pi_\calL(A_{\omega_1}^1) = \min \pi_{\oL}(A_{\omega_1}^{1,1})$.
\end{tabular}
\end{center}

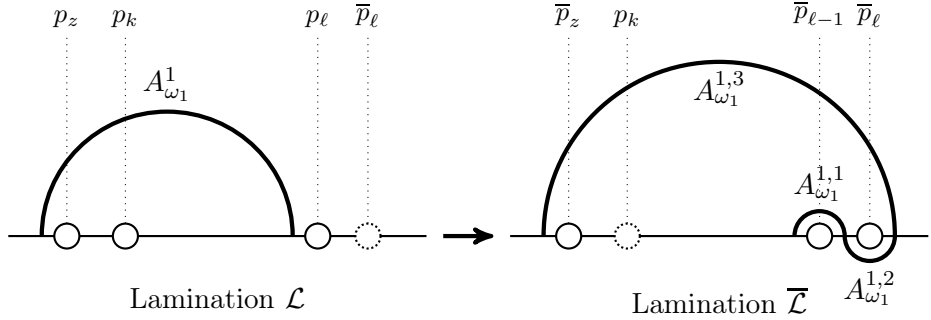
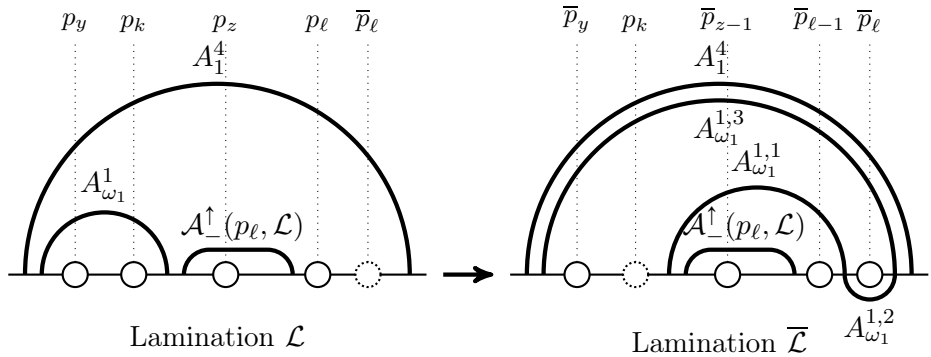
\begin{figure}[!ht]
\begin{center}
\begin{subfigure}[b]{\textwidth}
\begin{center}
\begin{tikzpicture}[scale=0.22]
\draw[draw=black,thick] (0,0) -- (25,0);

\draw[draw=black,ultra thick] (2,0) arc (180:0:7.5);

\draw[draw=black,dotted] (3.5,0) -- (3.5,12);
\draw[draw=black,dotted] (7,0) -- (7,12);
\draw[draw=black,dotted] (18.5,0) -- (18.5,12);
\draw[draw=black,dotted] (21.5,0) -- (21.5,12);

\draw[fill=white,draw=black,thick] (3.5,0) circle (0.75);
\node[anchor=south] at (3.5,12) {\small$p_z$};
\draw[fill=white,draw=black,thick] (7,0) circle (0.75);
\node[anchor=south] at (7,12) {\small$p_k$};
\draw[fill=white,draw=black,thick] (18.5,0) circle (0.75);
\node[anchor=south] at (18.5,12) {\small$p_\ell$};
\draw[fill=white,draw=black,thick,densely dotted] (21.5,0) circle (0.75);
\node[anchor=south] at (21.5,12) {\small$\op_\ell$};

\draw[->,>=stealth',draw=black,ultra thick] (26,0) -- (29,0);

\draw[draw=black,thick] (30,0) -- (55,0);

\draw[draw=black,ultra thick] (32,0) arc (180:0:10.5) arc (360:180:1.5) arc (0:180:1.5);

\draw[draw=black,dotted] (33.5,0) -- (33.5,12);
\draw[draw=black,dotted] (37,0) -- (37,12);
\draw[draw=black,dotted] (48.5,0) -- (48.5,12);
\draw[draw=black,dotted] (51.5,0) -- (51.5,12);

\draw[fill=white,draw=black,thick] (33.5,0) circle (0.75);
\node[anchor=south] at (33.5,12) {\small$\op_z$};
\draw[fill=white,draw=black,thick,densely dotted] (37,0) circle (0.75);
\node[anchor=south] at (37,12) {\small$p_k$};
\draw[fill=white,draw=black,thick] (48.5,0) circle (0.75);
\node[anchor=south] at (48.5,12) {\small$\op_{\ell-1}$};
\draw[fill=white,draw=black,thick] (51.5,0) circle (0.75);
\node[anchor=south] at (51.5,12) {\small$\op_\ell$};

\node[anchor=south] at (9.5,7.5) {$A_{\omega_1}^1$};
\node[anchor=south] at (48.5,1.5) {$A_{\omega_1}^{1,1}$};
\node[anchor=north] at (51.5,-1.5) {$A_{\omega_1}^{1,2}$};
\node[anchor=north] at (42.5,10.5) {$A_{\omega_1}^{1,3}$};

\node[anchor=north] at (12.5,-2.5) {Lamination $\calL$};
\node[anchor=north] at (42.5,-2.5) {Lamination $\oL$};
\end{tikzpicture}
\end{center}
\caption{Case \#1: $k\in \pi_\beta(\ell,-,\uparrow)$ --- assuming $z < k$}
\label{fig:compute-hard-case-thm-4.8-IIa}
\end{subfigure}

\medskip

\begin{subfigure}[b]{\textwidth}
\begin{center}
\begin{tikzpicture}[scale=0.22]
\draw[draw=black,thick] (0,0) -- (25,0);

\draw[draw=black,ultra thick] (1,0) arc (180:0:11.5);
\draw[draw=black,ultra thick] (2,0) arc (180:0:3.75);
\draw[draw=black,ultra thick] (10.5,0) arc (180:90:1.5) -- (15.5,1.5) arc (90:0:1.5);

\draw[draw=black,dotted] (4,0) -- (4,14);
\draw[draw=black,dotted] (7.5,0) -- (7.5,14);
\draw[draw=black,dotted] (13,0) -- (13,14);
\draw[draw=black,dotted] (18.5,0) -- (18.5,14);
\draw[draw=black,dotted] (21.5,0) -- (21.5,14);

\draw[fill=white,draw=black,thick] (4,0) circle (0.75);
\node[anchor=south] at (4,14) {\small$p_y$};
\draw[fill=white,draw=black,thick] (7.5,0) circle (0.75);
\node[anchor=south] at (7.5,14) {\small$p_k$};
\draw[fill=white,draw=black,thick] (13,0) circle (0.75);
\node[anchor=south] at (13,14) {\small$p_z$};
\draw[fill=white,draw=black,thick] (18.5,0) circle (0.75);
\node[anchor=south] at (18.5,14) {\small$p_\ell$};
\draw[fill=white,draw=black,thick,densely dotted] (21.5,0) circle (0.75);
\node[anchor=south] at (21.5,14) {\small$\op_\ell$};

\draw[->,>=stealth',draw=black,ultra thick] (26,0) -- (29,0);

\draw[draw=black,thick] (30,0) -- (55,0);

\draw[draw=black,ultra thick] (31,0) arc (180:0:11.5);
\draw[draw=black,ultra thick] (32,0) arc (180:0:10.5) arc (360:180:1.5) arc (0:180:5.25);
\draw[draw=black,ultra thick] (40.5,0) arc (180:90:1.5) -- (45.5,1.5) arc (90:0:1.5);

\draw[draw=black,dotted] (34,0) -- (34,14);
\draw[draw=black,dotted] (37.5,0) -- (37.5,14);
\draw[draw=black,dotted] (43,0) -- (43,14);
\draw[draw=black,dotted] (48.5,0) -- (48.5,14);
\draw[draw=black,dotted] (51.5,0) -- (51.5,14);

\draw[fill=white,draw=black,thick] (34,0) circle (0.75);
\node[anchor=south] at (34,14) {\small$\op_y$};
\draw[fill=white,draw=black,thick,densely dotted] (37.5,0) circle (0.75);
\node[anchor=south] at (37.5,14) {\small$p_k$};
\draw[fill=white,draw=black,thick] (43,0) circle (0.75);
\node[anchor=south] at (43,14) {\small$\op_{z-1}$};
\draw[fill=white,draw=black,thick] (48.5,0) circle (0.75);
\node[anchor=south] at (48.5,14) {\small$\op_{\ell-1}$};
\draw[fill=white,draw=black,thick] (51.5,0) circle (0.75);
\node[anchor=south] at (51.5,14) {\small$\op_\ell$};

\node[anchor=south] at (12,11.5) {$A_1^4$};
\node[anchor=south] at (5.75,3.75) {$A_{\omega_1}^1$};
\node[anchor=south] at (14,1.25) {$\upA_-(p_\ell,\calL)$};
\node[anchor=south] at (42,11.5) {$A_1^4$};
\node[anchor=south] at (44,1.25) {$\upA_-(p_\ell,\calL)$};
\node[anchor=south] at (44.75,5.25) {$A_{\omega_1}^{1,1}$};
\node[anchor=north] at (51.5,-1.5) {$A_{\omega_1}^{1,2}$};
\node[anchor=north] at (42.5,10.5) {$A_{\omega_1}^{1,3}$};

\node[anchor=north] at (12.5,-2.5) {Lamination $\calL$};
\node[anchor=north] at (42.5,-2.5) {Lamination $\oL$};
\end{tikzpicture}
\end{center}
\caption{Case \#2: $k \in \pi_\beta(\ell,+,\uparrow)$ and $k \notin \pi_\beta(\ell,-,\uparrow)$ --- assuming $y < k$}
\label{fig:compute-hard-case-thm-4.8-IIb}
\end{subfigure}
\end{center}
\caption{Computing $\pi_{\beta\lambda}^2(\ell-1,\pm,\updownarrow)$ when $k \in \pi_\beta(\ell,+,\uparrow)$}
\label{fig:compute-hard-case-thm-4.8-II}
\end{figure}

Adding these three cases, we obtain
\vspace{-7mm}

\begin{eqnarray*}
\pi^2_{\beta\lambda}(\ell-1,+,\uparrow) & = & \Theta^\uparrow(\pi^2_\beta(\ell,+,\uparrow)) \text{ if } k \notin \pi_\beta(\ell,+,\uparrow) \\
& & (\{\ell-1\},\emptyset) \text{ if } k \in \pi_\beta(\ell,-,\uparrow) \\
& & (\{z-1,\ldots,\ell\},\emptyset) \text{ if } k \in \pi_\beta(\ell,+,\uparrow) \text{ and } k \notin \pi_\beta(\ell,-,\uparrow); \\
\pi^2_{\beta\lambda}(\ell-1,-,\uparrow) & = & \Theta^\uparrow(\pi^2_\beta(\ell,-,\uparrow)) \text{ if } k \notin \pi_\beta(\ell,-,\uparrow) \\
& & (\{\ell-1\},\emptyset) \text{ if } k \in \pi_\beta(\ell,-,\uparrow); \\
\pi^2_{\beta\lambda}(\ell-1,+,\downarrow) & = & \Theta^\downarrow(\pi^2_\beta(\ell,+,\downarrow)) \text{ if } k \notin \pi_\beta(\ell,+,\uparrow) \\
& & (\{\ell\},\{z,\ldots,\ell\}) \text{ if } k \in \pi_\beta(\ell,-,\uparrow) \\
& & (\{\ell\},\{y,\ldots,\ell\}) \text{ if } k \in \pi_\beta(\ell,+,\uparrow) \text{ and } k \notin \pi_\beta(\ell,-,\uparrow); \\
\pi^2_{\beta\lambda}(\ell-1,-,\downarrow) & = & \Theta^\downarrow(\pi^2_\beta(\ell,-,\downarrow)).
\end{eqnarray*}

\item Let $i$ be an integer such that $i \notin \{k,\ell,\ell+1\}$. The neighbor arcs of $p_i$ in $\calL$ either are some arc $A_j^1 \in \Omega_1$
(which will be replaced by $A_j^{1,1}$ if $i > k$, or $A_j^{1,3}$ if $i < k$, when transforming $\calL$ into $\oL$) or 
are also neighbor arcs of $p_i$ in $\oL$. It follows that
\vspace{-7mm}

\begin{eqnarray*}
\pi^2_{\beta\lambda}(\ol{\psi}(i),+,\uparrow) & = & (\{i+1,\ldots,\ell\},\{\ell\})\text{ if } i < k,~ k \in \pi_\beta(i,+,\uparrow) \text{ and } \\
& & {\color{white}(\{i+1,\ldots,\ell\},\{\ell\})\text{ if } }\ell \notin \pi_\beta(i,+,\uparrow) \\
& & (\{i,\ldots,\ell-1\},\emptyset) \text{ if } k < i < \ell \text{ and } k \in \pi_\beta(i,+,\uparrow) \\
& & \Theta^\uparrow(\pi_\beta(i,+,\uparrow))\text{ otherwise}; \\
\pi^2_{\beta\lambda}(\ol{\psi}(i),-,\uparrow) & = & (\{i,\ldots,\ell\},\{\ell\})\text{ if } i < k,~ k \in \pi_\beta(i,-,\uparrow)
\text{ and } \ell \notin \pi_\beta(i,-,\uparrow) \\
& & (\{i-1,\ldots,\ell-1\},\emptyset)\text{ if } k < i < \ell \text{ and } k \in \pi_\beta(i,-,\uparrow) \\
& & \Theta^\uparrow(\pi_\beta(i,-,\uparrow)) \text{ otherwise}; \\
\pi^2_{\beta\lambda}(\ol{\psi}(i),+,\downarrow) & = & \Theta^\downarrow(\pi^2_\beta(i,+,\downarrow)); \\
\pi^2_{\beta\lambda}(\ol{\psi}(i),-,\downarrow) & = & \Theta^\downarrow(\pi^2_\beta(i,-,\downarrow)).
\end{eqnarray*}
\end{enumerate}

This disjunction of cases provides us with a complete characterization of $\pi^2_{\beta \lambda}$
as a function depending only on $\pi_\beta$ and of $\lambda$,
which completes the proof of Proposition~\ref{pro:automata-state}.
\end{proof}

\section{The Relaxation Normal Form is Automatic When $n \leq 3$}
\label{section:automatic}

We have proved that the relaxation normal form is regular.
A natural question is now that of the automaticity of this normal form.
We prove here that the relaxation normal form is synchronously biautomatic if and only if $n \leq 3$
and is not asynchronously right automatic for $n \geq 4$.

The notions of synchronous and asynchronous normal forms
can be characterized using results from~\cite{Epstein:1992:WPG:573874},
in which the results mentioned below without proof can be found.

\begin{dfn}{Difference set}
Let $G$ be a finitely presented group, and let $\textbf{NF}$ be a normal form on $G$.
Let $g$ be an element of $G$, let $K$ be a positive integer, and let $\calS$ be a subset of $G$.

For all elements $x, y \in G$, consider the words $\bx = \textbf{NF}(x)$ and $\by = \textbf{NF}(y)$,
with respective lengths $|\bx|$ and $|\by|$.
For all integers $k \leq |\bx|$ and $\ell \leq |\by|$,
let $\bx_k$ be the prefix of $\bx$ of length $k$, and let $\by_\ell$ be the prefix of $\by$ of length $\ell$,
then let $x_k$ and $y_\ell$ be elements of $G$ such that $\bx_k = \textbf{NF}(x_k)$ and $\by_\ell = \textbf{NF}(y_\ell)$.

We say that $\calS$ is an \emph{asynchronous difference set} for $\textbf{NF}$, $g$, $x$ and $y$ if
there exists non-decreasing functions $\bX : \{0,\ldots,|\bx|+|\by|\} \mapsto \{0,\ldots,|\bx|\}$ and 
$\bY : \{0,\ldots,|\bx|+|\by|\} \mapsto \{0,\ldots,|\by|\}$ such that
$\bX(k)+\bY(k) = k$ and $g x_{\bX(k)} \in y_{\bY(k)} \calS$ for all $k \leq |\bx|+|\by|$.
We further say that $\calS$ is a \emph{synchronous difference set} for $\textbf{NF}$, $g$, $x$ and $y$ if
there exists such functions such that $\bX(2k) = \bY(2k) = k$ for all $k \leq \min\{|\bx|,|\by|\}$.
\label{dfn:universal-diff-set}
\end{dfn}

\begin{thm}
Let $G$ be a finitely presented group, with neutral element $\varepsilon$.
A normal form $\textbf{NF}$ on $G$ is synchronously (respectively, asynchronously) left automatic if and only if,
for all elements $g \in G$, there exists a finite set $\calS \subseteq G$ such that,
for all $x \in G$, the set $\calS$ is a synchronous (respectively, asynchronous) difference set for $\textbf{NF}$, $g$, $x$ and $g x$.

Similarly, $\textbf{NF}$ is synchronously (respectively, asynchronously) right automatic if and only if,
for all elements $g \in G$, there exists a finite set $\calS \subseteq G$ such that,
for all $x \in G$, the set $\calS$ is a synchronous (respectively, asynchronous) difference set for $\textbf{NF}$, $\varepsilon$, $x$ and $x g$.
\label{thm:asynchronously-automatic}
\end{thm}

In the sequel, when $n \geq 4$, we first prove that the relaxation normal form is not synchronously left automatic,
then we prove that it is not asynchronously right automatic (hence, not synchronously right automatic) as well.

\begin{lem}
For all integers $k \geq 1$, the words
$\ba_k = [2 \curvearrowright 4] \cdot ([1 \curvearrowright 3] \cdot [1 \curvearrowbotright 4] \cdot [3 \curvearrowright 4])^k$ and
$\bb_k = ([2 \curvearrowright 4] \cdot [2 \curvearrowbotright 4])^k \cdot [1 \curvearrowright 4]$
are in relaxation normal form.
Furthermore, the equality $\sigma_1 \alpha_k = \beta_k$ holds, where
$\ba_k = \RNF(\alpha_k)$ and $\bb_k = \RNF(\beta_k)$.
\label{lem:not-sync-automatic}
\end{lem}

\begin{proof}
We first consider the braids $\ol{\beta}_k = ([2 \curvearrowright 4] [2 \curvearrowbotright 4])^k$
and their extended shadows.
Straightforward computations show that $\pi^2_{\alpha_2} = \pi^2_{\alpha_3}$ and $\pi^2_{\ol{\beta}_2} = \pi^2_{\ol{\beta}_3}$.
Hence, and due to Proposition~\ref{pro:automata-state},
an immediate induction on $k$ proves that
$\pi^2_{\alpha_k} = \pi^2_{\alpha_2}$ and $\pi^2_{\ol{\beta}_k} = \pi^2_{\ol{\beta}_2}$ for all $k \geq 2$.
One checks easily that $\ba_2$ and $\bb_2$ are in relaxation normal form, and it follows that
$\ba_k$ and $\bb_k$ are in relaxation normal form for all integers $k \geq 2$.

Observe that $\sigma_1 \alpha_k = \sigma_1 \sigma_2 \sigma_3 (\sigma_2^{-1} \sigma_1)^k$ and that $\beta_k = (\sigma_3^{-1} \sigma_2)^k  \sigma_1 \sigma_2 \sigma_3$.
Hence, the relation $\sigma_1 \alpha_k = \beta_k$ follows from the equality $\sigma_1 \sigma_2 \sigma_3 \sigma_2^{-1} \sigma_1 = \sigma_3^{-1} \sigma_2 \sigma_1 \sigma_2 \sigma_3$,
which is easy to check and therefore we have $\sigma_1 \alpha_k = \beta_k$ for all $k \geq 0$.
\end{proof}

\begin{cor}
In each braid group $B_n$ with $n \geq 4$ strands,
the relaxation normal form is not synchronously left automatic.
\label{cor:not-sync-automatic}
\end{cor}

\begin{proof}
Consider some integer $u \geq 1$, as well as the braids $x = \alpha_{3u}$ and $y = \beta_{3u} = \sigma_1 x$.
Using the notations of Definition~\ref{dfn:universal-diff-set}, we have $|\by| = 6u+1 \leq |\bx|$ and we also have
$x_{6u+1} = x (\sigma_2^{-1} \sigma_1)^{-u}$ and $y_{6u+1} = y$.
It follows that $y_{6u+1}^{-1} \sigma_1 x_{6u+1} = (\sigma_1 x)^{-1} \sigma_1 x (\sigma_2^{-1} \sigma_1)^{-u} = (\sigma_2^{-1} \sigma_1)^{-u}$.
No finite set can contain all powers of the braid $\sigma_2^{-1} \sigma_1$.
Hence, no finite set is a synchronous difference set for $\textbf{NF}$, $\sigma_1$, $x$ and $\sigma_1 x$ for all braids $x$, and
Theorem~\ref{thm:asynchronously-automatic} proves that the relaxation normal form is not synchronously left automatic.
\end{proof}

\begin{lem}
For all integers $k \geq 3$, the words
$\bc_k = [1 \curvearrowright 2]^k \cdot [3 \curvearrowright 4]^k$ and
$\bd_k = [3 \curvearrowright 4]^{k+1} \cdot [1 \curvearrowright 4]^2 \cdot [3 \curvearrowright 4]^{k-1}$
are in relaxation normal form.
Furthermore, the relation $\gamma_k \Delta_3 = \delta_k$ holds, where
$\bc_k = \RNF(\gamma_k)$, $\bd_k = \RNF(\delta_k)$ and $\Delta_3$ is the Garside element $\sigma_1 \sigma_2 \sigma_3 \sigma_1 \sigma_2 \sigma_1$.
\label{lem:not-automatic}
\end{lem}

\begin{proof}
We first consider the braids
$\ol{\gamma}_{k,\ell} = [1 \curvearrowright 2]^k [3 \curvearrowright 4]^\ell$ and
$\ol{\delta}_{k,\ell} = [3 \curvearrowright 4]^k [1 \curvearrowright 4]^2 [3 \curvearrowright 4]^\ell$ and their extended shadows.
Straightforward computations show that
$\pi^2_{[1 \curvearrowright 2]^3} = \pi^2_{[1 \curvearrowright 2]^2}$, $\pi^2_{\ol{\gamma}_{2,3}} = \pi^2_{\ol{\gamma}_{2,2}}$,
$\pi^2_{[3 \curvearrowright 4]^3} = \pi^2_{[3 \curvearrowright 4]^2}$ and $\pi^2_{\ol{\delta}_{2,3}} = \pi^2_{\ol{\delta}_{2,2}}$.
Hence, and due to Proposition~\ref{pro:automata-state}, immediate inductions on $k$ and $\ell$ prove that
$\pi^2_{[1 \curvearrowright 2]^k} = \pi^2_{[1 \curvearrowright 2]^2}$, $\pi^2_{\ol{\gamma}_{k,\ell}} = \pi^2_{\ol{\gamma}_{2,2}}$,
$\pi^2_{[3 \curvearrowright 4]^k} = \pi^2_{[3 \curvearrowright 4]^2}$ and $\pi^2_{\ol{\delta}_{k,\ell}} = \pi^2_{\ol{\delta}_{2,2}}$
for all $k,\ell \geq 2$.
One checks easily that $\bc_3$ and $\bd_3$ are in relaxation normal form, and it follows that
$\bc_k$ and $\bd_k$ are in relaxation normal form for all integers $k \geq 3$.

Furthermore, observe that $\sigma_1 \sigma_3 = \sigma_3 \sigma_1$, $\Delta_3 \sigma_1 = \sigma_3 \Delta_3$ and $\Delta_3 \sigma_3 = \sigma_1$.
Hence, the braid $(\sigma_1 \sigma_2 \sigma_3)^2 = \Delta_3 \sigma_1^{-1} \sigma_3$ is also equal to $\sigma_3^{-1} \Delta_3 \sigma_3$.
It follows that
\[\gamma_k \Delta_3 = \delta_k \Leftrightarrow \sigma_1^k \sigma_3^k \Delta_3 = \sigma_3^{k+1} (\sigma_1 \sigma_2 \sigma_3)^2 \sigma_3^{k-1} \Leftrightarrow
\sigma_3^k \Delta_3 \sigma_3^k = \sigma_3^{k+1} (\sigma_3^{-1} \Delta_3 \sigma_3) \sigma_3^{k-1}.\]
The last equality is obvious, which completes the proof.
\end{proof}

\begin{cor}
In each braid group $B_n$ with $n \geq 4$ strands,
the relaxation normal form is not asynchronously right automatic.
\label{cor:not-automatic}
\end{cor}

\begin{proof}
Consider some integer $k \geq 3$, as well as the braids $x = \gamma_k$ and $y = \delta_k = x \Delta_3$.
Using the notations of Definition~\ref{dfn:universal-diff-set}, we have $|\bx| = 2k \leq |\by|$ and we also have
$x_\ell = \sigma_1^\ell$ and $y_\ell = \sigma_3^\ell$ for all $\ell \leq k$.
Hence, if $\bX(k) + \bY(k) = k$, then
$y_{\bY(k)}^{-1} x_{\bX(k)}$ is a braid of Artin length $k$.
No finite set can contain Artin braids of arbitrarily Artin length.
Therefore, no finite set is the asynchronous difference set for $\textbf{NF}$, $\varepsilon$, $x$ and $x \Delta$ for all braids $x$,
and Theorem~\ref{thm:asynchronously-automatic} proves that the relaxation normal form is not asynchronously right automatic.
\end{proof}

If checking that a normal form is \emph{not} synchronously or asynchronously automatic can be done by
selecting infinite families of counter-examples, such as in Lemmas~\ref{lem:not-sync-automatic} and~\ref{lem:not-automatic},
proving that a normal form is automatic is computationally less easy or,
at least, witnesses that a normal form is automatic may be large.
However, there exists systematic approaches~\cite{Epstein:1992:WPG:573874}
for building such witnesses, if they exist.

\begin{dfn}{Left asynchronous and synchronous languages}
Let $G$ be a finitely generated group,
and let $\mathbf{NF}$ be a regular normal form on $G$.
Let $\calA = (\Sigma,Q,i,\delta,F)$ be an automaton that recognizes the language $\{\mathbf{NF}(x) : x \in G\}$,
where $\Sigma$ generates positively the group $G$.
In addition, let $\Sigma_{\varepsilon}$ denote the set $\Sigma \cup \{\varepsilon\}$,
where $\varepsilon$ is the neutral element of the group $G$,
and consider the new transition function $\ol{\delta} = \delta \cup \{(q,\varepsilon,q) : q \in Q\}$.

Now, let $T$ be a subset of $G$, and let $g$ be an element of $\mathcal{S}$.
Consider the finite automaton $\calA(g,T) = (\Sigma_{\varepsilon} \times \Sigma_{\varepsilon},Q \times Q \times T,i_a,\delta_a,F_a)$ with
initial state $i_a^g = (i,i,g)$, set of accepting states $F_a = F \times F \times \{\varepsilon\}$ and transition function
\[\delta_a = \{((q,r,x),(\lambda,\mu),(q',r',x')) : x' = \mu^{-1} x \lambda \text{, } q' \in \ol{\delta}(q,\lambda) \text{ and } r' \in \ol{\delta}(r,\mu)\}.\]
We call \emph{left asynchronous automaton} the automaton $\calA(g,T)$,
and \emph{left asynchronous language} the associated language, which we denote by $\mathcal{L}_a(g,T)$.

We also call \emph{left synchronous language} the language $\mathcal{L}_s(g,T)$ that consists of those words
$(\lambda_1,\mu_1) \cdot \ldots \cdot (\lambda_k,\mu_k)$ that belong to the left asynchronous language $\mathcal{L}_a(g,T)$
and such that $\lambda_i = \varepsilon \Rightarrow \lambda_{i+1} = \varepsilon$ and $\mu_i = \varepsilon \Rightarrow \mu_{i+1} = \varepsilon$
for all $i \leq k+1$.
\label{dfn:left-language}
\end{dfn}

Such languages provide us a characterisation of left automatic normal forms,
thanks to the notion of left $\varepsilon$-reduction.

\begin{dfn}{Left $\varepsilon$-reduction}
Let $\Sigma$ be a finite alphabet, and let $\varepsilon$ be an element of $\Sigma$.
We define the \emph{left $\varepsilon$-reduction} of a word $(w_1,x_1) \cdot \ldots \cdot (w_k,x_k)$ with letters in $\Sigma \times \Sigma$
as the word, with letters in $\Sigma$, that we obtain by deleting the letters $\varepsilon$ from the word $w_1 \cdot \ldots \cdot w_k$.
\end{dfn}

Observe that, if $L \subseteq (\Sigma \times \Sigma)^\ast$ is recognized by some finite automaton,
then we can compute an automaton that recognizes the set of left $\varepsilon$-reduction of all words in $L$.
Hence, a consequence of Theorem~\ref{thm:asynchronously-automatic} is the following one.

\begin{pro}
Let $G$ be a finitely generated group, let $\mathbf{NF}$ be a regular normal form on $G$,
that maps each group element to a word with letters in $\Sigma$, where $\Sigma$ is a generating subset of $G$.
For all sets $T \subseteq G$ and all elements $g \in T$,
let $\calL_a(g,T)$ and $\calL_s(g,T)$ be the left asynchronous and synchronous languages defined above.

The normal form $\mathbf{NF}$ is asynchronously left automatic if and only if there exists a finite superset $T$ of $\Sigma$ such that,
for all $g \in \Sigma$, the language $\{\mathbf{NF}(x) : x \in G\}$ is the set of left $\varepsilon$-reduction of all words in $\calL_a(g,T)$.
Similarly, $\mathbf{NF}$ is synchronously left automatic if and only if there exists a finite superset $T$ of $\Sigma$ such that,
for all $g \in \Sigma$, the language $\{\mathbf{NF}(x) : x \in G\}$ is the set of left $\varepsilon$-reduction of all words in $\calL_s(g,T)$.
\label{pro:async-left-automatic-constructive}
\end{pro}

Variants of the languages $\calL_a(g,T)$ and $\calL_s(g,T)$ also exist for characterising the right automaticity.
The only change that must be performed in such variants is to replace the initial state $i_a^g$ be the new initial state
$i_s = (i,i,\varepsilon)$ and the set of accepting states by $F_s^g = F \times F \times \{g^{-1}\}$.
A result analogous to Proposition~\ref{pro:async-left-automatic-constructive} then holds.

Although such characterisations of automatic normal form are \emph{not} suitable for proving that a normal form
is not automatic, they provide effective ways to prove that a normal form is automatic.
Using these ideas, we prove the following result.

\begin{pro}
The relaxation normal form is synchronously automatic when $n = 3$.
Furthermore, the relaxation normal form is also asynchronously left automatic when $n = 4$.
\label{pro:i-a-automatic}
\end{pro}

\begin{proof}
Using Proposition~\ref{pro:async-left-automatic-constructive} and
enumerating sets $T_0 = \Sigma \subseteq T_1 \subseteq T_2 \subseteq \ldots$ such that
$\bigcup_{k \geq 0} T_i = B_n$, we would eventually find witnesses of Proposition~\ref{pro:i-a-automatic}.

However, finding a suitable set $T$ directly is computationally intensive.
A convenient trick lies in considering a variant of the relaxation normal form
which we obtain by replacing every letter $[i \curvearrowright j]$ by the word $\sigma_i \cdot \sigma_{i+1} \cdot \ldots \cdot \sigma_{j-1}$,
and every letter $[i \curvearrowbotright j]$ by the word $\sigma_i^{-1} \cdot \sigma_{i+1}^{-1} \cdot \ldots \cdot \sigma_{j-1}^{-1}$.
Hence, $\Sigma$ is the set of Artin generators of $B_n$.
Then, adequate sets $T$ are of cardinality less than $200$ for $n = 3$, and less than $800$ for $n = 4$ (in the case of the aysnchronous left automaticity).

It follows directly that the relaxation normal form is asynchronous left automatic for $n = 4$,
but not yet that it is synchronously biautomatic for $n = 3$.
Indeed, consider two braids $\alpha$ and $\beta$, with $\beta = g \alpha$ or $\beta = \alpha g$ for some $g \in \{\sigma_1,\sigma_2\}$.
The braids $\alpha$ and $\beta$ are represented by words $\ba$ and $\bb$ in relaxation normal form,
and by words $\ol{\ba}$ and $\ol{\bb}$ in the variant introduced above.

Theorem~\ref{thm:asynchronously-automatic} states that the braids $\ol{b}_\ell^{-1} g \ol{a}_\ell$, obtained using the variant, belong to some finite set.
However, proving that the relaxation normal form is synchronously biautomatic
requires applying Theorem~\ref{thm:asynchronously-automatic} on braids $b_k^{-1} g a_k$ obtained using the original relaxation normal form.
Unfortunately, the braids $b_k^{-1} g a_k$ are not of the form $\ol{b}_\ell^{-1} g \ol{a}_\ell$ in general,
but only of the form $\ol{b}_\ell^{-1} g \ol{a}_m$ for some integers $m$ and $\ell$ that might differ from each other.

We overcome this problem as follows. Having computed the (left and right) asychronous automata $\calA(g,T)$ introduced in Definition~\ref{dfn:left-language},
as well as their synchronous variants, we are able to prove that, in the specific case of the relaxation normal form (for $n = 3$ and
$g \in \{\sigma_1,\sigma_2\}$), we always have $|m-\ell| \leq 8$, independently of the braids $\alpha$ and $\beta = g \alpha$ or $\beta = \alpha g$
that we consider. It is then easy to conclude that the original relaxation normal form is synchronously biautomatic too. 
\end{proof}

We gather all of the above results in Theorem~\ref{thm:who-is-automatic} and in Figure~\ref{fig:automatic-or-not}.

\begin{thm}
For $n = 2$ and $n = 3$, the relaxation normal form is synchronously biautomatic.
For $n = 4$, the relaxation normal form is asynchronously (but not synchronously) left automatic, and is not aysnchronously right automatic.
For $n \geq 5$, the relaxation normal form is not synchronously left automatic, and is not asynchronously right automatic;
it is yet unknown whether it is asynchronously left automatic.
\label{thm:who-is-automatic}
\end{thm}

\begin{figure}[!ht]
\begin{center}
{\def\arraystretch{1}
\begin{tabular}{|c|C{7mm}|C{7mm}|C{7mm}|C{7mm}|c|}
\cline{2-5}
\multicolumn{1}{c|}{} & \multicolumn{4}{|c|}{$n$} & \multicolumn{1}{|c}{} \\
\cline{2-5}
\multicolumn{1}{c|}{} & $2$ & $3$ & $4$ & $\geq 5$ & \multicolumn{1}{|c}{} \\
\hline
\multirow{2}{*}{Asynchronously} & $\checkmark$ & $\checkmark$ & $\checkmark$ & ? & Left \\
\cline{2-6}
& $\checkmark$ & $\checkmark$ & \ding{55} & \ding{55} & Right \\
\hline
\multirow{2}{*}{Synchronously} & $\checkmark$ & $\checkmark$ & \ding{55} & \ding{55} & Left \\
\cline{2-6}
& $\checkmark$ & $\checkmark$ & \ding{55} & \ding{55} & Right \\
\hline
\end{tabular}
}
\end{center}
\caption{Is the relaxation normal form automatic?}
\label{fig:automatic-or-not}
\end{figure}

\section{Relaxation Normal Form and Braid Positivity}
\label{section:positivity}

One of the main features of the braid group is that it is
\emph{left-orderable}, meaning that
there exists a total order $\lhd$ on $B_n$ such that,
if $\alpha$, $\beta$ and $\gamma$ are braids such that $\alpha \lhd \beta$, then
$\gamma \alpha \lhd \gamma \beta$.
This property allows us to characterize the order $\lhd$
just by knowing its positive elements, i.e. the set
$\{\alpha \in B_n: \varepsilon \lhd \alpha\}$.

One such left-order is called the \emph{$\sigma$-order}.
This order has been thoroughly studied~\cite{dehornoy_order_1, Dehornoy1994, dehornoy2008ordering},
and its set of positive elements can be represented simply in terms of
\emph{$\sigma$-positive} braids.

\begin{dfn}{$\sigma_i$-positivity and $\sigma$-positivity}
Let $\beta \in B_n$ be a braid on $n$ strands, and
$\sigma_i \in B_n$ be an Artin generator, where $i \leq n$.
We say that $\beta$ is \emph{$\sigma_i$-neutral} if it belongs to the subgroup of $B_n$
generated by the set $\{\sigma_j : i+1 \leq j \leq n\}$.

We also say that $\beta$ is \emph{$\sigma_i$-positive}
(respectively, \emph{$\sigma_i$-negative})
if it can be expressed as a product
$\beta = \gamma_0 \sigma_i^\epsilon \gamma_1 \sigma_i^\epsilon \ldots \sigma_i^\epsilon \gamma_k$ such that
$k \geq 1$, each braid $\gamma_j$ is $\sigma_i$-neutral, and $\epsilon = 1$ (respectively, $\epsilon = -1$).

Finally, we say that $\beta$ is \emph{$\sigma$-positive}
(respectively, \emph{$\sigma$-negative})
if it is $\sigma_i$-positive (respectively, $\sigma_i$-negative)
for some $i \leq n$.
\end{dfn}

These notions of $\sigma$-positivity and $\sigma$-negativity come with a wealth of properties,
including the fact that every non-trivial braid is either $\sigma$-positive or $\sigma$-negative, but not both
(a proof of this result can be found in~\cite{dehornoy2008ordering}).
It immediately follows that the $\sigma$-order $\lhd $, defined as
$\alpha \lhd \beta$ if and only if $\alpha^{-1} \beta$ is $\sigma$-positive,
has the property of being a total left-order.

Moreover, $\sigma$-positivity and $\sigma$-negativity are directly expressible in terms of tight laminations.

\begin{figure}[!ht]
\begin{center}
\begin{tikzpicture}[scale=0.22]
\draw[draw=gray] (4,0) -- (27,0);

\draw[fill=gray,draw=gray,thick] (9,0) circle (1);

\draw[draw=gray,ultra thick] (11,0) arc (0:360:2);
\draw[draw=gray,ultra thick] (25,0) arc (0:180:9.5);
\draw[draw=gray,ultra thick] (26,0) arc (0:180:10.5);
\draw[draw=gray,ultra thick] (27,0) arc (0:360:11.5);
\draw[draw=gray,ultra thick] (12,0) arc (180:0:4.5);

\draw[draw=gray,ultra thick] (21,0) arc (180:360:2);
\draw[draw=gray,ultra thick] (16,0) arc (360:180:5.5);
\draw[draw=gray,ultra thick] (12,0) arc (360:180:3);

\draw[fill=white,draw=gray,thick] (14,0) circle (1);
\draw[fill=white,draw=gray,thick] (18,0) circle (1);
\draw[fill=white,draw=gray,thick] (23,0) circle (1);

\draw[draw=black,thick] (32,4) -- (18,4) -- (18,2);
\draw[draw=black,ultra thick] (16,0) arc (180:0:2);
\draw[draw=black,thick] (32,-4) -- (23,-4) -- (23,-3);
\draw[draw=black,ultra thick] (20,0) arc (180:360:3);

\draw[draw=black,thick] (3,2.5) -- (14,2.5) -- (14,0);
\draw[draw=black,thick] (3,6) -- (16,6) -- (16,0);
\draw[draw=black,thick] (3,-2.5) -- (20,-2.5) -- (20,0);

\dvertex{14}
\dvertex{16}
\dvertex{20}
\node[anchor=west] at (-1,2.5) {$p$};
\node[anchor=west] at (-1,6) {$p_\calL^+$};
\node[anchor=west] at (-1,-2.5) {$p_\calL^{++}$};
\node[anchor=west] at (32,4) {$\upA_{++}(p)$};
\node[anchor=west] at (32,-4) {$\lowA_{++}(p)$};

\end{tikzpicture}
\end{center}
\caption{A puncture and its second right neighbor and arcs}
\label{fig:arc-trees-second-neighbor}
\end{figure}
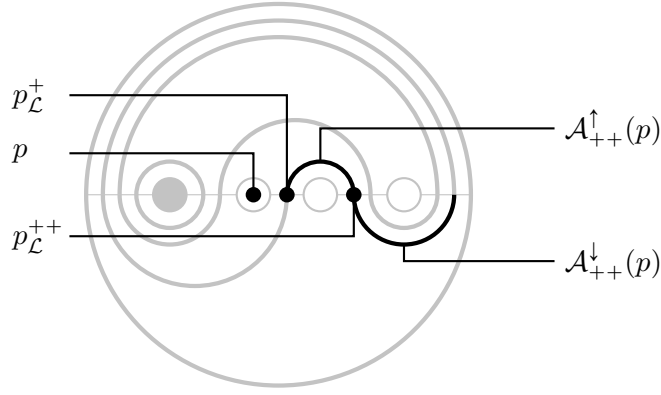

\begin{dfn}{Second right arcs}
Let $\calL$ be a tight lamination and let $p$ be some puncture of $\calL$, except the rightmost one.
Since the set $\calL_\RR = \calL \cap \RR$ intersects both intervals $(p,p_n)$ and $(p_n,\infty)$, the 
point $p_\calL^{++} = \left(p_\calL^+\right)_\calL^+ = \min \{z \in \calL_\RR: z > p_\calL^+\}$ is well-defined.
We call this point the \emph{second right neighbor point} of $p$ in $\calL$.

The point $p_\calL^{++}$ belongs to two arcs of $\calL$.
We call these arcs the \emph{second right upper arc} and the \emph{second right lower arc} of $p$ in $\calL$,
and denote them respectively by $\upA_{++}(p,\calL)$ and by $\lowA_{++}(p,\calL)$.
\label{dfn:second-right-arcs}
\end{dfn}

Figure~\ref{fig:arc-trees-second-neighbor} shows some tight lamination,
in which a puncture $p$, the right neighbor point and the the second right neighbor of $p$, and the second right arcs of $p$ have been highlighted.
Second right arcs provide us with a geometrical characterization of $\sigma_1$-positive and $\sigma_1$-negative braids
(see~\cite{dehornoy2008ordering} for details), which we reformulate here.

\begin{pro}
Let $\calL$ be the tight lamination of a braid $\beta \in B_n$ and let $i$ be some integer such that $1 \leq i \leq n-1$.
The braid $\beta$ is $\sigma_i$-neutral if and only if $0 \in \pi_\beta(j,+,\uparrow) \cap \pi_\beta(j,+,\downarrow)$ for all $j \leq i$.
In addition, $\beta$ is $\sigma_i$-positive if and only if $\beta$ is $\sigma_{i-1}$-neutral and
$0 \notin \pi_\calL(\upA_{++}(p_{i-1},\calL))$.
Similarly, $\beta$ is $\sigma_i$-negative if and only if $\beta$ is $\sigma_{i-1}$-neutral and
$0 \notin \pi_\calL(\lowA_{++}(p_{i-1},\calL))$.
\label{pro:sigma1-positive-words}
\end{pro}

From Proposition~\ref{pro:sigma1-positive-words} follows a characterization of the $\sigma_i$-positive and $\sigma_i$-negative braids
according to their relaxation normal forms.
Indeed, for each integer $j \in \{1,\ldots,n\}$, let $\calS^\uparrow_j$, $\calS^\downarrow_j$ and $\Sigma_j$ be respectively the subsets
$\{[j \curvearrowright v] : j < v\}$, $\{[j \curvearrowbotright v] : j < v\}$ and
$\bigcup_{k \geq j} (\calS^\uparrow_k \cup \calS^\downarrow_k)$
of the set $\Sigma$ of all right-oriented sliding braids.

\begin{thm}
Let $\beta \in B_n$ be a braid. The braid $\beta$ is $\sigma_i$-positive (respectively, $\sigma_i$-negative)
if and only if $\RNF(\beta) \in \Sigma_{i+1}^\ast \cdot \calS^\uparrow_i \cdot \Sigma_i^\ast$
(respectively, $\RNF(\beta) \in \Sigma_{i+1}^\ast \cdot \calS^\downarrow_i \cdot \Sigma_i^\ast$).
\label{thm:right-normal-form-positivity}
\end{thm}

\begin{proof}
The sets $\{\varepsilon\}$, $\Sigma_{i+1}^\ast \cdot \calS^\uparrow_i \cdot \Sigma_i^\ast$
and $\Sigma_{i+1}^\ast \cdot \calS^\downarrow_i \cdot \Sigma_i^\ast$ for ($1 \leq i \leq n-1$)
form a partition of the free monoid $\Sigma^\ast$.
Moreover, a braid $\beta$ is clearly $\sigma_i$-positive if $\RNF(\beta) \in \Sigma_{i+1}^\ast \cdot \calS^\uparrow_i$,
or $\sigma_i$-negative if $\RNF(\beta) \in \Sigma_{i+1}^\ast \cdot \calS^\downarrow_i$.
Hence, and without loss of generality, it suffices to prove that
if $\beta$ is a $\sigma_i$-positive braid and if $\bR(\beta [i \curvearrowbotright j]) = [i \curvearrowbotleft j]$ for some $j > i$, then
$\beta [i \curvearrowbotright j]$ is $\sigma_i$-positive.

Then, let $\calL$ and $\oL$ be tight laminations of $\beta$ and of $\beta [i \curvearrowbotright j]$,
and let $p_0,\ldots,p_n$ and $\op_0,\ldots,\op_n$ be their respective punctures.
Since $\beta$ is $\sigma_i$-neutral, so is $\beta [i \curvearrowright j]$.
Moreover, all upper arcs of $\calL$ remain lower arcs of $\oL$.
Hence, consider the arc $A = \upA_{++}(p_{i-1},\calL) = \upA_{++}(\op_{i-1},\oL)$.
Proposition~\ref{pro:sigma1-positive-words} states that $A$ does not cover the fixed puncture $p_0 = \op_0$,
hence it states that $\beta [i \curvearrowbotright j]$ is $\sigma$-positive.
\end{proof}

If follows that the sets $\{\RNF(\beta) : \beta$ is $\sigma_i$-positive$\}$ and
$\{\RNF(\beta) : \beta$ is $\sigma_i$-negative$\}$ are regular, and that
each prefix of a $\sigma_i$-positive word must be $\sigma_i$-positive or $\sigma_i$-neutral.
Hence, an immediate consequence of Theorem~\ref{thm:right-normal-form-positivity} is the following.

\begin{thm}
There exists a regular language $\bL$ such that, for each braid $\beta \in B_n$,
the braid $\beta$ is $\sigma$-positive if and only if its relaxation normal form
$\RNF(\beta)$ belongs to $\bL$.
\end{thm}

\section{Remembering Extended Shadows is Not Overkill}
\label{section:automaton-size}

Theorem~\ref{thm:here-is-the-automaton} provides us with a deterministic automaton that
accepts the relaxation normal form.
Is this automatonc minimal or close to minimal?

A first step in answering this question is to find an upper bound on the number of
possible extended shadows of all braids $\beta \in B_n$.
Such a crude upper bound is obtained as follows.
There exists $4n$ triples $(i,\diamond,\vartheta) \in \{1,\ldots,n\} \times \{+,-\} \times \{\downarrow,\uparrow\}$,
and each such triple is mapped to two (possibly equal) subintervals of $\{0,\ldots,n\}$.
Hence, there exists at most $(n+1)^{4 \times 4n}$ extended shadows.
However, we can prove that there exists $2^{\mathcal{O}(n)}$ only extended shadows,
using the notion of neighbor trees, which is illustrated in Figure~\ref{fig:neighbo-trees}.

\begin{dfn}{Neighbor trees}
Let $\calL$ be tight lamination, with punctures $p_0, \ldots, p_n$,
and let $p_k$ be the rightmost puncture that is covered by some bigon of $\calL$.
Let $\bA_n^\top$ and $\bA_n^\bot$ be the two arcs contained in the curve $\calL_n$,
i.e. the two arcs that cover all other upper arcs and punctures of $\calL$.
We say that an arc $\calA$ of $\calL$ is \emph{nice} if $\calA$ is either one neighbor arc of a puncture $p_i$, or is $\bA_n^\top$ or $\bA_n^\bot$,
or covers the puncture $p_k$ and shares its right endpoint with a neighbor arc of a puncture.

The \emph{upper neighbor tree} of the lamination $\calL$
which we denote by $\calN^\uparrow(\calL)$, and defined as follows.
The vertices of $\calN^\uparrow(\calL)$ are of the form
$v_A$, where $A$ is a nice upper arc of $\calL$,
or of the form $v_p$, where $p$ is a puncture of $\calL$.
A vertex $v_A$ is an ancestor of $v_B$ in $\calN^\uparrow(\calL)$ if and only if $A$ covers $B$.
Hence, the children of every vertex can be ordered from left to right.
If $v_{A_1},\ldots,v_{A_k}$ are the children of a vertex $v_A$ in if $v_A$ is a vertex whose children in $\calN^\uparrow(\calL)$, taken from left to right,
then we say that $v_{A_i}$ is the $i$-th child of $v_A$.

We define similarly the \emph{lower neighbor tree} of $\calL$, which we denote by $\calN^\downarrow(\calL)$.
\label{dfn:neighbor-trees}
\end{dfn}

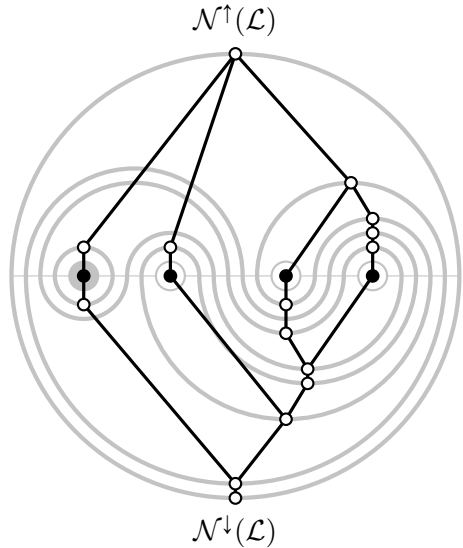
\begin{figure}[!ht]
\begin{center}
\begin{tikzpicture}[scale=0.19]
\draw[draw=gray] (36,0) -- (67,0);

\draw[fill=gray,draw=gray,thick] (41,0) circle (1);

\draw[draw=gray,ultra thick] (36,0) arc (180:-180:15.5);
\draw[draw=gray,ultra thick] (37,0) arc (180:0:7.5) arc (180:360:3) arc (180:0:3) arc (360:180:7.5) arc (0:180:2) arc (180:360:10)
arc (0:180:4) arc (360:180:2) arc (180:0:6.5) arc (360:180:14.5);
\draw[draw=gray,ultra thick] (38,0) arc (180:0:6.5) arc (180:360:4) arc (180:0:2) arc (360:180:6.5) arc (0:180:3) arc (360:180:3);
\draw[draw=gray,ultra thick] (39,0) arc (180:-180:2);

\draw[fill=white,draw=gray,thick] (47,0) circle (1);
\draw[fill=white,draw=gray,thick] (55,0) circle (1);
\draw[fill=white,draw=gray,thick] (61,0) circle (1);

\draw[black,very thick] (51.5,15.5) -- (41,2) -- (41,-2) -- (51.5,-14.5) -- (51.5,-15.5);
\draw[black,very thick] (51.5,15.5) -- (47,2) -- (47,0) -- (55,-10) -- (51.5,-14.5);
\draw[black,very thick] (51.5,15.5) -- (59.5,6.5) -- (55,0) -- (55,-4) -- (56.5,-6.5) -- (56.5,-7.5) -- (55,-10);
\draw[black,very thick] (59.5,6.5) -- (61,4) -- (61,0) -- (56.5,-6.5);

\bvertex{41}{-2}
\bvertex{41}{2}
\bvertex{47}{2}
\bvertex{51.5}{-15.5}
\bvertex{51.5}{-14.5}
\bvertex{51.5}{15.5}
\bvertex{55}{-10}
\bvertex{55}{-4}
\bvertex{55}{-2}
\bvertex{56.5}{-7.5}
\bvertex{56.5}{-6.5}
\bvertex{59.5}{6.5}
\bvertex{61}{2}
\bvertex{61}{3}
\bvertex{61}{4}

\dvertex{41}
\dvertex{47}
\dvertex{55}
\dvertex{61}

\node[anchor=south] at (51.5,16) {$\calN^\uparrow(\calL)$};
\node[anchor=north] at (51.5,-16) {$\calN^\downarrow(\calL)$};
\end{tikzpicture}
\end{center}
\caption{Neighbor trees of a tight lamination $\calL$}
\label{fig:neighbo-trees}
\end{figure}

Following Corollary~\ref{cor:all-cover},
the leaves of $\calN^\uparrow(\calL)$ and $\calN^\downarrow(\calL)$ are the punctures $p_i$.
Furthermore, a puncture $p_i$ belongs to an upper (respectively, lower) bigon if and only if it has no sibling
in $\calN^\uparrow(\calL)$ (respectively, in $\calN^\downarrow(\calL)$).

\begin{lemma}
Let $\beta$ and $\beta'$ be braids, with respective tight laminations $\calL$ and $\calL'$.
If $\calN^\uparrow(\calL) = \calN^\uparrow(\calL')$ and
$\calN^\downarrow(\calL) = \calN^\downarrow(\calL')$, then
$\pi^2_\beta = \pi^2_{\beta'}$.
\label{lem:neighbor-trees-TR}
\end{lemma}

\begin{proof}
Let $v_{\bA_n^\top}$ and $v_{\bA_n^\bot}$ be the respective roots of $\calN^\uparrow(\calL)$ and of $\calN^\downarrow(\calL)$.
Let $\Lambda^\uparrow$ be the set of nice upper arcs of $\calL$,
and let $\Lambda^\downarrow$ be the set of nice lower arcs of $\calL$.

In addition, let $p$ be some puncture of $\calL$.
The arc $v_{\upA_-(p,\calL)}$ is either the left sibling of $p$ in $\calN^\uparrow(\calL)$, if such a left sibling exists,
or the parent of $p$ in $\calN^\uparrow(\calL)$.
We identify similarly the vertices $v_{\lowA_-(p,\calL)}$ and $v_{\calA_+^\updownarrow(p,\calL)}$
among the nodes of $\calN^\uparrow(\calL)$ and $\calN^\downarrow(\calL)$.

Moreover, let $k$ be the rightmost index of $\calL$.
We identify $k$ since $p_k$ is the rightmost puncture that does not have siblings in both $\calN^\uparrow(\calL)$ and $\calN^\downarrow(\calL)$.
Let $A_1, \ldots, A_u$ be the nice upper arcs that cover $p_k$, such that $A_x$ is covered by $A_y$ if and only if $x < y$.
Similarly, let $B_1,\ldots,B_v$ be the nice lower arcs that cover $p_k$, such that $B_x$ is covered by $B_y$ if and only if $x < y$.
It comes immediately that $u = v$ and that the arcs $A_j$ and $B_j$ share their right endpoints for all $j \in \{1,\ldots,u\}$.
Hence, we identify each of the nice arcs that cover the puncture $p_k$.

Consequently, we can compute $\pi^2_\beta(i,\diamond,\vartheta)$ for
each triple $(i,\diamond,\vartheta) \in \{0,\ldots,n\} \times \{+,-\} \times \{\uparrow,\downarrow\}$, which means that
the trees $\calN^\uparrow(\calL)$ and $\calN^\downarrow(\calL)$ uniquely determine $\pi^2_\beta$. This completes the proof.
\end{proof}

\begin{cor}
Let $\calA = (\Sigma, Q, i, \delta, Q)$ be the automaton
provided in Theorem~\ref{thm:here-is-the-automaton}.
Its state set $Q$ is of size $\#Q \leq 2^{20(n+1)}$.
\label{cor:automaton-small-size}
\end{cor}

\begin{proof}
Let $\bN$ be the set $\{(\calN^\uparrow(\calL),\calN^\downarrow(\calL)) : \calL$ is a tight lamination$\}$ and let
$\Pi$ be the set $\{\pi^2_\beta : \beta \in B_n\}$.
Lemma~\ref{lem:neighbor-trees-TR} states that there exists some surjective projection $\bN \mapsto \Pi$, hence that
$\#\Pi \leq \#\bN$.
Since $Q = \Pi$, it remains to show that $\#\bN \leq 2^{20(n+1)}$.

Let $\calL$ be some tight lamination.
The tree $\calN^\uparrow(\calL)$ contains at most $n+1$ nodes of the type $v_{p_i}$,
$2(n+1)$ nodes of the type $v_{\calA_\pm^\uparrow(p_i)}$, $1$ node of the type $\bA_n^\top$ and
$2(n+1)$ nodes of the type $v_A$, where $A$ is a nice upper arc that covers $p_k$.
This proves that $\calN^\uparrow(\calL)$ has at most $5n+6$ nodes.
Similarly, $\calN^\downarrow(\calL)$ has at most $5n+6$ nodes.

Moreover, both $\calN^\uparrow(\calL)$ and $\calN^\downarrow(\calL)$ are rooted ordered trees.
For each integer $k$, there exists $C_{k-1}$ rooted ordered trees with $k$ nodes, where $C_k = \frac{1}{k+1}\binom{2k}{k}$
is the $k$-th Catalan number (see~\cite[p.~35]{Flajolet:2009:AC:1506267}).
Hence, the relations
\[\sum_{i=0}^k C_i \leq (k+1) C_k = \binom{2k}{k} = \prod_{i=1}^k \frac{2i}{i} \cdot \frac{2i-1}{i} \leq 2^{2k}\]
show that there exist at most $2^{10(n+1)}$ rooted ordered trees with at most $5n+6$ nodes.
It follows that $\#Q = \#\Pi \leq \#\bN \leq 2^{20(n+1)}$.
\end{proof}

We can then prove that the size of the automaton $\calA$ has the same order of magnitude
as the size of the minimal automaton.

\begin{pro}
Let $\calA_{\min} = (\Sigma, Q_{\min}, i_{\min}, \delta_{\min}, F_{\min})$ be the minimal deterministic automaton
that accepts the set of relaxation normal words for the braid group $B_n$.
The sets $F_{\min}$ and $Q_{\min}$ are equal,
with cardinality $\#Q_{\min} \geq 2^{n/2-1}$.
\label{pro:automaton-big-size}
\end{pro}

\begin{proof}
Since $\calA_{\min}$ is minimal,
each of its states is co-accessible:
from each state $s \in Q_{\min}$, one can reach a state $s' \in F_{\min}$.
Since the relaxation normal form is prefix-closed, it follows that $Q_{\min} \subseteq F_{\min}$, i.e. that
$F_{\min} = Q_{\min}$.

To each braid $\alpha$ corresponds a unique relaxation normal word in $\Sigma^\ast$,
hence one unique state in $Q_{\min}$. We denote this state by $\delta^\ast(\alpha)$.
Now, let $m = \lfloor \frac{n-1}{2} \rfloor$.
To each tuple $\epsilon = (\epsilon_1,\epsilon_2,\dots,\epsilon_m) \in \{-1,1\}^m$,
we associate the braid
$\beta_\epsilon = \sigma_1^{\epsilon_1} \sigma_3^{\epsilon_2} \dots \sigma_{2m-1}^{\epsilon_m} \in B_n$.
An immediate induction on $m$ shows that
$\sigma_1^{\epsilon_1} \cdot \sigma_3^{\epsilon_2} \cdot \dots \cdot \sigma_{2m-1}^{\epsilon_m}$
is a relaxation normal word.

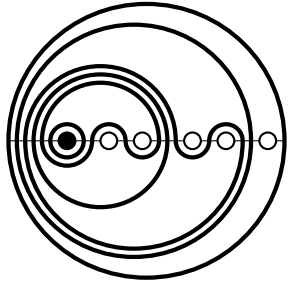
\begin{figure}[!ht]
\begin{center}
\begin{tikzpicture}[scale=0.11]
\draw[draw=black] (3,0) -- (36,0);

\draw[fill=black,draw=black,thick] (10,0) circle (1);

\draw[draw=black,ultra thick] (3,0) arc (180:-180:16.5);
\draw[draw=black,ultra thick] (4,0) arc (180:-180:14);
\draw[draw=black,ultra thick] (5,0) arc (-180:0:13);
\draw[draw=black,ultra thick] (5,0) arc (180:0:9);
\draw[draw=black,ultra thick] (6,0) arc (180:-180:8);
\draw[draw=black,ultra thick] (7,0) arc (180:0:7);
\draw[draw=black,ultra thick] (7,0) arc (-180:0:3);
\draw[draw=black,ultra thick] (8,0) arc (180:-180:2);

\draw[draw=black,ultra thick] (17,0) arc (0:180:2);
\draw[draw=black,ultra thick] (17,0) arc (180:360:2);
\draw[draw=black,ultra thick] (27,0) arc (180:0:2);
\draw[draw=black,ultra thick] (27,0) arc (360:180:2);

\draw[fill=white,draw=black,thick] (15,0) circle (1);
\draw[fill=white,draw=black,thick] (19,0) circle (1);
\draw[fill=white,draw=black,thick] (25,0) circle (1);
\draw[fill=white,draw=black,thick] (29,0) circle (1);
\draw[fill=white,draw=black,thick] (34,0) circle (1);;
\end{tikzpicture}
\end{center}
\caption{The braid $\beta_{(1,-1)}$ (for $n = 5$)}
\label{fig:beta}
\end{figure}

Then, if $\epsilon$ and $\eta$ are different tuples in $\{-1,1\}^m$,
consider some integer $i \leq m$ such that $\epsilon_i \neq \eta_i$.
Without loss of generality, we assume that $\epsilon_i = 1$ and that $\eta_i = -1$.
One shows easily that $\RNF(\beta_\epsilon) \cdot [2i \curvearrowbotright n]$ is a relaxation normal word, although
$\RNF(\beta_\eta) \cdot [2i \curvearrowbotright n]$ is not.
This shows that $\delta^\ast(\beta_\epsilon) \neq \delta^\ast(\beta_\eta)$ and, consequently, that
$\#Q_{\min} \geq 2^m \geq 2^{n/2-1}$.
\end{proof}

For example, Figure~\ref{fig:beta} shows the $5$-strand braid $\beta_{(1,-1)} = [1 \curvearrowright 2] [3 \curvearrowbotright 4]$:
we have $\RNF(\beta_{(1,-1)}) = [1 \curvearrowright 2] \cdot [3 \curvearrowbotright 4]$,
$\RNF(\beta_{(1,-1)} [2 \curvearrowright 5]) = [1 \curvearrowright 5] \cdot [3 \curvearrowbotright 4]$ and
$\RNF(\beta_{(1,-1)} [4 \curvearrowright 5]) = [1 \curvearrowright 2] \cdot [3 \curvearrowbotright 4] \cdot [4 \curvearrowright 5]$.

Corollary~\ref{cor:automaton-small-size} and Proposition~\ref{pro:automaton-big-size} prove
that the automaton constructed in Theorem~\ref{thm:here-is-the-automaton}
is of minimal size, up to an exponent independent of $n$.

\begin{thm}
Both the automaton $\calA$ of Theorem~\ref{thm:here-is-the-automaton}
and the minimal automaton $\calA_{\min}$ of Proposition~\ref{pro:automaton-big-size}
have state sets with cardinalities $2^{\Omega(n)}$.
\label{thm:final-result}
\end{thm}

In particular, Theorem~\ref{thm:final-result} can be interpreted
from an algorithmic point of view.
A streaming algorithm for checking the membership in the relaxation normal that would
rely on remembering extended shadows would require
the storage of up to $20(n+1)$ memory bits. Moreover, any streaming algorithm
that would perform this membership checking would require storing at least $n/2-1$ bits.
The space consumption of first algorithm is therefore optimal up to a multiplicative constant.

\section{Investigating Variants of the Relaxation Normal Form}
\label{section:other-NF}

In the above sections, we investigated properties of the relaxation normal form obtained by follwing Definition~\ref{dfn:RNF}.
When relaxing the tight lamination of a braid,
which picked the rightmost puncture $p_k$ covered by a bigon,
then slid $p_k$ along its right neighbor arc or, if this first choice turned out to be impossible,
we slid $p_k$ along its left neighbor arc.
We review here some variants of the relaxation normal form, obtained by selecting another puncture to be slid,
or another arc along which to slid the puncture:
\begin{enumerate}
 \item the \emph{simple right normal form} is obtained by selecting the rightmost puncture $p_k$ covered by a bigon,
 then systematically sliding $p_k$ along its left neighbor arc (which, due to Lemma~\ref{lem:left-endpoint-is-left},
 is always possible);
 \item the \emph{left normal form} is obtained by selecting the leftmost puncture $p_\ell$ covered by a bigon,
 then sliding $p_\ell$ along its left neighbor arc, or along its right neighbor arc if needed;
 \item the \emph{simple left normal form} is obtained by selecting the leftmost puncture $p_\ell$ covered by a bigon,
 then sliding $p_\ell$ along its right neighbor arc (which also is always possible);
 \item the \emph{outermost normal form} is obtained as follows:
 for each puncture $p_m$ covered by a bigon, consider the integer $i$ such that
 $p_m$ lies in the area enclosed between the curves $\calL_i$ and $\calL_{i+1}$;
 we select the puncture $p_m$ whose associated integer $i$ is maximal,
 then we slide $p_m$ along its right neighbor arc, or along its left neighbor arc if needed;
 \item the \emph{right covered normal form} is obtained as follows:
 for each puncture $p_m$ covered by a bigon, let $i$ be the number of arcs that cover $p_m$ and no other puncture;
 we select the puncture $p_m$ whose associated integer $i$ is maximal
 (in case of equality, we select the rightmost such puncture),
 then we slide $p_m$ along its right neighbor arc, or along its left neighbor arc if needed.
\end{enumerate}

First, the analysis of the relaxation normal form performed above
would work similarly with the simple right normal form.
More precisely, Proposition~\ref{pro:caracterisation-extension}
also holds for the simple right normal form,
provided that its original requirement $(2)$ be replaced by the new requirement $(2')$:
$\pi_\beta(k,+,\uparrow) \subseteq \{0,\ldots,k\}$.
Then, Proposition~\ref{pro:automata-state}, Theorem~\ref{thm:here-is-the-automaton}, Theorem~\ref{thm:master-thm}
and subsequent properties investigated in sections~\ref{section:automatic},~\ref{section:positivity} and~\ref{section:automaton-size}
also hold for the simple right normal form.

Second, let $\phi_\Delta$ denote the conjugation by the Garside element $\Delta_n = (\sigma_1) (\sigma_2 \sigma_1) \ldots$ $(\sigma_{n-1} \ldots \sigma_2 \sigma_1)$.
The involutive group automorphism $\phi_\Delta$ maps each Artin generator $\sigma_i$ to the generator $\sigma_{n+1-i}$, and vice-versa.
Hence, consider a braid $\beta \in B_n$, and let $w_1 \cdot \ldots \cdot w_k$ be the relaxation normal form of the braid $\phi_\Delta(\beta)$:
the left normal form of $\beta$ is the word $\phi_\Delta(w_1) \cdot \ldots \cdot \phi_\Delta(w_k)$.
Consequently, we say that the relaxation normal form and the left normal form are conjugate to each other via the conjugation automorphism $\phi_\Delta$.
Similarly, the simple right normal form and the simple left normal form are conjugate to each other via the conjugation automorphism $\phi_\Delta$.
Hence, the results of sections~\ref{section:braid-laminations} to~\ref{section:automaton-size} also hold in that case.

Third, the outermost normal form consists in braiding the second strand of the braid with the first one,
then the third strand with the first two strands, and so on.
Hence, it is easy to check that this normal form is regular,
even without using sophisticated tools such as extended shadows.

Finally, the right covered normal form is likely to provide us with short words,
like the normal form studied by Dynnikov and Wiest~\cite{dynnikov:hal-00001267} while investigating the geometric complexity of braids.
Indeed, if we relax a lamination $\calL$ by sliding a puncture $p_m$ whose associated integer is $i$,
then the lamination obtained after relaxing $\calL$ is of complexity at most $\|\calL\| - 2i$.
Hence, heuristically, choosing punctures $p_m$ with large associated integers $i$
should be a wise choice.
However, for $n \geq 6$, the word
$[2 \curvearrowbotright 3]^k \cdot [5 \curvearrowbotright 6]^\ell \cdot [1 \curvearrowbotright 2] \cdot [4 \curvearrowbotright 5]$
is in normal form if and only if $k \leq \ell$,
which proves that the right covered normal form is not regular.

Overall, there exists a wide class of algorithms, based on relaxing tight laminations, and
whose associated normal forms are regular.
However, if such algorithms require counting arbitrarily many arcs of the lamination,
then they may unsurprisingly give rise to non-regular normal forms.

An interesting challenge would be to identify natural algorithms whose
associated normal form would be regular, but also synchronously automatic (or, at least
asynchronously automatic) for all $n \geq 2$,
and in which a product $x_1 x_2 \ldots x_k$ of Artin generators
would have a ``short'' normal form, i.e. a normal form of length at most $K_n k$, for some constant $K_n$.
In particular, although the relaxation normal form was proven not to be asynchronously automatic for $n \geq 4$,
it remains possible that it might produce short words.

\section*{Acknowledgments}

The author is very thankful to an anonymous referee for his (her) insightful remarks and suggestions,
which led to improving the overall readability of the article and to adding section~\ref{section:automatic},
and to Bertold Wiest, whose help was crucial for finding most results of that section.

\pagebreak

\end{document}